\documentclass[11pt]{amsart}

\usepackage{amsfonts, amssymb, amsmath, amsthm, amscd}

\usepackage[dvips]{epsfig}
\usepackage{fancyhdr}
\usepackage{afterpage}
\usepackage{chngcntr}
\usepackage{graphicx}
\usepackage{latexsym}
\usepackage{bm}
\usepackage{psfrag}
\usepackage{url}
\usepackage{lscape}
\usepackage[usenames]{color}
\usepackage[all]{xy}
\usepackage{enumitem}
\usepackage{pinlabel}
\usepackage[olditem, oldenum]{paralist}

\usepackage[unicode, bookmarksnumbered=true, backref=true, dvips]{hyperref}

\newtheorem{theorem}{Theorem}[section]
\newtheorem{lemma}[theorem]{Lemma}

\newtheorem{conjecture}[theorem]{Conjecture}
\newtheorem{corollary}[theorem]{Corollary}
\newtheorem{proposition}[theorem]{Proposition}

\theoremstyle{definition}
\newtheorem{definition}[theorem]{Definition}

\newtheorem{remark}[theorem]{Remark}

\newtheorem{example}[theorem]{Example}

\numberwithin{equation}{section}

\def\Z{\mathbb{Z}}

\def\F{\mathbb{F}}

\def\CFK {\widehat{\operatorname{CFK}}}

\def\HFK {\widehat{\operatorname{HFK}}}

\def\HF {\widehat{\operatorname{HF}}}

\def\HFKm {\operatorname{HFK}^-}
\def\HS {\operatorname{HS}}

\def\Kh{\widetilde{\operatorname{Kh}}}

\def\o{\mathfrak{o}}

\def\CFKtil {\widetilde{\operatorname{CFK}}}

\def\CFKtt {\underline{\widetilde{\operatorname{CFK}}}}
\def\HFKtil {\widetilde{\operatorname{HFK}}}
\def\HFKtt {\underline{\widetilde{\operatorname{HFK}}}}

\def\KK {\mathcal{K}}
\def\LL {\mathcal{L}}

\def\O{\mathbb{O}}
\def\P{\mathbb{P}}
\def\X{\mathbb{X}}
\def\A{\mathbb{A}}

\def\AA {\mathcal{A}}

\def\DD {\mathcal{D}}
\def\FF {\mathcal{F}}
\def\HH {\mathcal{H}}

\def\MM {\mathcal{M}}

\def\RR {\mathcal{R}}
\def\SS {\mathcal{S}}

\def\XX {\mathcal{X}}
\def\YY {\mathcal{Y}}

\def\T {\mathbb{T}}
\def\s {\mathfrak{s}}

\def\p{\mathbf{p}}
\def\r{\mathbf{r}}

\def\x{\mathbf{x}}
\def\y{\mathbf{y}}
\def\z{\mathbf{z}}
\def\w{\mathbf{w}}

\newcommand{\abs}[1] {\left\lvert #1 \right\rvert}

\newcommand{\gen}[1] {\langle #1 \rangle}
\def\Th{^{\text{th}}}
\def\minus{\setminus} \def\co{\colon\thinspace}

\DeclareMathOperator{\Sym}{Sym} \DeclareMathOperator{\Spin}{Spin}

\definecolor{purple}{RGB}{128,0,128}
\definecolor{darkgreen}{RGB}{0,128,0}

\title[A combinatorial spanning tree model for $\HFK$]{A combinatorial spanning tree model for knot Floer homology}
\author[John A.\ Baldwin]{John A.\ Baldwin}
\address{Department of Mathematics, Boston College, Chestnut Hill, MA 02467}
\email{john.baldwin@bc.edu}

\author[Adam Simon Levine]{Adam Simon Levine}
\address{Department of Mathematics, Princeton University, Princeton, NJ 08544} \email{asl2@math.princeton.edu}
\thanks{Both authors were partially supported by NSF Postdoctoral Fellowships. This article appeared in \emph{Advances in Mathematics} \textbf{231} (2012) 1886--1939. The references have been updated from the published version.}

\newcommand\cfkt{\CFKtil}
\newcommand\HFKt{\HFKtil}
\newcommand\Ta{\mathbb{T}_{\alpha}}
\newcommand\Tb{\mathbb{T}_{\beta}}

\def\OO{\mathbb{O}}
\def\XX{\mathbb{X}} \DeclareMathOperator\rk{rk} 
\def\ts{\textsection}
\def\ul{\underline}
\def\ol{\overline}

\begin{document}
\begin{abstract}
We iterate Manolescu's unoriented skein exact triangle in knot Floer homology with coefficients in the field of rational functions over $\Z/2\Z$. The result is a spectral sequence which converges to a stabilized version of $\delta$-graded knot Floer homology. The $(E_2,d_2)$ page of this spectral sequence is an algorithmically computable chain complex expressed in terms of spanning trees, and we show that there are no higher differentials. This gives the first combinatorial spanning tree model for knot Floer homology.
\end{abstract}

\maketitle
 \setlength{\parskip}{5pt}

\section{Introduction}
\label{sec:introduction}
Knot Floer homology is an invariant of oriented links in the 3-sphere,
originally defined by Ozsv{\'a}th-Szab{\'o} \cite{OSzKnot} and by Rasmussen
\cite{RasmussenThesis} using Heegaard diagrams and holomorphic disks. This
invariant comes in several flavors. The simplest is a bigraded vector space
over $\F=\mathbb{Z}/2\mathbb{Z}$,
\[
\HFK(L) = \bigoplus_{m,a}\HFK_m(L,a),
\]
from which one can recover the Seifert genus of $L$ \cite{OSzGenus} and
determine whether $L$ is fibered \cite{GhigginiFibred, NiFibered}. In addition,
knot Floer homology categorifies the Alexander polynomial:
\begin{equation}\label{eq:geuler}
\sum_{m,a}(-1)^{m+(|L|-1)/2}\,\text{rk}\,\HFK_m(L,a)\cdot t^a =
(t^{-1/2}-t^{1/2})^{|L|-1} \cdot \Delta_L(t),
\end{equation}
where $|L|$ is the number of components of $L$.

In 2006, Manolescu--Ozsv\'ath--Sarkar \cite{ManolescuOzsvathSarkar} and Sarkar--Wang \cite{SarkarWang} discovered algorithms for computing knot Floer homology via Heegaard diagrams in which the counts of holomorphic disks are completely combinatorial. The following year, Ozsv\'ath and Szab\'o \cite{OSzCube} gave an algebro-combinatorial formulation of knot Floer homology using a singular cube of resolutions construction which takes as input a
marked braid-form projection of a knot. The purpose of this article is to give an entirely novel combinatorial description of the \emph{$\delta$-graded} knot Floer homology groups,
\begin{equation}\label{eq:deltagradedhfk}
\HFK_{\delta}(L) = \bigoplus_{a-m=\delta}\HFK_m(L,a),
\end{equation}
in terms of \emph{spanning trees}. Before launching into this description, we provide some background and motivation.

Let $\DD$ be a connected planar projection of $L$, and color its complementary
regions black and white in a checkerboard fashion, so that the unbounded region
of $\mathbb{R}^2\minus\DD$ is colored white. One forms the \emph{black graph} $B(\DD)$
by placing a vertex in each black region and connecting two vertices by an edge
for every crossing of $\DD$ that joins the corresponding regions. 
A \emph{spanning tree} of $B(\DD)$ is a connected, acyclic subgraph of $B(\DD)$
that contains all vertices of $B(\DD)$. The Alexander and Jones polynomials of
$L$ can be expressed as sums of monomials associated to such trees. When $L$ is
a knot, for example,
\begin{equation} \label{eq:kauffman}
\Delta_L(t) = \sum_{s \in\mathcal{T}(B(\DD))} (-1)^{M(s)}\cdot t^{A(s)},
\end{equation}
where $\mathcal{T}(B(\DD))$ is the set of spanning trees of $B(\DD)$, and
$A(s)$ and $M(s)$ are integers \cite{KauffmanFormal}.

Since knot Floer homology encodes the Alexander polynomial, one expects that it should also admit a formulation in terms of spanning trees. Indeed, in \cite{OSzAlternating}, Ozsv{\'a}th and Szab{\'o} associate to $\DD$ a doubly-pointed Heegaard diagram $(\Sigma, \bm\alpha, \bm\beta, z,w)$ for $L$ for which generators of the chain complex $\CFK(\Sigma, \bm\alpha, \bm\beta,z,w)$ are in 1-to-1 correspondence with spanning trees of $B(\DD)$, with the bigrading given by the quantities $A(s)$ and $M(s)$ in \eqref{eq:kauffman}. Using this Heegaard diagram, they prove that the knot
Floer homology of an alternating knot is determined by its Alexander polynomial and signature. However, despite numerous efforts, no one has managed to find a combinatorial description of the differential on this complex, largely because
there is no general algorithm for counting the relevant holomorphic disks.


In this article, we introduce a complex for knot Floer homology whose differential is combinatorial and can be described explicitly in terms of spanning trees. Our construction starts with an oriented, connected planar projection $\DD$ for $L$. We choose $m$ marked points on the edges of $\DD$ so that every edge contains at least one such point. Let $\FF = \F(T)$, the field of rational functions in a single variable $T$ with coefficients in $\F$. In Section \ref{sec:complex}, we define a graded chain complex $(C^\Omega(\DD), \partial^{\Omega}),$ where $C^\Omega(\DD)$ is a direct sum of $2^{m-1}$-dimensional vector spaces over $\FF$, one for each spanning tree of $B(\DD)$, and $\partial^{\Omega}$ can be described explicitly in terms of the planar embedding of $B(\DD)$, the marked points, and a generic function $\Omega$ from the crossings of $\DD$ to the integers.

Our main theorem is the following.

\begin{theorem} \label{thm:main}
The homology of $(C^\Omega(\DD),\partial^{\Omega})$ is isomorphic as a graded $\FF$-vector space to $\HFK(L) \otimes_{\mathbb{F}} V^{\otimes(m-|L|)}\otimes_{\mathbb{F}} \FF$ with respect to the $\delta$-grading on $\HFK(L)$, where $V$ is a two-dimensional vector space over $\mathbb{F}$ supported in grading zero.
\end{theorem}

Our construction makes use of Manolescu's unoriented skein exact triangle \cite{ManolescuSkein}, which relates the knot Floer homology of $L$ with those of its two resolutions at a crossing. Under mild technical assumptions, one can iterate Manolescu's triangle in the manner of Ozsv{\'a}th-Szab{\'o} \cite{OSzDouble}. The result is a cube of resolutions spectral sequence $\SS_\F$ that converges to $\HFK(L)\otimes V^{\otimes(m-|L|)}$ and whose $E_1$ page is a direct sum,
\[
\bigoplus_{I\in \{0,1\}^n} \HFK(L_I)\otimes V^{\otimes(m-|L_I|)},
\]
over complete resolutions $L_I$ of $\DD$. The $d_1$ differential of $\SS_{\F}$ can be described explicitly. Unfortunately, however, $E_2(\SS_{\F})$ is not an invariant of $L$ (see Remark \ref{rmk:untwistedd1}).

To skirt this issue, we perform the above iteration instead over $\FF$, using a system of twisted coefficients determined by $\Omega$. With these coefficients, the knot Floer homologies of disconnected resolutions vanish, and the $E_1$ page of the resulting spectral sequence, $\SS^{\Omega}_{\FF}$, is a direct sum of vector spaces associated to \emph{connected} resolutions, which are themselves in 1-to-1 correspondence with spanning trees of $B(\DD)$. This page is isomorphic to the complex $C^\Omega(\DD)$, and $d_1(\SS^{\Omega}_{\FF})$ is identically zero since no edge in the cube of resolutions of $\DD$ can join two connected resolutions. We identify the differential $d_2(\SS^{\Omega}_{\FF})$ with $\partial^{\Omega}$ and, based on a grading argument, show that $\SS^{\Omega}_\FF$ collapses at its $E_3$ page. This proves Theorem \ref{thm:main}.

For the remainder of this section, we shall denote the homology $H_*(C^\Omega(\DD),\partial^{\Omega})$ by $\HS_*(L,m)$.

Although the $\delta$-grading on knot Floer homology contains less information than the bigraded theory (e.g. one generally needs the bigrading to determine Seifert genus), it is still a rather powerful invariant with several
applications. Below, we briefly recast some of these in terms of $\HS(L,m)$. Recall that the \emph{homological width} of $L$ is
\[
w(L) = 1+\max\{\delta \mid \HFK_{\delta}(L)\neq 0\}-\min\{\delta \mid \HFK_{\delta}(L)\neq 0\}.
\]
If $w(L)=1$, we say that $L$ is \emph{thin}. One of the most useful features of our theory is that it measures width. Indeed, by Theorem \ref{thm:main},
\[
w(L) = 1+\max\{\delta \mid \HS_{\delta}(L,m)\neq 0\}-\min\{\delta \mid
\HS_{\delta}(L,m)\neq 0\}.
\]
Note that when $L$ is thin, its bigraded knot Floer homology is completely determined by $\HS(L,m)$ and $\Delta_L(t)$. Theorem \ref{thm:main} and the results of Ozsv\'ath-Szab\'o \cite{OSzGenus}, Ghiggini \cite{GhigginiFibred} and Ni \cite{NiActions} therefore imply the following.

\begin{corollary}
\noindent
\begin{enumerate}
\item $L$ is the $k$-component unlink if and only if $w(L) = k$ and $\rk_{\FF}\HS(L,m)=2^{m-1}$.
\item $L$ is the figure-eight knot if and only if $L$ is thin and $\Delta_L(t) = -t^{-1}+3-t$.
\item $L$ is the left- or right-handed trefoil if and only if $\rk_{\FF}\HS(L,m)=3\cdot 2^{m-1}$ and $\HS(L,m)$ is supported in the grading $-1$ or $+1$, respectively.
\end{enumerate}
\end{corollary}

Moreover, when $L$ is a thin \emph{knot}, its genus is simply the degree of
$\Delta_L(t)$ \cite{OSzGenus}, and it is fibered if and only if $\Delta_L(t)$
is monic \cite{GhigginiFibred, NiFibered}. In addition, the concordance
invariant $\tau(L)$, whose absolute value is a lower bound for the smooth
four-ball genus of $L$, is equal to the unique grading in which $\HS(L,m)$ is
supported \cite{OSz4Genus}.

It would be interesting to find a refinement of our construction which captures
the full bigrading on $\HFK$. However, the fact that the $\delta$-grading is
especially natural from our vantage hints that our theory may be well-suited to
certain applications, which we now describe.

The \emph{reduced Khovanov homology} of a link $L \subset S^3$ is a bigraded
vector space over $\F$,
\[
\Kh(L) = \bigoplus_{i,j}\Kh{}^{i,j}(L),
\]
which categorifies the Jones polynomial of $L$. In spite of their disparate
origins, Khovanov homology and knot Floer homology possess intriguing
similarities. For instance, although the bigrading on Khovanov homology behaves
quite differently from that on knot Floer homology, one can collapse the former
into a single grading,
\[
\Kh{}^{\delta}(L) = \bigoplus_{j/2-i = \delta}\Kh{}^{i,j}(L),
\]
and all available evidence points to the following conjecture, first formulated
by Rasmussen \cite{RasmussenHomologies} in the case of knots.
\begin{conjecture} \label{conj:khhfk}
For any link $L\subset S^3$,
\[
2^{\abs{L}-1-\eta(L)}\cdot\rk_{\F} \Kh^{\delta}(L) \ge \rk_{\F}
\HFK_{\delta}(L),
\]
where $\eta(L)$ is the rank of the Alexander module of $L$ over $\Z[H_1(S^3
\minus L;\mathbb{Z})]$.
\end{conjecture}
A proof of this conjecture would imply that Khovanov homology detects not only
the unknot, a fact recently established by Kronheimer and Mrowka
\cite{KronheimerMrowkaUnknot} using instanton Floer homology, but also the
trefoils and unlinks.\footnote{Using \cite{KronheimerMrowkaUnknot}, Hedden and Ni showed that the total rank of $\Kh$ detects the 2-component unlink \cite{HeddenNiRanks} and that $\Kh$, equipped with some additional algebraic structure, detects all unlinks \cite{HeddenNiUnlink}.}

Our new description for knot Floer homology bears an intriguing resemblance to recent work by Roberts \cite{RobertsTwisted} and Jaeger \cite{JaegerTwisted} that provides a spanning tree model for reduced Khovanov homology. Specifically, Roberts defines a complex $(C'(\DD),\partial')$ whose generators (over a field $\FF'$ of rational functions in several variables) corresponding to spanning trees, with the same grading as in our complex $C^\Omega(\DD)$. Moreover, the component of our differential $\partial^\Omega$ from the summand corresponding to a spanning tree $T$ to the summand corresponding to $T'$ is nonzero precisely when the same is true in $\partial'$. Jaeger then proves that when $L$ is a knot, the homology of $(C'(\DD), \partial')$ is precisely $\Kh(L) \otimes \FF'$ with its $\delta$ grading.\footnote{Note that Champanerkar--Kofman \cite{ChampanerkarKofmanSpanning} and Wehrli \cite{WehrliSpanning} independently discovered a different spanning tree model for Khovanov homology. However the differential on this complex is not known explicitly in terms of spanning trees; to compute it, one must effectively compute the entire Khovanov complex. An advantage of their model, however, is that it provides the entire bigrading on $\Kh$, not just the $\delta$ grading.} Because of this similarity, we hope that our new model for knot Floer homology may shed some light on Conjecture \ref{conj:khhfk}. For a simple example in this vein, see Corollary \ref{cor:disjointtrees} below.

Many of the ideas in this paper can be traced to work of Ozsv{\'a}th and Szab{\'o} \cite{OSzDouble}, who discovered a spectral sequence relating $\Kh(L)$ to the Heegaard Floer homology of $\Sigma(\ol L)$, the double cover of
$S^3$ branched along the mirror of $L$. Generalizations and applications of this spectral sequence have made for an active area of reseach in recent years; see, e.g., \cite{BaldwinPlamenevskaya,  BloomMonopole, GrigsbyWehrliColored, HeddenWatson, RobertsDouble, KronheimerMrowkaUnknot}. In forthcoming work, Ozsv{\'a}th, Szab{\'o}, and the first author define an analogous construction with twisted coefficients, the result of which is a spectral sequence $\SS$, converging to the twisted Heegaard Floer homology of $\Sigma(\ol L)$, whose $E_2$ page is a spanning tree complex that formally resembles both our complex $C^\Omega(\DD)$ and Roberts' $C'(\DD)$.\footnote{Kriz and Kriz \cite{KrizKrizInvariant} have proven that the homology of $(E_2(\SS), d_2(\SS))$ is a link invariant.} In contrast with our setup, it is not clear whether $\SS$ collapses at its $E_3$ page. However, the similarities between $(E_2(\SS),d_2(\SS))$ and $(C^\Omega(\DD),\partial^{\Omega})$ suggest that one might hope to prove a relationship between $\HFK(L)$ and $\HF(\Sigma(L))$, as was also proposed by Greene \cite{GreeneSpanning}. Available evidence suggests the following.

\begin{conjecture}
\label{conj:hfkbdc} For any link $L\subset S^3$, \[\rk_{\F}
\HFK_{*+(|L|-1)/2}(L) \ge 2^{\abs{L}-1-\eta(L)}\cdot \rk_{\F}
\HF_{*}(\Sigma(L)),\] where the two gradings above are the mod $2$ $\delta$-
and Maslov gradings, respectively.
\end{conjecture}

A third potential application of our construction has to do with
\emph{mutation}, an operation on planar link diagrams in which one removes a
4-strand tangle and reglues it after a half-rotation, as in the figure below.
Mutation leaves all classical link polynomials unchanged and preserves the
homeomorphism type of the branched double cover. Moreover, Wehrli
\cite{WehrliMutation} and Bloom \cite{BloomMutation} have shown that it
preserves reduced Khovanov homology (with coefficients in $\F$).
\begin{figure}[!htbp]

\begin{center}
 \labellist
 \hair 2pt
 \pinlabel $R$ at 68 72
  \pinlabel \rotatebox{180}{$R$} at 273 72

\endlabellist
\includegraphics[width=4.8cm]{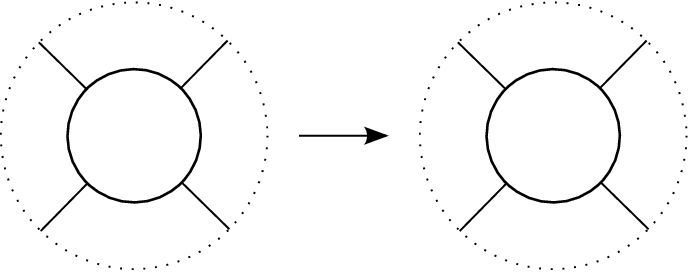}
\end{center}
\end{figure}
In contrast, mutation can change the bigraded knot Floer homology of a knot
since it need not preserve Seifert genus \cite{OSzMutation}. Somewhat
surprisingly, however, the computations in \cite{BaldwinGillam} support the
following conjecture:
\begin{conjecture}
\label{conj:mutation} If $L'$ is obtained from $L$ by mutation, then
$\HFK_{\delta}(L) \cong \HFK_{\delta}(L')$.
\end{conjecture}
Indeed, if Conjectures \ref{conj:khhfk} and \ref{conj:hfkbdc} hold, then
mutation cannot have too drastic an effect on these $\delta$-graded groups.
Moreover, since the Alexander polynomial is mutation-invariant, a proof of this
conjecture would imply that for thin knots, mutation preserves genus,
fiberedness and the $\tau$ invariant.

Our model provides a reasonable starting point from which to approach Conjecture \ref{conj:mutation} since $(C^\Omega(\DD), \partial^\Omega)$ is formulated largely in terms of black graph data, much of which is preserved by mutation. In particular, spanning trees of $B(\DD)$ are in 1-to-1 correspondence with spanning trees of $B(\DD')$ for any mutant $\DD'$ of $\DD$.


One of the most compelling features of our construction is that the complex $(C^\Omega(\DD), \partial^\Omega)$ is largely determined by formal properties; very little direct computation is required. This suggests that our approach might be
used to give an axiomatic characterization of knot Floer homology or to prove that $\HFK$ is isomorphic to other knot homology theories, such Kronheimer and Mrowka's monopole knot homology \cite{KronheimerMrowkaSutures}. It is known (or
soon will be) that monopole knot homology agrees with knot Floer homology, as a result of nearly one thousand pages of work of Taubes \cite{TaubesECH=SW1, TaubesECH=SW2, TaubesECH=SW3, TaubesECH=SW4, TaubesECH=SW5}, Kutluhan--Lee--Taubes
\cite{KutluhanLeeTaubes1, KutluhanLeeTaubes2, KutluhanLeeTaubes3, KutluhanLeeTaubes4, KutluhanLeeTaubes5} and Colin--Ghiggini--Honda \cite{ColinGhigginiHonda1, ColinGhigginiHonda2}, combined with work of Lekili \cite{LekiliThesis}. Still, it would be nice to prove this equivalence (and to find a combinatorial formulation of monopole knot homology) without resorting to their SW = ECH = HF machinery. The key will be to define an analogue of Manolescu's exact triangle in the monopole setting; if done correctly, almost everything should follow from purely formal considerations.


Finally, it is worth mentioning some advantages of our model over the other combinatorial formulations of knot Floer homology. For an $n$-crossing projection with $2n$ marked points, the dimension of our complex (over $\FF$) is $s(\DD) \cdot 2^{2n-1}$, where $s(\DD) \leq 2^n$ is the number of spanning trees of $B(\DD)$, whereas the dimension of Manolescu--Ozsv\'ath--Sarkar's grid complex is on the order of $n!$ (albeit over the simpler field $\F$). Thus, our theory should be more computable for large knots. Furthermore, in contrast with Ozsv{\'a}th and Szab{\'o}'s singular braid model \cite{OSzCube}, our construction does not require a braid projection, and it applies to arbitrary links rather than just knots. (Of course, the main drawback is that our complex computes only $\HFK(L)$ with its $\delta$ grading, not the more robust version $\HFKm(L)$ or the bigrading on $\HFK(L)$.)

\subsection*{Organization}
In Section \ref{sec:complex}, we define the complex $(C^\Omega(\DD),\partial^{\Omega})$. In Section \ref{sec:background}, we provide background on knot Floer homology with twisted coefficients and we introduce an action on knot Floer homology defined by counting disks which pass over basepoints. In Section \ref{sec:unlinks}, we compute the twisted knot Floer
homologies of unknots and unlinks in terms of this action. In Section \ref{sec:ss}, we iterate Manolescu's exact triangle with twisted coefficients in $\FF$. The result of this iteration is a filtered cube of resolutions complex that computes knot Floer homology. In Section \ref{sec:delta}, we determine the $\delta$-grading shifts of the maps in this filtered complex and show that the associated spectral sequence $\SS^{\Omega}_{\FF}$ collapses at its $E_3$ page. In Section \ref{sec:twisted}, we compute the $(E_2,d_2)$ page of $\SS^{\Omega}_\FF$ and show that it is isomorphic to $(C^\Omega(\DD), \partial^{\Omega})$, proving Theorem \ref{thm:main}.

\subsection*{Acknowledgements}
The authors thank Jon Bloom, Josh Greene, Eli Grigsby, Ciprian Manolescu, Peter
Ozsv{\'a}th and Zolt{\'a}n Szab{\'o} for helpful conversations. In particular,
many of the ideas found in this text had their origins in work of Peter and
Zolt{\'a}n. The first author is also grateful for Josh's 4:00 AM nightmare
which led to a breakthrough in this project.

\section{Definition of the complex}
\label{sec:complex}
Fix an oriented, connected planar projection $\DD$ of $L$. Let $c_1,\dots,c_n$ denote the crossings of $\DD$, and let $\p=\{p_1,\dots,p_m\}$ be a set of marked points on the edges of $\DD$ so that every edge is marked, and so that $p_1$ lies on an outermost edge of $\DD$. Let $n_+(\DD)$ and $n_-(\DD)$ denote the numbers of positive and negative crossings in $\DD$, respectively. Additionally, we fix an arbitrary orientation on the edges of $B(\DD)$.

The $0$- and $1$-resolutions of $\DD$ at a crossing $c_j$ are the diagrams obtained from $\DD$ by smoothing $c_j$ according to the convention in Figure \ref{fig:resolutions}. Taking the $\infty$-resolution of $c_j$ means leaving
the crossing unchanged. For each $I = (I_1,\dots,I_n) \in \{0,1\}^n$, let $\DD_I$ be the \emph{complete resolution} of $\DD$ gotten by replacing $c_j$ with its $I_j$-resolution. $\DD_I$ is a planar unlink, and we shall orient its
components as the boundaries of the black regions. (This orientation is not, in general, consistent with any orientation on $L$.) Let $\abs{\DD_I}$ denote the number of components of $\DD_I$, and let $\abs{I}=I_1+\dots+I_n$.

\begin{figure}[!htbp]
 \labellist
 \small\hair 2pt \pinlabel $0$ at 27 8 \pinlabel $1$ at 91 8
 \pinlabel $\infty$ at 156 8
\endlabellist
\begin{center}
\includegraphics[width=6.5cm]{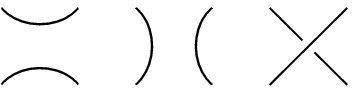}
\caption{\quad The $0$-, $1$- and $\infty$- resolutions of the crossing on the
right.} \label{fig:resolutions}
\end{center}
\end{figure}

For $j=1,\dots,n$, let $e_j$ denote the edge of $B(\DD)$ which corresponds to
the crossing $c_j$. Given a spanning subgraph $\gamma\subset B(\DD)$ --- i.e.,
a subgraph containing all vertices of $B(\DD)$ --- one obtains a complete
resolution of $\DD$ by smoothing each crossing $c_j$ in such a way as to join
the black regions incident to $c_j$ if and only if $e_j$ is contained in
$\gamma$; see Figure \ref{fig:spanning}(b). Let $\gamma_I$ denote the subgraph
corresponding to the resolution $\DD_I$. It is not hard to see that $\DD_I$ is
connected if and only if $\gamma_I$ is a spanning tree.


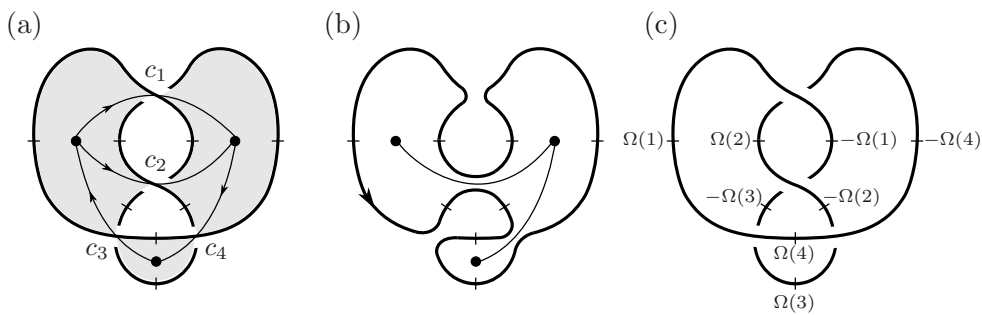
\begin{figure}[!htbp]
\labellist
 \pinlabel (a) at 40 182
 \pinlabel (b) at 241 182
 \pinlabel (c) at 442 182
 \small \hair 2pt
 \pinlabel $c_1$ at 123 152
 \pinlabel $c_2$ at 123 95
 \pinlabel $c_3$ at 86 39
 \pinlabel $c_4$ at 161 39
 \tiny \hair 2pt
 \pinlabel $\Omega(1)$ at 431 109
 \pinlabel $-\Omega(1)$ at 573 109
 \pinlabel $\Omega(2)$ at 486 109
 \pinlabel $-\Omega(2)$ at 562 73
 \pinlabel $\Omega(3)$ at 526 6
 \pinlabel $-\Omega(3)$ at 488 73
 \pinlabel $\Omega(4)$ at 526 37
 \pinlabel $-\Omega(4)$ at 626 109
\endlabellist
\begin{center}
\includegraphics[width=13.5cm]{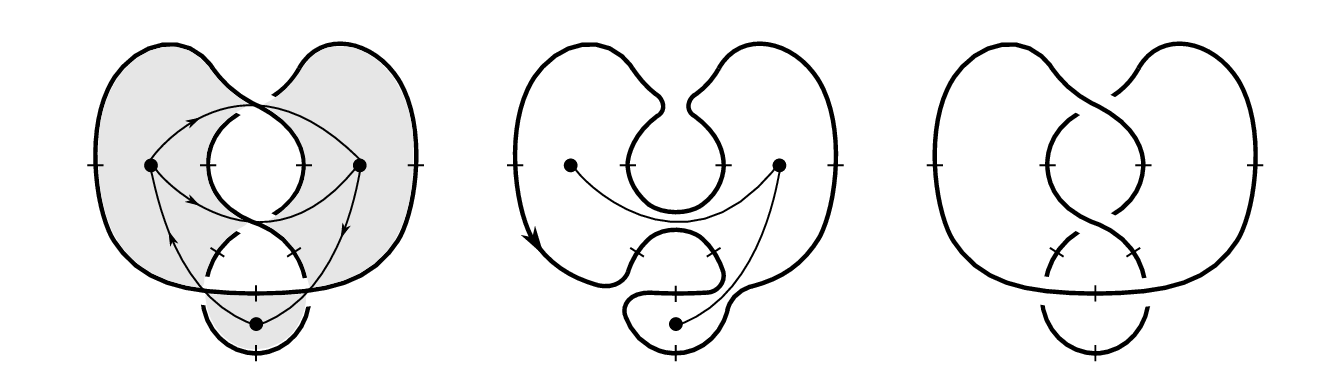}
\caption{\quad (a) A pointed diagram $\DD$ for the unknot, along with its black graph and a choice of orientations on the edges of $B(\DD)$. The marked points are indicated by dashes. (b) A spanning tree of $B(\DD)$ and the corresponding resolution of $\DD$. (c) The values $r_i$ associated to the marked points, as in Definition \ref{def:special}.} \label{fig:spanning}
\end{center}
\end{figure}

In order to work with twisted coefficients, we need to specify certain cohomology classes via the following definition.

\begin{definition} \label{def:weights}
A \emph{system of weights} is a tuple $\r = (r_1, \dots, r_m)$ satisfying $r_1 + \dots + r_m = 0$, which are associated with the marked points $p_1, \dots, p_m$. Given $\r$, let $\omega_\r \in H^2(S^3 \minus L;\Z)$ be the cohomology class whose evaluation on the boundary torus of a tubular neighborhood of each component $L_j$ of $L$ equals the sum of the weights on $L_j$. (Note that the sum of these tori equals zero in homology, so the condition that $r_1 + \dots + r_m = 0$ is needed.) A system of weights is called \emph{generic} if for every $I \in \{0,1\}^n$ for which the resolution $\DD_I$ is disconnected, the sum of the weights on each component of $\DD_I$ is nonzero.
\end{definition}

We shall often make use of systems of weights coming from the following construction.

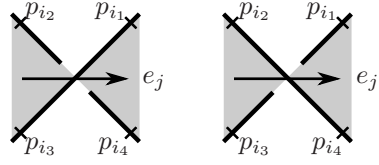
\begin{figure}[!htbp]
\labellist \small \pinlabel $p_{i_2}$ at 18 53 \pinlabel $p_{i_4}$ at 44 5
\pinlabel $p_{i_1}$ at 43 53 \pinlabel $p_{i_3}$ at 18 5 \pinlabel $e_j$ at 58
28 \pinlabel $p_{i_2}$ at 94 53 \pinlabel $p_{i_4}$ at 120 5 \pinlabel
$p_{i_1}$ at 118 53 \pinlabel $p_{i_3}$ at 94 5 \pinlabel $e_j$ at 134 28
\endlabellist
\begin{center}
\includegraphics[width=4.6cm]{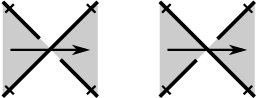}
\caption{\quad Two possibilities for the neighborhood of $c_j$.} \label{fig:convention}
\end{center}
\end{figure}

\begin{definition} \label{def:special}
A function $\Omega\co \{1,\dots,n\}\rightarrow \mathbb{Z}$ is called \emph{generic} if the values $\Omega(1), \dots, \Omega(n)$ do not satisfy any nontrivial linear relation with coefficients in $\{-1,0,1\}$. (For instance, the function $\Omega(i)=2^i$ is generic.) Such a function (generic or not) determines a system of weights $\r_\Omega = (r_1, \dots, r_m)$ by the following construction. For each $j=1,\dots,n$, view the crossing $c_j$ so that the oriented edge $e_j$ points from left to right. If $p_{i_1}$, $p_{i_2}$, $p_{i_3}$, and $p_{i_4}$ are the closest marked points to $c_j$ on the four edges of $\DD$ incident to $c_j$, starting in the upper right and going counterclockwise, define $r_{i_2} = \Omega(j)$ and $r_{i_4} = -\Omega(j)$. This convention determines $2n$ of the integers $r_1, \dots, r_m$. Define the remaining ones to be zero. (See Figure \ref{fig:spanning}(c) for an example.) Additionally, we call $i_1$ and $i_3$ the \emph{special indices} associated to $c_j$.
\end{definition}

\begin{lemma} \label{lemma:generic}
If $\r = \r_\Omega$ for a function $\Omega \co \{1,\dots, n\} \to \Z$, then $\omega_{\r} = 0$ in $H^2(S^3 \minus L;\Z)$. Moreover, if $\Omega$ is generic in the sense of Definition \ref{def:special}, then $\r$ is generic in the sense of Definition \ref{def:weights}.
\end{lemma}

\begin{proof}[Proof of Lemma \ref{lemma:generic}]
The first statement is true because for each $j = 1, \dots, n$, the two marked points with weights $\pm \Omega(j)$ lie on the same component of $L$, so the sum of the weights on each component is $0$.

For the second statement, let $I \in \{0,1\}^n$ be such that $\DD_I$ is disconnected, and call its components $\DD_I^1, \dots, \DD_I^{\ell_I}$. Let $i \in \{1, \dots, \ell_I\}$, and suppose that $p_{a_1}, \dots, p_{a_k}$ are the marked points on $\DD_I^i$. Suppose, toward a contradiction, that
\begin{equation} \label{eq:componentweight}
r_{a_1} + \dots + r_{a_k} = 0.
\end{equation}
By Definition \ref{def:special}, the nonzero terms on the left-hand side of \eqref{eq:componentweight} are distinct elements of the set $\{\pm \Omega(1), \dots, \pm\Omega(n)\}$, so \eqref{eq:componentweight} gives a linear relation among $\Omega(1), \dots, \Omega(n)$ with coefficients in $\{-1,0,1\}$. Because the diagram $\DD$ is connected, there is some crossing $c_j$ which connects $\DD_I^i$ with some other component of $\DD_I$. By Definition \ref{def:special}, one of the two marked points with weight $\pm \Omega(j)$ is on $\DD_I^i$ and one is not. Therefore, the coefficient of $\Omega(j)$ in \eqref{eq:componentweight} is nonzero, which contradicts the genericity of $\Omega$.
\end{proof}

Henceforth, we fix a generic system of weights $\r$, not necessarily arising from Definition \ref{def:special}.

Let $\mathbb{F}[\mathbb{Z}]$ denote the mod-2 group ring of the integers, which we think of as the ring of Laurent polynomials in $T$ with coefficients in $\F$. As in the Introduction, let $\FF = \F(T)$ denote the field of rational functions in $T$ over $\F$; this equals the fraction field of $\F[\Z]$. Let $\YY$ denote the vector space over $\FF$ generated freely by $y_1, \dots, y_m$.

Let $\RR(\DD)$ denote the set of $I\in \{0,1\}^n$ for which $\DD_I$ is
connected. For each $I \in \RR(\DD)$, let $\sigma_I \in \mathfrak{S}_{m}$ be
the permutation of $\{1, \dots, m\}$ such that $\sigma_I(1)=1$ and such that
the marked points are ordered $p_{\sigma_I(1)}, \dots, p_{\sigma_I(m)}$
according to the orientation on $\DD_I$. Let $\YY_I$ be the quotient of $\YY$ by the relation
\begin{equation} \label{eq:YIrelation}
\sum_{i=1}^m T^{r_{\sigma_I}(1) + \dots + r_{\sigma_I}(i)} y_{\sigma_I(i)} = 0,
\end{equation}
so that $\dim_\FF(\YY_I) = m-1$. That is, the power of $T$ in the coefficient of
$y_j$ is the sum of the weights of the marked points on the oriented segment of
$\KK_I$ from $p_1$ to $p_j$, including the endpoints. Note that the coefficient of
$y_{\sigma_I(m)}$ in \eqref{eq:YIrelation} is always $1$, since $\sum_{i=1}^m r_i=0$.

For $I,I'' \in \RR(\DD)$, we say that $I''$ is a \emph{double successor} of $I$
if it is obtained from $I$ by changing two $0$s to $1$s. For every such pair
$I,I''$, we shall define a linear map
\[
d_{I,I''}\co \Lambda^*(\YY_I)\rightarrow \Lambda^*(\YY_{I''}),
\]
as follows. Suppose $j_1$ and $j_2$ are the two coordinates in which $I$ and
$I''$ differ, and let $I^{1}$ (resp. $I^{2}$) be the tuple obtained from $I$ by
changing its $j_1\Th$ (resp. $j_2\Th$) coordinate from a $0$ to a $1$. Without
loss of generality, let us assume that $\sigma_I$ is the identity. Choose
$1\leq a <b<c<d\leq m$ so that $a,c$ are the special indices associated to
$c_{j_1}$ and $b,d$ are the special indices associated to $c_{j_2}$. In
particular, this establishes which crossing is $c_{j_1}$ and which is
$c_{j_2}$; see Figure \ref{fig:ResUnlink} for an example.

\begin{figure}[!htbp]
\labellist \pinlabel $\DD_I$ at 22 195 \pinlabel $\DD_{I^2}$ at 152 108
\pinlabel $\DD_{I^1}$ at 152 281 \pinlabel $\DD_{I''}$ at 281 195

\small \pinlabel $p_m$ at 11 168 \pinlabel $p_1$ at 11 143 \pinlabel $p_a$ at 61 97
\pinlabel $p_b$ at 56 178 \pinlabel $p_c$ at 93 118 \pinlabel $p_d$ at 89 200
\pinlabel $p_{a+1}$ at 53 118 \pinlabel $p_{b+1}$ at 99 178 \pinlabel $p_{c+1}$
at 96 97 \pinlabel $p_{d+1}$ at 65 200

\pinlabel $p_m$ at 141 81 \pinlabel $p_1$ at 141 56 \pinlabel $p_a$ at 191 10
\pinlabel $p_b$ at 186 91 \pinlabel $p_c$ at 224 31 \pinlabel $p_d$ at 222 113
\pinlabel $p_{a+1}$ at 183 31 \pinlabel $p_{b+1}$ at 229 91 \pinlabel $p_{c+1}$
at 226 10 \pinlabel $p_{d+1}$ at 192 113

\pinlabel $p_m$ at 141 254 \pinlabel $p_1$ at 141 229 \pinlabel $p_a$ at 191 183
\pinlabel $p_b$ at 186 264 \pinlabel $p_c$ at 224 204 \pinlabel $p_d$ at 222 288
\pinlabel $p_{a+1}$ at 183 204 \pinlabel $p_{b+1}$ at 229 264 \pinlabel $p_{c+1}$
at 226 183 \pinlabel $p_{d+1}$ at 192 286

\pinlabel $p_m$ at 270 168 \pinlabel $p_1$ at 270 143 \pinlabel $p_a$ at 320 97
\pinlabel $p_b$ at 316 178 \pinlabel $p_c$ at 352 118 \pinlabel $p_d$ at 348 200
\pinlabel $p_{a+1}$ at 312 118 \pinlabel $p_{b+1}$ at 358 178 \pinlabel $p_{c+1}$
at 355 97 \pinlabel $p_{d+1}$ at 324 200

\endlabellist
\begin{center}
\includegraphics[scale=0.9]{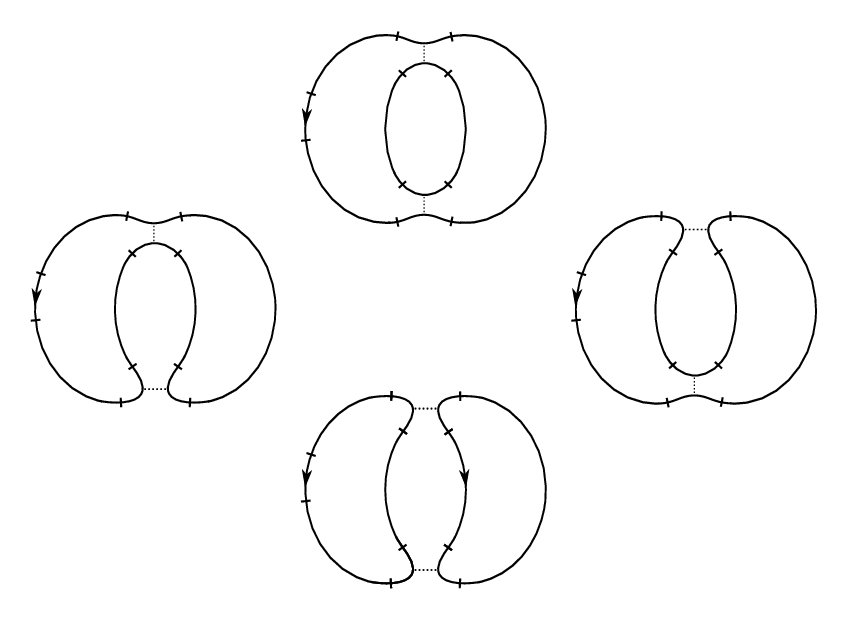}
\caption{\quad The resolutions $\DD_I$, $\DD_{I^1}$, $\DD_{I^2}$ and $\DD_{I''}$ in the case that $\gamma_{I^1} = \gamma_I \cup e_{j_1}$, along with the marked points $p_i$. The dotted lines indicate the traces of the crossings $c_{j_1}$ (bottom) and $c_{j_2}$ (top). (If $m = d$, then $p_1$ plays the role of $p_{d+1}$.)} \label{fig:ResUnlink}
\end{center}
\end{figure}
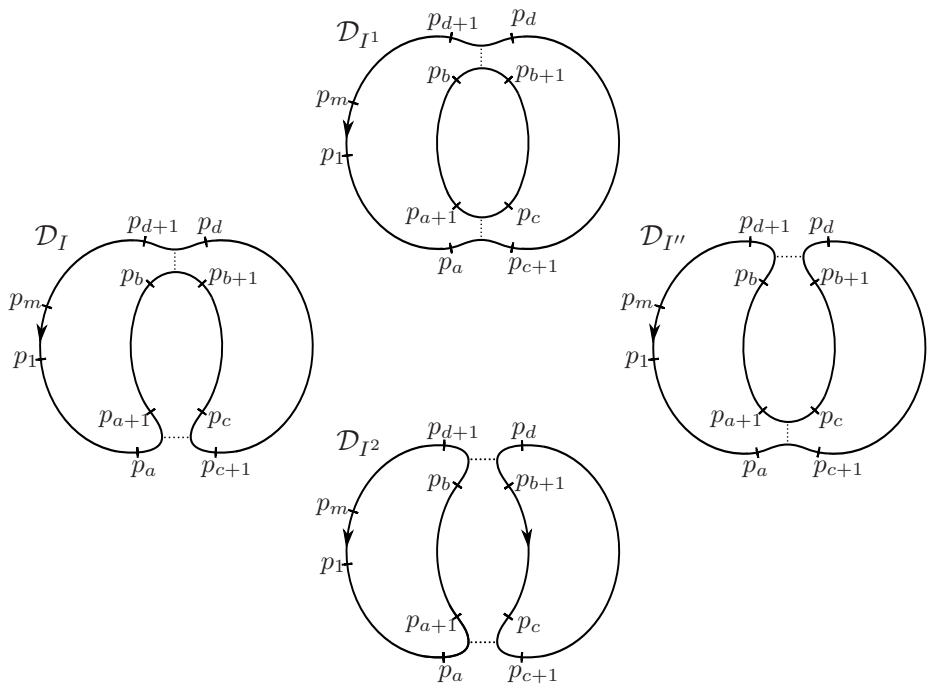

In $\DD_{I^1}$, the marked points on one component are
$p_1,\dots,p_a,p_{c+1},\dots,p_m$, and those on the other are $p_{a+1},
\dots,p_c$, ordered according to the orientation of $\DD_{I^1}$. Likewise, the
marked points on the two components of $\DD_{I^2}$ are
$p_1,\dots,p_b,p_{d+1},\dots,p_m$ and $p_{b+1},\dots,p_d$. Let
\[
A = \sum_{i=1}^a r_i, \quad B = \sum_{i=a+1}^b r_i,\quad C = \sum_{i=b+1}^c
r_i, \quad D = \sum_{i=c+1}^d r_i.
\]
The weights of the components of $\DD_{I^1}$ and $\DD_{I^2}$ that do not contain $p_1$ are $B+C$ and $C+D$,
respectively. The genericity of $\r$ guarantees that these two numbers are are nonzero.

%

In defining the map $d_{I,I''}$, there are two cases to consider; either
\[
\gamma_{I^1} = \gamma_I \cup e_{j_1} \ \text{ or } \ \gamma_{I^1}  =
\gamma_I \minus e_{j_1}.
\] We shall distinguish these cases
with a number $\nu = \nu_{I,I''} \in \{0,1\},$ defined to be $1$ in the first
case and $0$ in the second.

\begin{definition} \label{def:d}
The map $d_{I,I''}$ is the sum
\[
d_{I,I''} = d_{I,I''}^{1,1} + d_{I,I''}^{1,2} + d_{I,I''}^{2,1} +
d_{I,I''}^{2,2},
\]
where $d_{I,I''}^{k,l}\co \Lambda^*(\YY_I) \rightarrow \Lambda^*(\YY_{I''})$ are the
$\FF$-linear maps defined by the rules (omitting the subscripts for
convenience)
\begin{align}
\label{eq:d11(1)} d^{1,1}(1) &= 0 \\
\label{eq:d12(1)} d^{1,2}(1) &= \frac{T^{\nu C}}{1+T^{C+D}} \\
\label{eq:d21(1)} d^{2,1}(1) &= \frac{T^{B + \nu C}}{1+T^{B+C}} \\
\label{eq:d22(1)} d^{2,2}(1) &= \frac{T^{-A + \nu C}}{ (1+T^{B+C})(1+T^{C+D}) }
\sum_{i=1}^{m} T^{r_1 + \dots + r_i} y_i,
\end{align}
and for any monomial $x$ in $y_1, \dots, y_m$, any $i=1, \dots, m$, and any $k,l \in \{1,2\}$,
\begin{equation} \label{eq:dkl(xyi)}
d^{k,l}(x y_i) = \begin{cases}
 d^{1,l}(x) y_i  + d^{2,l}(x) & \text{if } k=1 \text{ and } i\in \{a,c\} \\
 d^{k,1}(x) y_i  + d^{k,2}(x) & \text{if } l=1 \text{ and } i\in \{b,d\} \\
 d^{k,l}(x) y_i  & \text{otherwise}.
\end{cases}
\end{equation}
\end{definition}

To be more precise, viewing $\Lambda^*(\YY_I)$ and $\Lambda^*(\YY_{I''})$ as modules over the exterior algebra $\Lambda^*(\YY)$, $d^{2,2}$ is defined to be the $\Lambda^*(\YY)$-module homomorphism determined by \eqref{eq:d22(1)}. Since the right-hand side of \eqref{eq:d22(1)} is a multiple of the defining relator for $\YY_I$ given in \eqref{eq:YIrelation}, $d^{2,2}$ is well-defined. Next, $d^{1,2}$ and $d^{2,1}$ are defined on all monomials by induction on degree using \eqref{eq:d12(1)}, \eqref{eq:d21(1)}, and \eqref{eq:dkl(xyi)}. To check that these are well-defined --- i.e., that they vanish on multiples of the defining relator for $\YY_I$ --- note that the values of $d^{1,2}(1)$ and $d^{2,1}(1)$ are chosen such that
\begin{align}
\label{eq:d12welldefined}
d^{1,2}\left( \sum_{i=1}^m T^{r_1 + \dots + r_i} y_i \right) &= d^{1,2}(1) \sum_{i=1}^m T^{r_1 + \dots + r_i} y_i + (T^{A} + T^{A+B+C}) d^{2,2}(1) = 0 \\
\label{eq:d21welldefined}
d^{2,1}\left( \sum_{i=1}^m T^{r_1 + \dots + r_i} y_i \right) &= d^{2,1}(1) \sum_{i=1}^m T^{r_1 + \dots + r_i} y_i + (T^{A+B} + T^{A+B+C+D}) d^{2,2}(1) = 0.
\end{align}
Induction using \eqref{eq:dkl(xyi)} then shows that $d^{1,2}$ and $d^{2,1}$ vanish on any expression of the form
\[
y_{i_1} \cdots y_{i_k} \sum_{i=1}^m T^{r_1 + \dots + r_i} y_i,
\]
as required. Finally, $d^{1,1}$ is defined on all monomials by induction on degree using \eqref{eq:d11(1)} and \eqref{eq:dkl(xyi)}, and the proof of well-definedness goes through in the same way.

\begin{remark}
The map $d^{1,1}$ decreases degree (of polynomials in the $y_i$) by one, $d^{1,2}$ and $d^{2,1}$ preserve degree, and $d^{2,2}$ increases degree by one. Knowing just this, the total map $d_{I,I''}$ is determined up to an overall scalar by \eqref{eq:YIrelation} and \eqref{eq:dkl(xyi)}, since the value of $d^{2,2}(1)$ is forced to be a multiple of the relator on $\YY_I$, and the values of $d^{1,2}(1)$ and $d^{2,1}(1)$ are forced in order for \eqref{eq:d12welldefined} and \eqref{eq:d21welldefined} to hold. In particular, the maps in the two cases distinguished by $\nu$ differ only by an overall factor of $T^C$. (Compare Section \ref{sec:commutation}.)
\end{remark}

We now define the complex $(C^\r(\DD),\partial^{\r})$ as follows:
\begin{definition} \label{def:complex}
Define
\[
C^\r(\DD) = \bigoplus_{I\in\RR(\DD)} \Lambda^*(\YY_I),
\]
where $\Lambda^*(\YY_I)$ is supported in the grading $(\abs{I}-n_-(\DD))/2$, and let $\partial^\r$ be the direct sum of the maps $d_{I,I''}\co \Lambda^*(\YY_I) \rightarrow \Lambda^*(\YY_{I''})$. If $\r = \r_\Omega$ as in Definition \ref{def:special}, we denote $(C^\r(\DD), \partial^\r)$ by $(C^\Omega(\DD), \partial^\Omega)$ as in the Introduction.
\end{definition}
The fact that $\partial^\r$ squares to zero will be established at the end of Section \ref{sec:twisted}, when we identify $(C^\r(\DD), \partial^\r)$ with the $E^2$ page of the cube of resolutions spectral sequence that we construct below. A more general version of Theorem \ref{thm:main} is then as follows.

\begin{theorem} \label{thm:main2}
The homology of $(C^\r(\DD),\partial^{\r})$ is isomorphic as a graded $\FF$-vector space to
\[
\ul\HFK(L, \omega_\r; \FF) \otimes_\FF (V^{\otimes(m-|L|)} \otimes_{\F} \FF),
\]
where $\ul\HFK(L, \omega_\r; \FF)$ denotes the twisted knot Floer homology of $L$ with perturbation $\omega_\r$, equipped with its $\delta$-grading. (See Proposition \ref{prop:invariancetwisted} for a precise definition of this invariant.)
\end{theorem}

When $\r = \r_\Omega$, we have $\omega_\r=0$, so $\ul\HFK(L, \omega_\r; \FF)$ is simply the untwisted knot Floer homology, tensored with $\FF$, giving Theorem \ref{thm:main}.

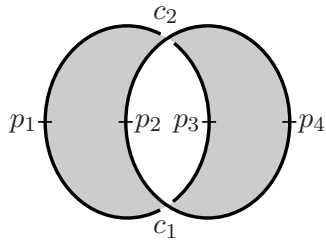
\begin{figure}
\labellist
 \pinlabel $p_1$ at 8 60
 \pinlabel $p_2$ at 67 60
 \pinlabel $p_3$ at 85 60
 \pinlabel $p_4$ at 144 60
 \pinlabel $c_1$ at 75 10
 \pinlabel $c_2$ at 75 111
\endlabellist
\includegraphics[scale=0.8]{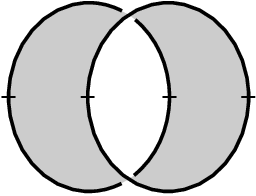}
\caption{Diagram for a two-component unlink whose cube of resolutions is Figure \ref{fig:ResUnlink}.} \label{fig:2compunlink}
\end{figure}

\begin{example} \label{ex:2compunlink}
Let $\DD$ be the diagram for the two-component unlink $L$ shown in Figure \ref{fig:2compunlink}, whose cube of resolutions is precisely Figure \ref{fig:ResUnlink} with $a=1$, $b=2$, $c=3$, and $d=m=4$. The connected resolutions of $\DD$ correspond to $I = (0,0)$ and $I'' = (1,1)$; note that $\nu_{I,I''} = 1$. For ease of notation, define $r = r_1 = A$, $s = r_2 = B$, $t = r_3 = C$, and $u = r_4 = D$. The defining relations on $\YY_I$ and $\YY_{I''}$ give:
\[
\begin{aligned}
T^r y_1 + T^{r+s} y_2 + T^{r+s+t}y_3 + T^{r+s+t+u} y_4 &=0 & \text{in }&\YY_{I} \\
T^r y_1 + T^{r+u} y_4 + T^{r+t+u}y_3 + T^{r+s+t+u} y_2 &=0 & \text{in }&\YY_{I''}. \\
\end{aligned}
\]
We shall use these relations and the fact that $r + s + t + u = 0$ to eliminate $y_4$ wherever it appears.
For conciseness, we define
\[
\lambda = \frac{1}{1+T^{t+u}} \quad \text{and} \quad \mu = \frac{T^{s+t}}{1+T^{s+t}},
\]
so that
\[
\begin{aligned}
d^{2,2}(1) &= T^{t} \lambda \mu (y_1 + T^{s} \, y_2 + T^{s + t} \, y_3 + T^{s + t + u} \, y_4) \\
&= T^t \lambda \mu (y_1 + T^{s} \, y_2 + T^{s + t} \, y_3 + T^{s + t + u} (T^t \, y_3 + T^{s+t} \, y_2 + T^{r+s+t}\, y_1)) \\
&= T^t \lambda \mu ( (1+T^{s+t}) \, y_1 + T^s(1 + T^{s+2t+u}) \, y_2 + T^{s + t}(1+T^{t+u}) \, y_3 )\\
&= T^t \lambda \mu ( \mu^{-1} \, y_1 + T^{2s+t}(\mu^{-1} + \lambda^{-1}) \, y_2 + T^{s + t}\lambda^{-1} \, y_3 )\\
&= T^t (\lambda \, y_1 + T^s(\lambda + \mu) \, y_2 + \mu \, y_3).
\end{aligned}
\]
Using the inductive procedure described above, we can see that the values of the four functions $d^{1,1}$, $d^{1,2}$, $d^{2,1}$, and $d^{2,2}$ on a basis for $\Lambda^*(Y_I)$ are as follows:

\[
\begin{array}{|c|c|} \hline
x & d^{2,2}(x) \\ \hline
1 & T^t \lambda \, y_1 + T^{s+t} (\lambda + \mu) \, y_2 + T^t \mu \, y_3 \\
y_1 & T^{s+t}(\lambda + \mu) \, y_1 y_2 + T^t \mu \, y_1 y_3 \\
y_2 & T^t \lambda \, y_1 y_2 + T^t \mu \, y_2 y_3 \\
y_3 & T^t \lambda \, y_1 y_3 + T^{s+t}(\lambda + \mu) \, y_2 y_3  \\
y_1 y_2 & T^t \mu \, y_1 y_2 y_3 \\
y_1 y_3 & T^{s+t} (\lambda + \mu) \, y_1 y_2 y_3   \\
y_2 y_3 & T^t \lambda \, y_1 y_2 y_3  \\
y_1 y_2 y_3 & 0 \\ \hline
\end{array}
\]

\[
\begin{array}{|c|c|c|} \hline
x & d^{1,2}(x) & d^{2,1}(x) \\ \hline
1
 & T^t \lambda
 & \mu \\
y_1
 & T^{s+t}(\lambda + \mu) \, y_2 + T^t \mu \, y_3
 & \mu \, y_1 \\
y_2
 & T^t \lambda \, y_2
 & T^t \lambda \, y_1 + T^{s+t} (1+\lambda) \, y_2 + T^t \mu \, y_3 \\
y_3
 & T^t \lambda \, y_1 + T^{s+t}(\lambda + \mu) \, y_2 + T^t(\lambda +\mu) \, y_3
 & \mu \, y_3 \\
y_1 y_2
 & T^t \mu \, y_2 y_3
 & T^{s+t} (1+\lambda) \, y_1 y_2 + T^t \mu \, y_1 y_3 \\
y_1 y_3
 & T^{s+t}(\lambda + \mu) \, y_1 y_2 + T^t \mu \, y_1 y_3 + T^{s+t}(\lambda + \mu) \, y_2 y_3
 & \mu \, y_1 y_3 \\
y_2 y_3
 & T^t \lambda \, y_1 y_2 + T^t(\lambda +\mu) \, y_2 y_3
 & T^t \lambda \, y_1 y_3 + T^{s+t} (1+\lambda) \, y_2 y_3 \\
y_1 y_2 y_3
 & T^t \mu \, y_1 y_2 y_3
 & T^{s+t} (1+\lambda) \, y_1 y_2 y_3  \\ \hline
\end{array}
\]

\[
\begin{array}{|c|c|} \hline
x & d^{1,1}(x) \\ \hline
1 & 0 \\
y_1 & \mu \\
y_2 & T^t \lambda \\
y_3 & \mu \\
y_1 y_2 & T^{s+t} (1+\lambda) \, y_2 + T^t \mu \, y_3 \\
y_1 y_3 & \mu \, y_1 + \mu \, y_3 \\
y_2 y_3 & T^t \lambda \, y_1 + T^{s+t} (1+\lambda) \, y_2 + T^t (\lambda + \mu) \, y_3 \\
y_1 y_2 y_3 & T^{s+t} (1+\lambda) \, y_1 y_2 + T^t \mu \, y_1 y_3 + T^{s+t} (1+\lambda) \, y_2 y_3
\\ \hline
\end{array}
\]

If the weights are determined by a generic function $\Omega\co \{1,2\} \to \Z$ as in Definition \ref{def:special}, we have $r = -t$ and $s = -u$, while $s \ne \pm t$ and both $s$ and $t$ are nonzero. In this case, some linear algebra shows that $d$ has rank $4$, with kernel generated by the following four elements of $\Lambda^*(\YY_I)$:
\[
\begin{gathered}
(1+T^s + T^t + T^{s+t}) + (T^{-t}+T^s) \, y_1 + (T^{s-t} + T^{s+t}) \, y_2 + (T^s + T^t) \, y_3 \\
(T^s + T^t) + (1 + T^s + T^t + T^{s+t}) \, y_1 + (T^{s-t} + T^{s+t}) \, y_1 y_2 + (T^s + T^t) \, y_1 y_3 \\
(1+T^{s+t}) + (1 + T^s + T^t + T^{s+t}) \, y_2 + (T^{-t} + T^s) \, y_1 y_2 + (T^s + T^t) \, y_2 y_3 \\
(1+T^{s+t}) \, y_1 + (T^s + T^t) \, y_2 + (1 + T^s + T^t + T^{s+t}) \, y_1 y_2 + (T^s + T^t) \, y_1 y_2 y_3.
\end{gathered}
\]
Thus, $H_*(C(\DD), \partial^\Omega)$ has dimension $8$, supported in gradings $\pm 1/2$, which agrees with the fact that $\HFK(L)$ is two-dimensional, supported in $\delta$ gradings $\pm 1/2$. On the other hand, if the weights are chosen such that $r \ne -t$ or $s \ne -u$, while $r + s + t + u = 0$, it is not hard to show that $d$ is an isomorphism, so the homology vanishes. This is consistent with the fact that by Proposition \ref{prop:unlink}, the twisted knot Floer homology group $\ul\HFK(L,\omega_\r;\FF)$ vanishes since the cohomology class $\omega_\r$ is nonzero in this case.
\end{example}

The preceding example can be generalized to show that the maps that make up $\partial^\Omega$ are almost always isomorphisms, as follows.

\begin{lemma} \label{lemma:isomorphism}
Let $\DD$ be a diagram with $n \ge 3$ crossings for a knot or a nonsplit link, and let $\r$ be the system of weights coming from a generic function $\Omega\co \{1, \dots, n\}$. For any double successor pair $I$, $I''$, the map $d^{I,I''}\co \Lambda^*(\YY_I) \to \Lambda^*(\YY_{I''})$ is an isomorphism.
\end{lemma}

\begin{proof}
Without loss of generality, assume that $\gamma_{I^1} = \gamma_I \cup e_{j_1}$ and $\gamma_{I^2} = \gamma_I \minus e_{j_2}$ and that $\sigma_I$ is the identity permutation. Just as in Example \ref{ex:2compunlink}, the mapping cone of $d^{I,I''}$ can be identified with the complex associated to a two-crossing diagram of the two-component unlink $Q$ with the same $m$ marked points, using the same choice of weights $\r$. By Theorem \ref{thm:main2} and Proposition \ref{prop:unlink}, it suffices to show that the associated cohomology class $\omega_\r$ has nonzero value on a generator of $H_2(S^3 \minus Q;\Z) \cong \Z$.

Suppose, toward a contradiction, that
\begin{equation} \label{eq:componentweight2}
r_1 + \dots + r_a + r_{b+1} + \dots + r_c + r_{d+1} + \dots + r_m  = 0.
\end{equation}
Note that the left-hand side of this equation automatically equals
\[
-r_{a+1} - \dots - r_b - r_{c+1} - \dots -r_d.
\]
Just as in the proof of Lemma \ref{lemma:generic}, \eqref{eq:componentweight2} is a linear relation among $\Omega(1), \dots, \Omega(n)$ with coefficients in $\{-1, 0 ,1\}$, and we must show that at least one of these coefficients is nonzero, which will contradict the genericity of $\Omega$. Because $\DD$ represents a nonsplit link, there is some crossing $c_{j_3}$ in $\DD$ whose trace connects the two components of $Q$. Therefore, the marked points with weights $\pm \Omega(j_3)$ are on different components of $S$. It follows that the sum on the left-hand side of \eqref{eq:componentweight2} includes a non-canceling $\pm \Omega(j_3)$ term.
\end{proof}

As a corollary, we may describe a family of knots whose $\delta$-graded knot Floer homology and reduced Khovanov homology are isomorphic. For a projection $\DD$, let $\Gamma_\DD$ denote the directed graph with vertices corresponding to $\RR(\DD)$ and with an edge from $I$ to $I''$ whenever $I''$ is a double successor of $I$.

\begin{corollary} \label{cor:disjointtrees}
Let $K$ be a knot, and suppose that $K$ admits a projection $\DD$ such that $\Gamma_\DD$ is a disjoint union of trees. Then $\HFK(K)$ and $\Kh(K)$, equipped with their $\delta$ gradings, are isomorphic.
\end{corollary}

\begin{proof}
Say that $\DD$ has $n$ crossings, and put exactly one marked point on each of the $2n$ edges of $\DD$. Choose a generic function $\Omega \co \{1, \dots, n\} \to \Z$ and consider the complex $(C(\DD), \partial^\Omega)$. If $\Gamma_\DD$ is a disjoint union of trees, then we may inductively find bases for the vector spaces $\Lambda^*(\YY_I)$ with respect to which each map $d_{I,I''}$ is represented by the $2^{2n-1} \times 2^{2n-1}$ identity matrix. Thus, $(C(\DD), \partial^\Omega)$ splits as a direct sum of $2^{2n-1}$ copies of $X \otimes_\F \FF$, where $X$ is a complex generated freely over $\F$ by $\RR(\DD)$ in which the differential of $I \in \RR(\DD)$ is equal to the sum of the double successors of $I$. (Although we could define $X$ in this manner for any link projection, in general the differential may not square to zero.)

The same argument can be used to show that Roberts' spanning tree complex $(C'(\DD),\partial')$ is isomorphic to $X \otimes_\F \FF'$, where $\FF'$ is the field of rational functions in multiple indeterminates over which $C'(\DD)$ is defined \cite{RobertsTwisted}. By the universal coefficient theorem, we have
\[
H_*(C^\Omega(\DD), \partial^\Omega) \cong \bigoplus^{2^{2n-1}} H_*(X) \otimes_\F \FF
\quad
\text{and}
\quad
H_*(C'(\DD), \partial') \cong H_*(X) \otimes_\F \FF'.
\]
Since these homology groups are isomorphic to $\HFK(K) \otimes_\F V^{\otimes 2n-1}\otimes_\F \FF$ and $\Kh(K) \otimes_\F \FF'$, respectively, the result follows.
\end{proof}

Via the Gordon--Litherland signature fomula \cite{GordonLitherlandSignature}, Corollary \ref{cor:disjointtrees} can be used to give a new proof of the fact that for an alternating knot $K$, $\Kh(K)$ and $\HFK(K)$ are both thin and supported in $\delta$ grading $-\sigma(K)/2$.

\section{Background on knot Floer homology}
\label{sec:background}
In this section, we review the construction of knot Floer homology with twisted
coefficients and multiple basepoints, and we describe the maps on knot Floer
homology induced by counting pseudo-holomorphic polygons. In Section
\ref{sec:psi}, we describe some additional algebraic structure which comes from
counting disks that pass over basepoints. We shall assume throughout that the
reader has some familiarity with knot Floer homology; for a more basic
treatment, see \cite{OSzKnot, OSzLink} and \cite[Section 8]{OSz3Manifold}.

\subsection{Multiple basepoints and twisted coefficients}  \label{sec:hfktwisted}

Recall that a \emph{multi-pointed} Heegaard diagram is a tuple $\mathcal{H} = (\Sigma,\bm
\alpha, \bm \beta, \O, \X)$, where
\begin{itemize}
\item $\Sigma$ is an Riemann surface of genus $g$,
\item $\bm\alpha = \{\alpha_1,\dots,\alpha_{g+m-1}\}$ and $\bm\beta =
\{\beta_1,\dots,\beta_{g+m-1}\}$ are sets of pairwise disjoint, simple closed
curves on $\Sigma$ which span $g$-dimensional subspaces of
$H_1(\Sigma;\mathbb{Z})$, and
\item $\O = (O_1,\dots,O_m)$ and $\X = (X_1,\dots,X_m)$ are tuples of basepoints
such that every component of $\Sigma \setminus \bm\alpha$ and $\Sigma \setminus
\bm\beta$ contains exactly one point of $\O$ and one of $\X$.
\end{itemize}
The sets $\bm\alpha$ and $\bm\beta$ specify handlebodies $U_\alpha$ and
$U_\beta$ with $\partial U_\alpha = \Sigma = -\partial U_\beta$. Let $Y$ denote
the 3-manifold with Heegaard decomposition $U_\alpha\cup_{\Sigma}U_\beta$.
$\mathcal{H}$ determines an oriented link $L\subset Y$ according to the
following procedure. Fix $m$ disjoint, oriented, embedded arcs in
$\Sigma\setminus \bm\alpha$ from points in $\O$ to points in $\X$, and form
$\xi^\alpha_1,\dots,\xi^\alpha_m$ by pushing their interiors into $U_\alpha$.
Similarly, define pushoffs $\xi^\beta_1,\dots,\xi^\beta_m$ in $U_\beta$ of
oriented arcs in $\Sigma \setminus \bm \beta$ from points in $\X$ to points in
$\O$. $L$ is the union
\[
L = \xi^\alpha_1\cup\dots\cup\xi^\alpha_m
\cup\xi^\beta_1\cup\dots\cup\xi^\beta_m.
\]
The tuple $\X$ also determines an ordered marking $\p = (p_1,\dots,p_m)$ on
$L$.

The pair $\LL=(L,\p)$ is called an \emph{$m$-pointed link}, and we say that $\mathcal{H}$ is a
\emph{compatible} Heegaard diagram for $\LL$. More generally, an $m$-pointed link is an oriented link together with a marking $\p$ such that every component contains some $p_i$. We consider two such links $(L,\p)$ and $(L',\p')$ to be equivalent if there is an orientation-preserving diffeomorphism of $Y$ sending $L$ to $L'$ and $\p$ to $\p'$. A standard Morse-theoretic argument shows that every pointed link arises from a Heegaard diagram as above, and that compatible
Heegaard diagrams for equivalent pointed links can be connected via a sequence
of index one/two (de)stabilizations, and isotopies and handleslides avoiding
$\O\cup\X$.


Following \cite{OSz3Manifold}, we view
\[\Ta=\alpha_1\times\dots\times\alpha_{g+m-1} \ \text{ and } \
\Tb = \beta_1\times\dots\times\beta_{g+m-1}\] as tori in the symmetric product
$\Sym^{g+m-1}(\Sigma)$. For $\x$ and $\y$ in $\Ta\cap\Tb$, we denote by
$\pi_2(\x,\y)$ the set of homotopy classes of Whitney disks from $\x$ to $\y$.
For $\phi\in \pi_2(\x,\y)$ and $a\in \Sigma \setminus (\bm\alpha \cup \bm
\beta)$, let $a(\phi)$ be the algebraic intersection number
\[
\#(\phi\,\cap\,(\{a\}\times \Sym^{g+m-2}(\Sigma))).
\]
Label the regions of $\Sigma \setminus (\bm\alpha \cup \bm \beta)$ by
$D_1,\dots,D_k$ and choose a point $z_i$ in each $D_i$. The \emph{domain} of
$\phi$ is the formal $\mathbb{Z}$-linear combination \[D(\phi) = \sum_{i=1}^k
z_i(\phi)D_i.\] More generally, we refer to any linear combination \[D =
\sum_{i=1}^k a_iD_i\] as a domain, and we define $a(D)$ to be $a_i$ if $a$ and
$z_i$ are in the same component of $\Sigma \setminus (\bm\alpha \cup \bm
\beta)$.

A \emph{periodic domain} is a domain whose boundary is a union of closed curves
in $\bm\alpha$ and $\bm\beta$. Periodic domains form a group
$\Pi_{\alpha\beta}$ under addition. The subgroup $\Pi^0_{\alpha\beta}$ of
$\Pi_{\alpha\beta}$ consisting of periodic domains which avoid $\O\cup\X$ is
isomorphic to $H_2(Y\setminus L;\mathbb{Z})$. The diagram $\mathcal{H}$ is said
to be \emph{admissible} if every nontrivial element of $\Pi^0_{\alpha\beta}$
has both positive and negative coefficients.

To define a system of twisted coefficients, we fix a collection $\mathbb{A}$ of
points in $\Sigma \setminus (\bm\alpha \cup \bm \beta)$ together with a
function $\omega:\mathbb{A}\rightarrow\mathbb{\Z}$, and we let
\begin{equation}
\label{eqn:pairing}\gen{\omega,\phi}  = \sum_{a \in \mathbb{A}}a(\phi)\omega(a)
\end{equation}
for any $\phi\in\pi_2(\x,\y)$. The map $\gen{\omega,\cdot}$ restricts to a
linear functional on $\Pi^0_{\alpha\beta}$ and therefore determines a
cohomology class $[\omega] \in H^2(Y\setminus L;\mathbb{Z})$.

Now, suppose that $\mathcal{H}$ is admissible and let $\MM$ be a module over
$\mathbb{F}[\mathbb{Z}]$. The twisted knot Floer complex with coefficients in
$\MM$ is defined as
\[
\CFKtt(\HH,\omega;\MM) = \F[\Z]\gen{\T_\alpha \cap \T_\beta} \otimes_{\F[\Z]}
\MM,
\]
with differential given by
\[
\partial (\x) = \sum_{\y \in \Ta \cap \Tb}
\sum_{\substack{\phi \in \pi_2(\x,\y) \\ \mu(\phi)=1 \\
O_i(\phi) =X_i(\phi) = 0 \,\,\,\forall i}} \#(\mathcal{M}(\phi)/\mathbb{R})
\cdot T^{\gen{\omega, \phi}} \y.
\]
Here, $\mu(\phi)$ is the Maslov index of $\phi$ and ${\mathcal{M}}(\phi)$ is
the moduli space of pseudo-holomorphic representatives of $\phi$.

Henceforth, we shall assume that $L$ is null-homologous. Define
\[
O(\phi) = O_1(\phi) + \dots + O_{m}(\phi),\quad X(\phi) = X_1(\phi) + \dots +
X_{m}(\phi), \quad P(\phi) = O(\phi) + X(\phi).
\]
If $\x$ represents a torsion $\Spin^c$ structure on $Y$, then it has an
\emph{Alexander} grading $A(\x)\in\mathbb{Z}$ and a \emph{Maslov} grading
$M(\x)\in \mathbb{Q}$. Following \cite{ManolescuOzsvathQA,RasmussenHomologies},
we define the $\delta$-grading of $\x$ to be $\delta(\x) = a(\x)-m(\x)$. If
$\x$ and $\y$ represent the same torsion $\Spin^c$ structure on $Y$, then their
gradings are related as follows,
\begin{align}
\label{eqn:maslovgrading} M(\x)-M(\y) &= \mu(\phi)-2O(\phi) \\
\label{eqn:alexandergrading} A(\x)-A(\y) &= X(\phi)-O(\phi) \\
\label{eqn:grading}\delta(\x)-\delta(\y)
&=P(\phi)-\mu(\phi),
\end{align} for any $\phi \in \pi_2(\x,\y)$.

\begin{remark}\label{rmk:orientationgradings}
Note that the relative $\delta$-grading in \eqref{eqn:grading} does not depend
on which basepoints are in $\OO$ and which are in $\XX$, which is to say, on
the orientation of $L$. In contrast, the relative Maslov and Alexander gradings
and the absolute $\delta$-grading \emph{do} generally depend on the orientation
of $L$.
\end{remark}

\begin{remark} \label{rmk:twistings}
For $\MM=\mathbb{F}[\mathbb{Z}]/(T-1)\cong \F$, the complex
$\ul\cfkt(\mathcal{H},\omega;\MM)$ does not depend on the marking
$(\A,\omega)$. We refer to this as the \emph{untwisted} knot Floer complex,
$\cfkt(\mathcal{H})$.\footnote{When $\HH$ is a grid diagram for a link $L$ in
$S^3$, $\cfkt(\mathcal{H})$ is just the complex $\widetilde{CL}(\HH)$ defined
in \cite{ManolescuOzsvathSzaboThurston}.}
\end{remark}

The following is a straightforward adaptation of \cite[Lemma 2.2]{OSzCube}.

\begin{lemma} \label{cohomologyclass}
For markings $(\mathbb{A},\omega)$ and $(\mathbb{A}',\omega')$ such that
$[\omega] = [\omega']$ in $H^2(Y\setminus L;\mathbb{Z})$, the complexes
$\ul\cfkt(\mathcal{H},\omega;\MM)$ and $\ul\cfkt(\mathcal{H},\omega';\MM)$ are
isomorphic.
\end{lemma}


\begin{proof}[Proof of Lemma \ref{cohomologyclass}]
For each relative $\Spin^c$ structure $\ul\s$ on $Y\minus L$, fix some generator
$\x_{\ul\s}\in\T_\alpha\cap \T_\beta$ which represents $\ul\s$. For any other
generator $\x$ representing $\ul\s$, there exists a Whitney disk
$\phi\in\pi_2(\x_{\ul\s},\x)$ which avoids $\X\cup \O$. Let
\begin{equation}\label{eqn:epsilon}\epsilon_{\ul\s}(\x) = \gen{\omega',\phi} -
\gen{\omega,\phi}.\end{equation} Since $[\omega] = [\omega']$,
$\gen{\omega,D}=\gen{\omega',D}$ for all periodic domains
$D\in\Pi^0_{\alpha,\beta}$, which implies that the quantity in
(\ref{eqn:epsilon}) does not depend on our choice of $\phi$. Finally, let \[f:
\ul\cfkt(\mathcal{H},\omega;\MM)\rightarrow \ul\cfkt(\mathcal{H},\omega';\MM)\]
be the linear map which sends a generator $\x$ representing $\ul\s$ to $f(\x) =
T^{\epsilon_{\ul\s}(\x)}\cdot\x.$ It is easy to check that $f$ is a chain map,
and it is obviously an isomorphism.
\end{proof}

Suppose that $\HH$ and $\HH'$ are compatible Heegaard diagrams for $\LL$, with markings $(\A,\omega)$ and $(\A',\omega')$, respectively. As mentioned above, $\HH$ and $\HH'$ are related by a sequence
of index one/two (de)stabilizations, and isotopies and handleslides avoiding
$\O\cup\X$. These Heegaard moves induce a bijection $\rho$ between periodic domains of $\HH$ and those of $\HH'$ (which restricts to a bijection between periodic domains that avoid $\X\cup\O$). 

\begin{proposition} \label{prop:invariancetwisted}
If $\gen{\omega,P} = \gen{\omega',\rho(P)}$ for
all periodic domains $P$ in $\Pi_{\alpha,\beta}^0$, then the complexes
$\ul\cfkt(\mathcal{H},\omega;\MM)$ and $\ul\cfkt(\mathcal{H'},\omega';\MM)$ are
quasi-isomorphic. Therefore, the homology
\begin{equation*}\label{eqn:homo}
\ul\HFKt(\LL,[\omega];\MM) = H_*(\ul \cfkt(\mathcal{H},\omega;\MM),\partial)
\end{equation*} depends only on
the $m$-pointed link $\LL$ and $[\omega]$. (When each component of $\LL$ has a single basepoint, we denote this group by $\ul\HFK(L,[\omega];\MM)$.)
\end{proposition}

\begin{proof}[Proof of Proposition \ref{prop:invariancetwisted}] It is not always possible to perform the above Heegaard moves while avoiding $\A$ --- an isotopy might get ``stuck" on a point of $\A$ as in Figure \ref{fig:changemarking}(a). Modifying the marking as in Figure \ref{fig:changemarking}(b) does not change the associated cohomology class, but allows one to proceed with the isotopy in the complement of the new marking. In this way, the triple $(\HH',\A',\omega')$ may be obtained from $(\HH,\A,\omega)$ via a combination of marking changes which preserve cohomology class, and Heegaard moves which avoid the basepoints and the markings. These marking changes induce isomorphisms, by Lemma \ref{cohomologyclass}. Moreover, the standard Heegaard Floer arguments \cite{OSz3Manifold} show that these Heegaard moves induce quasi-isomorphisms, and that the chain homotopy type of $\ul\cfkt(\mathcal{H},\omega;\MM)$ is invariant under changes of almost-complex structure.
\end{proof}

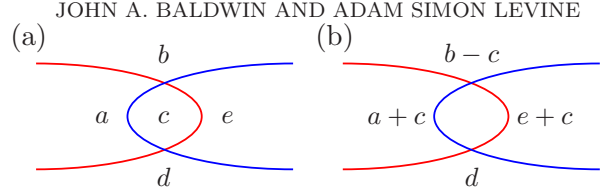
\begin{figure}
\labellist
 \pinlabel (a) at 11 71
 \pinlabel (b) at 145 71
\small
 \pinlabel $a$ at 43 36
 \pinlabel $b$ at 70 64
 \pinlabel $c$ at 70 36
 \pinlabel $d$ at 70 10
 \pinlabel $e$ at 98 36
 \pinlabel $a+c$ at 172 36
 \pinlabel $b-c$ at 205 64
 \pinlabel $d$ at 205 10
 \pinlabel $e+c$ at 237 36
\endlabellist
\begin{center}
\includegraphics[width=7.5cm]{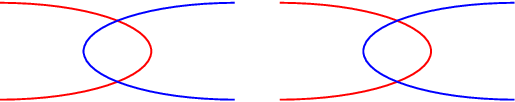}
\caption{\quad We have labeled markings by the values that $\omega$ takes on
them. In (a), the isotopy gets stuck at the point labeled $c$. In (b), we have
removed this point and adjusted the values of $\omega$ on the four nearby
points.} \label{fig:changemarking}
\end{center}
\end{figure}

When $[\omega] = 0$, as for a knot $L\subset S^3$, we may choose $\mathbb{A}$
to be the empty set. Therefore,
\begin{equation} \label{eqn:err}
\ul\HFKt(\LL,0;\MM) \, \cong \,\HFKt(\LL)\otimes_{\mathbb{F}}\MM,
\end{equation}
where $\HFKt(\LL)$ denotes the homology of $\cfkt(\mathcal{H})$. Moreover, it is well-known that
\begin{equation} \label{eqn:err2}
\HFKt(\LL)\,\cong\, \HFK(L) \otimes_{\mathbb{F}} V^{\otimes (m-|L|)},
\end{equation}
where $V$ is a 2-dimensional vector space over $\F$ supported in the $(m,a)$-bigradings $(0,0)$ and $(-1,-1)$ (see, e.g., \cite{ManolescuOzsvathSzaboThurston} for links in $S^3$). Combining the isomorphisms in \eqref{eqn:err} and
\eqref{eqn:err2}, we see that
\begin{equation} \label{eqn:err3}
\ul\HFKt(\LL,0;\MM) \,\cong \,\HFK(L) \otimes_{\mathbb{F}} V^{\otimes (m-|L|)} \otimes_{\mathbb{F}} \MM.
\end{equation}
Furthermore, it is not hard to see that a twisted version holds as well:
\begin{equation} \label{eqn:err4}
\ul\HFKt(\LL, \omega; \MM) \,\cong\, \ul\HFK(L, \omega; \MM) \otimes_\F V^{\otimes (m-|L|)} .
\end{equation}

We shall generally suppress $\MM$ from our notation unless we wish to emphasize
the module we are working over. When we state a result about
$\ul\cfkt(\mathcal{H},\omega)$ or $\ul\HFKt(\LL,[\omega])$, we shall mean that
it holds with coefficients in any $\MM$. Also, we shall often use
$\ul\cfkt(\bm\alpha,\bm\beta)$ to denote $\ul\cfkt(\mathcal{H},\omega)$, as
long as $\Sigma$, $\OO$, $\XX$ and $(\A,\omega)$ are clear from the context.

\subsection{Pseudo-holomorphic polygons} \label{sec:polygon}
A multi-pointed Heegaard \emph{multi-diagram} is a tuple
\[
\mathcal{H} = (\Sigma,\bm\eta^1,\dots,\bm \eta^n, \O,\X)
\] for which each
sub-tuple $(\Sigma,\bm \eta ^i,\bm \eta ^j,\O,\X)$ is a multi-pointed Heegaard
diagram of the sort described in \ts \ref{sec:hfktwisted}. Fix a marking
$(\mathbb{A},\omega)$ on $\mathcal{H}$. For distinct indices $i_1, \dots, i_k$
and intersection points $\x_1 \in \T_{\eta ^{i_1}}\cap\T_{\eta ^{i_2}}, \dots,
\x_{k-1} \in \T_{\eta^{i_{k-1}}}\cap\T_{\eta^{i_k}}$ and $\x_k \in \T_{\eta
^{i_1}}\cap\T_{\eta^{i_k}}$, we denote by $\pi_2(\x_1,\dots,\x_k)$ the set of
homotopy classes of Whitney $k$-gons connecting them. For
$\phi\in\pi_2(\x_{1},\dots,\x_{k})$ and $a\in
\Sigma\setminus(\bm\eta^1\cup\dots\cup\bm\eta^n)$, let $a(\phi)$ denote the
intersection of $\phi$ with $\{a\}\times\Sym^{g+m-2}(\Sigma)$, and define the
pairing $\langle \omega, \phi \rangle$ as in \eqref{eqn:pairing}.

A \emph{multi-periodic domain} is a formal $\Z$-linear combination of the
regions in $\Sigma\setminus(\bm\eta^1\cup\dots\cup\bm\eta^n)$ whose boundary is
a union of curves among the sets $\bm\eta^1,\dots,\bm\eta^n$. Let $\Pi_{\eta^1
\dots \eta^n}$ denote the group of multi-periodic domains, and let
$\Pi^0_{\eta^1 \dots \eta^n}$ denote the subgroup of $\Pi_{\eta^1 \dots
\eta^n}$ consisting of multi-periodic domains that avoid $\O\cup\X$. As before,
we say that $\mathcal{H}$ is admissible if every nontrivial element of
$\Pi^0_{\eta^1 \dots \eta^n}$ has both positive and negative coefficients.

Suppose that $\mathcal{H}$ is admissible, and let
$\ul\cfkt(\bm\eta^{i_s},\bm\eta^{i_t})$ denote the complex associated to
$(\Sigma,\bm\eta^{i_s},\bm\eta^{i_t},\O,\X)$ and $(\mathbb{A},\omega)$. For
$k\geq 3$, we define a map
\[
F_{\eta^{i_1}\dots\eta^{i_k}}: \ul \cfkt(\bm
\eta^{i_1},\bm \eta^{i_2}) \otimes\dots\otimes \ul \cfkt(\bm \eta^{i_{k-1}},\bm
\eta^{i_k}) \rightarrow \ul \cfkt(\bm \eta^{i_1},\bm \eta^{i_k})
\]
by
\begin{equation}\label{eqn:kgon}
F_{\eta^{i_1}\dots\eta^{i_k}}(\x_{1}\otimes\dots\otimes\x_{k-1}) = \sum_{\x_{k}
\in \T_{\eta^{i_1}} \cap \T_{\eta^{i_k}}}\,\,
\sum_{\substack{\phi \in \pi_2(\x_{1},\x_{2},\dots,\x_{k}) \\ \mu(\phi)=3-k \\
O_i(\phi) =X_i(\phi) = 0 \ \forall i}} \#({\mathcal{M}}(\phi)) \cdot
T^{\langle\omega, \phi\rangle} \x_{k}.
\end{equation} Here, ${\mathcal{M}}(\phi)$ is the moduli space pseudo-holomorphic representatives of $\phi$, where the conformal structure on the source is allowed to vary. For a $k$-gon, this set of conformal structures forms an associahedron of dimension $k-3$, so ${\mathcal{M}}(\phi)$ has expected dimension zero when $\mu(\phi) = 3-k$.

These $F_{\eta^{i_1}\dots\eta^{i_k}}$ are chain maps when $k=3$. Counting the
ends of the 1-dimensional moduli spaces $\mathcal{M}(\phi)$, for all Whitney
$k$-gons $\phi$ with $\mu(\phi)=4-k$ and $O_i(\phi) = X_i(\phi)=0$ for all $i$,
one obtains the $\mathcal{A}_{\infty}$ relation
\begin{equation}\label{eqn:Ainfty}
\sum_{1\leq s<t\leq k}
F_{\eta^{i_1}\dots\eta^{i_s}\eta^{i_t}\dots\eta^{i_k}}(\x_{1}\otimes\dots\otimes\x_{s-1}\otimes
F_{\eta^{i_s}\dots\eta^{i_t}}(\x_{s}\otimes\dots\otimes\x_{t-1})\otimes\x_{t}\otimes\dots\otimes\x_{k})
= 0,
\end{equation}
where $F_{\eta^{i_s}\eta^{i_t}}$ is understood to mean the differential on the
complex $\ul\cfkt(\bm\eta^{i_s},\bm\eta^{i_t})$.

\subsection{The basepoint action} \label{sec:psi}
Let $\mathcal{H} = (\Sigma,\bm\alpha,\bm\beta,\O,\X)$ be an admissible
multi-pointed Heegaard diagram with marking $(\mathbb{A},\omega)$. For each $i=1,\dots,m$, let
\[
\Psi_i\co \ul\cfkt(\mathcal{H},\omega) \rightarrow \ul
\cfkt(\mathcal{H},\omega)
\]
be the map given by \begin{equation}\label{eqn:psi}\Psi_i (\x) = \sum_{\y \in
\Ta \cap \Tb}
\sum_{\substack{\phi \in \pi_2(\x,\y) \\ \mu(\phi)=1 \\
O_j(\phi) = 0 \,\,\,\forall j\\
X_j(\phi) = 0 \,\,\, \forall j\neq i\\
X_i(\phi) = 1}} \#(\mathcal{M}(\phi)/\mathbb{R}) \cdot T^{\langle\omega,
\phi\rangle} \y.\end{equation}
Counting the ends of the moduli spaces $\mathcal{M}(\phi)/\mathbb{R}$, for all
Whitney disks $\phi$ satisfying the basepoint conditions in \eqref{eqn:psi} but
with $\mu(\phi)=2$, we find that
\[
\Psi_i\circ\partial+\partial\circ\Psi_i=0.
\]
Therefore, $\Psi_i$ is a chain map and induces a map $\psi_i$ on homology.
Similar degeneration arguments show that $\psi_i^2=0$ and that
$\psi_i\psi_j=\psi_j\psi_i$. Thus, we have an action of the exterior algebra
$\Lambda^*(\F[\Z]\langle\psi_1,\dots,\psi_m\rangle)$ on
$H_*(\ul\cfkt(\mathcal{H},\omega),\partial)$. Moreover, a straightforward
generalization of \cite[Lemma 6.2]{OSz3Manifold} shows that $\psi_i$ does not
depend on our choices of analytic data. Note that $\psi_i$ lowers Alexander and
Maslov gradings by $1$ and therefore preserves the $\delta$-grading.

The following is an immediate analogue of Lemma \ref{cohomologyclass}.

\begin{lemma}
\label{cohomologyclass2} Suppose $(\A,\omega)$ and $(\A',\omega')$ are markings
on $\HH$ such that $\gen{\omega,P} = \gen{\omega',P}$ for \emph{every} periodic
domain $P$ of $\HH$. Then there is an isomorphism from
$\ul\cfkt(\mathcal{H},\omega)$ to $\ul\cfkt(\mathcal{H},\omega')$ which
commutes with the action of
$\Lambda^*(\F[\Z]\langle\psi_1,\dots,\psi_m\rangle)$. \qed
\end{lemma}

These $\psi_i$ interact nicely with the maps defined by counting higher
polygons, as follows. Given an admissible multi-diagram $\mathcal{H} =
(\Sigma,\bm\eta^1,\dots,\bm \eta^n, \O,\X)$, we let
\[
\Psi_i^{\eta^{i_1}\dots\eta^{i_k}}\co \ul \cfkt(\bm \eta^{i_1},\bm \eta^{i_2})
\otimes\dots\otimes \ul \cfkt(\bm \eta^{i_{k-1}},\bm \eta^{i_k}) \rightarrow
\ul \cfkt(\bm \eta^{i_1},\bm \eta^{i_k})
\]
be the map which counts pseudo-holomorphic $k$-gons that pass once over $X_i$ and avoid all other basepoints, in analogy with \eqref{eqn:Ainfty}. When $k=2$,
$\Psi_i^{\eta^{i_1} \eta^{i_2}}$ is just the map on $\CFKtt(\bm\eta^{i_2},
\bm\eta^{i_2})$ defined in \eqref{eqn:psi}. These maps fit into an $\AA_\infty$
relation,
\begin{multline} \label{eqn:Ainfty2}
\sum_{1\leq s<t\leq k}
 F_{\eta^{i_1}\dots\eta^{i_s}\eta^{i_t}\dots\eta^{i_k}}(\x_{1}\otimes\dots\otimes\x_{s-1}\otimes
 \Psi_j^{\eta^{i_s}\dots\eta^{i_t}}(\x_{s}\otimes\dots\otimes\x_{t-1})\otimes\x_{t}\otimes\dots\otimes\x_{k})
 \\
+\sum_{1\leq s<t\leq k}
 \Psi_j^{\eta^{i_1}\dots\eta^{i_s}\eta^{i_t}\dots\eta^{i_k}}(\x_{1}\otimes\dots\otimes\x_{s-1}\otimes
 F_{\eta^{i_s}\dots\eta^{i_t}}(\x_{s}\otimes\dots\otimes\x_{t-1})\otimes\x_{t}\otimes\dots\otimes\x_{k})
 = 0.
\end{multline}
When $k=3$, writing $(\bm\alpha, \bm\beta, \bm\gamma) = (\bm\eta^{i_1},
\bm\eta^{i_2}, \bm\eta^{i_3})$, this becomes
\begin{multline*} F_{\alpha\beta\gamma} (\Psi_j^{\alpha\beta}(\x) \otimes \y) +
F_{\alpha\beta\gamma} (\x \otimes \Psi_j^{\beta\gamma}(\y) ) +
\Psi_j^{\alpha\gamma} ( F_{\alpha\beta\gamma} (\x \otimes \y)) \\
+ \Psi_j^{\alpha\beta\gamma}(\partial_{\alpha\beta}(\x) \otimes \y) +
\Psi_j^{\alpha\beta\gamma}(\x \otimes \partial_{\beta\gamma}(\y)) +
\partial_{\alpha\gamma} (\Psi_j^{\alpha\beta\gamma}(\x \otimes \y)) =0.
\end{multline*}
In particular, if $\y$ is a cycle in $\ul \cfkt(\bm\beta,\bm\gamma)$ and $y$ is
its homology class, then the maps $f_{y}$ and $f_{\psi_j^{\beta\gamma}(y)}$, induced on homology by $F_{\alpha\beta\gamma}(\cdot \otimes \y)$ and
$F_{\alpha\beta\gamma}(\cdot \otimes \Psi_i^{\beta\gamma}(\y))$, satisfy
\begin{equation} \label{eqn:psif}
f_y(\psi^{\alpha\beta}_j(x)) + \psi_j^{\alpha\gamma}(f_y(x)) +
f_{\psi^{\beta\gamma}_j(y)}(x)=0
\end{equation}
for any $x\in H_*(\ul\cfkt(\bm\alpha,\bm\beta),\partial_{\alpha\beta})$.

\begin{proposition}\label{prop:invcofpsi}
Suppose $\mathcal{H}'$ is obtained from $\mathcal{H}$ via an isotopy,
handleslide or index one/two (de)stabilization in the complement of
$\mathbb{A}\cup\O\cup\X$. Then the induced isomorphism
\[
\Phi\co H_*(\ul\cfkt(\mathcal{H},\omega),\partial) \rightarrow
H_*(\ul\cfkt(\mathcal{H}',\omega),\partial)
\]
satisfies $\Phi\circ\psi_i = \psi_i\circ\Phi$.
\end{proposition}

\begin{proof}[Proof of Proposition \ref{prop:invcofpsi}]
The isomorphism on knot Floer homology associated to a handleslide is defined
by counting pseudo-holomorphic triangles. Consider, for example, the set
$\bm\beta' = \{\beta_1',\dots,\beta_{g+m-1}'\}$, where $\beta_1'$ is obtained
by handlesliding $\beta_1$ over some $\beta_i$, and $\beta_j'$ is the image of
$\beta_j$ under a small Hamiltonian isotopy for $j=2,\dots,g+m-1$. Since this
handleslide takes place in the complement of $\mathbb{A}$, there is a unique
top-dimensional generator $\Theta^{\beta\beta'}$ of
$H_*(\ul\cfkt(\bm\beta,\bm\beta'),\partial_{\beta\beta'})$, and the associated
isomorphism $\Phi$ is just the map $f_{\Theta^{\beta\beta'}}$. It is easy to
see that each $X_i$ is in the same region of
$\Sigma\setminus(\bm\beta\cup\bm\beta')$ as some $O_j$. Therefore, the map
$\psi_i^{\beta\beta'}$ is identically zero, and (\ref{eqn:psif}) implies that
\[f_{\Theta^{\beta\beta'}}(\psi^{\alpha\beta}_i(x)) +
\psi_i^{\alpha\beta'}(f_{\Theta^{\beta\beta'}}(x))=0.\] Handleslides among the
$\bm\alpha$ curves are treated in the same manner.

The isomorphism on knot Floer homology associated to an isotopy may also be
defined by counting pseudo-holomorphic triangles \cite{LipshitzCylindrical,
RobertsDouble} (though it was not originally defined in this way). The above
reasoning then proves Proposition \ref{prop:invcofpsi} in this case.

The proof of Proposition \ref{prop:invcofpsi} for index one/two (de)stabilization
is immediate.
\end{proof}

Now, suppose that
$\HH$ and $\HH'$ are compatible diagrams for the pointed link $\LL$, with
markings $(\A,\omega)$ and $(\A',\omega')$, respectively. As before, $\HH$ and
$\HH'$ are related by a sequence of Heegaard moves which avoid the basepoints. Let $\rho$
denote the induced bijection between the periodic domains of $\HH$ and those of
$\HH'$. The combination of Proposition \ref{prop:invcofpsi} and Lemma
\ref{cohomologyclass2} implies the following immediate analogue of Proposition
\ref{prop:invariancetwisted}.

\begin{proposition}
\label{prop:invcofpsi2} If $\gen{\omega,P} = \gen{\omega',\rho(P)}$ for
\emph{every} periodic domain $P$ of $\HH$, then there is a quasi-isomorphism
from $\ul\cfkt(\mathcal{H},\omega)$ to $\ul\cfkt(\mathcal{H}',\omega')$ which
commutes with the action of
$\Lambda^*(\F[\Z]\langle\psi_1,\dots,\psi_m\rangle)$. \qed
\end{proposition}

In particular, the actions of $\psi_1,\dots,\psi_m$ on
$H_*(\ul\cfkt(\mathcal{H},\omega),\partial)$ satisfy the same linear relations as those on $H_*(\ul\cfkt(\mathcal{H'},\omega'),\partial)$.

\section{Unknots and unlinks}
\label{sec:unlinks}
In this section, we prove a few results about the twisted knot Floer homologies of
unknots and unlinks that will be useful later on. We start with a result about gradings. According to Remark \ref{rmk:orientationgradings}, the absolute $\delta$-grading on the chain complex $\ul\cfkt(\HH,\omega)$ for a pointed link $\LL=(L,\p)$ generally depends on the orientation of $L$. The lemma below says that this is not the case if $L$ is an unlink.

\begin{lemma}\label{lem:orientationunlink}
If $L$ is an unlink in $S^3$, then $\ul\cfkt(\HH,\omega)$ has a canonical absolute $\delta$-grading, independent of the orientation of $L$.
\end{lemma}

\begin{proof}[Proof of Lemma \ref{lem:orientationunlink}]
Let $\o$ and $\o'$ be two orientations of $L$, and let $\delta_{\o}$ and $\delta_{\o'}$ denote the corresponding absolute $\delta$-gradings on the untwisted complex $\cfkt(\HH)$. Since any two $k$-component oriented unlinks are isotopic as oriented links, the $\delta$-gradings on $\HFKtil(\LL)$ induced by $\delta_{\o}$ and $\delta_{\o'}$ are the same (this homology is non-trivial). Suppose that $\x$ is a cycle in $\cfkt(\HH)$ which generates the maximal $\delta$-grading of $\HFKtil(\LL)$ with respect to $\delta_{\o}$. Since the relative $\delta$-gradings induced by $\delta_{\o}$ and $\delta_{\o'}$ are the same, $\x$ generates the maximal $\delta$-grading of $\HFKtil(\LL)$ with respect to $\delta_{\o'}$ as well. As this maximal $\delta$-grading is independent of orientation, $\delta_{\o}(\x) = \delta_{\o'}(\x)$, which implies that $\delta_{\o} = \delta_{\o'}$.
\end{proof}

For the proposition below, let $\LL = (L,\p)$ be a pointed unlink in $S^3$ with $k$ components, and
denote the marked points on the $i\Th$ component of $L$ by
$p^i_1,\dots,p^i_{s_i}$, according to its orientation.

\begin{proposition} \label{prop:unlink} 
If $k>1$ and $[\omega]\neq 0$, then $\ul\HFKt(\LL,[\omega];\FF)=0$.
\end{proposition}

\begin{proof}[Proof of Proposition \ref{prop:unlink}]
Figure \ref{fig:unlink} shows an admissible multi-pointed Heegaard diagram
$\mathcal{H} = (S^2,\bm\alpha,\bm\beta,\OO,\XX)$ for $\LL$, with the points of
$\X$ labeled just like those of $\p$. Let $\mathbb{A} = \{a_1,\dots,a_{k-1}\}$
as shown. For each $i=1,\dots,k-1$, there is a unique periodic domain $P_i$ which is
bounded by the curves $\beta^i_{s_i}, \alpha^i_1,\dots,\alpha^i_{s_i}$ and
contains $a_i$, obtained as the difference of the light and dark regions in
Figure \ref{fig:unlink}. These domains correspond to generators of
$H_2(S^3\setminus L;\mathbb{Z})$; thus, we may obtain any cohomology class $[\omega]$ by defining $\omega(a_i)$ to be the evaluation of the desired class on $P_i$. Thus, the above choice of $\A$ suffices. Since $[\omega] \neq 0$, we may assume, without loss of generality, that $\omega(a_1)\neq 0$.

\begin{figure}[!htbp]
 \labellist
 \small
 \pinlabel {{\color{red} $\alpha^i_1$}} at 97 237
 \pinlabel {{\color{blue} $\beta^i_1$}} at 160 237
 \pinlabel {{\color{blue} $\beta^i_{s_i-1}$}} at 285 237
 \pinlabel {{\color{red} $\alpha^i_{s_i}$}} at 347 237
 \pinlabel {{\color{blue} $\beta^i_{s_i}$}} at 373 134
 \pinlabel {{\color{blue} $\beta^k_1$}} at 160 103
 \pinlabel {{\color{blue} $\beta^k_{s_k-1}$}} at 282 103
 \pinlabel {{\color{red} $\alpha^k_1$}} at 347 103
 \pinlabel $a_i$ at 236 145
 \pinlabel $O^i_1$ at 83 193
 \pinlabel $X^i_1$ at 128 193
 \pinlabel $O^i_2$ at 191 193
 \pinlabel $O^i_{s_i}$ at 315 193
 \pinlabel $X^i_{s_i}$ at 359 193
 \pinlabel $O^k_1$ at 83 60
 \pinlabel $X^k_1$ at 128 60
 \pinlabel $O^k_2$ at 191 60
 \pinlabel $O^k_{s_k}$ at 315 60
 \pinlabel $X^k_{s_k}$ at 359 60
 \tiny\hair 2pt
 \pinlabel $X^i_{s_i-1}$ at 254 193
 \pinlabel $X^k_{s_k-1}$ at 254 60
 \pinlabel $e^i_1$ at 128 227
 \pinlabel $d^i_1$ at 128 162
 \pinlabel $e^i_{s_i-1}$ at 256 229
 \pinlabel $d^i_{s_i-1}$ at 256 159
 \pinlabel $e^i_{s_i}$ at 386 228
 \pinlabel $d^i_{s_i}$ at 386 158
 \pinlabel $c^k_1$ at 191 94
 \pinlabel $b^k_1$ at 191 27
 \pinlabel $c^k_{s_k-1}$ at 315 97
 \pinlabel $b^k_{s_k-1}$ at 320 25
 \pinlabel {{\color{red} $\alpha^i_2$}} at 207 234
 \pinlabel {{\color{red} $\alpha^k_2$}} at 207 100
 \pinlabel {{\color{red} $\alpha^k_{s_k-1}$}} at 242 100
\endlabellist
\begin{center}
\includegraphics[width=10cm]{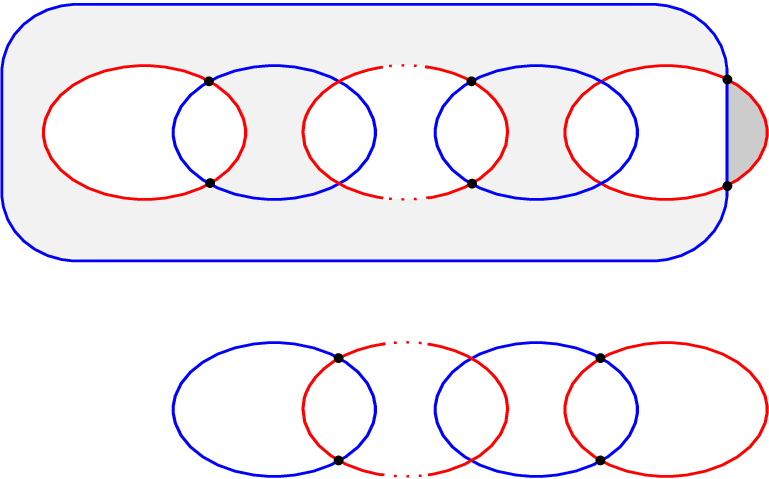}
\caption{\quad A Heegaard diagram for $\LL$. There is a copy of the upper
portion for each $i=1,\dots,k-1$.} \label{fig:unlink}
\end{center}
\end{figure}
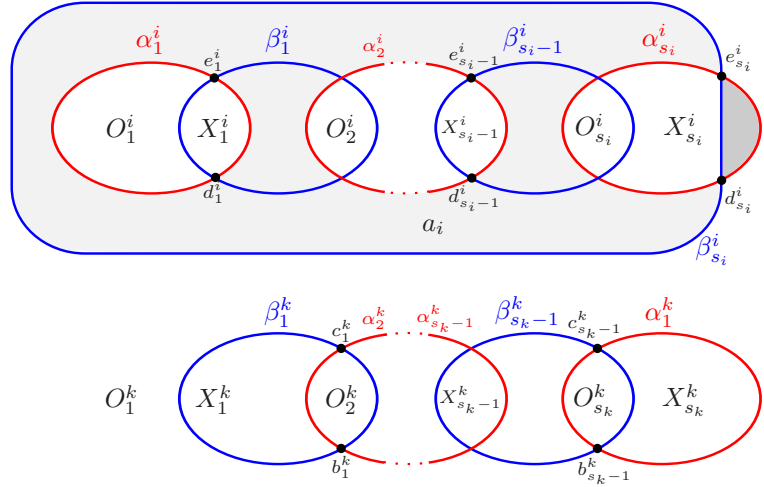

A generator $\x$ of $\ul \cfkt(\mathcal{H},\omega)$ consists of a choice of
$d^i_j$ or $e^i_j$ for each $i=1,\dots,k-1$ and $j = 1,\dots,s_i$ as well as a
choice of $b^k_j$ or $c^k_j$ for each $j=1,\dots,s_k-1$. In particular, the
rank of the \emph{untwisted} complex $\CFKtil(\HH)$ is $2^{s_1+\dots+s_k-1}$
over $\F$, which agrees with the rank of its homology. Therefore, the
pseudo-holomorphic disks which count for the differential on $\ul
\cfkt(\mathcal{H},\omega)$ come in canceling pairs. Their domains are the
heavily-shaded bigons and the lightly-shaded punctured bigons in Figure
\ref{fig:unlink}, with vertices at $d^i_{s_i}$ and $e^i_{s_i}$.

Let $C_e$ denote the subcomplex of $\ul \cfkt(\mathcal{H},\omega;\FF)$ consisting
of intersection points which contain $e^1_{s_1}$, and let $C_d$ be its quotient
complex. Let $\tau: C_d\rightarrow C_e$ be the map which, on generators,
replaces $e^1_{s_1}$ with $d^1_{s_1}$; note that $\tau$ is an isomorphism of vector
spaces. The discussion above implies that $(\ul \cfkt(\mathcal{H},\omega;\FF),
\partial)$ is isomorphic to the mapping cone of
$(1+T^{\omega(a_1)})\cdot\tau$. Since $1+T^{\omega(a_1)} \ne 0$ and $\FF$ is a
field, we have $H_*(\ul \cfkt(\mathcal{H},\omega;\FF),\partial) = 0$.
\end{proof}

Next, we describe the structure of $H_*(\ul\cfkt(\HH,\omega),\partial)$ as a
module over $\Lambda^*(\F[\Z]\langle \psi_1,\dots,\psi_m \rangle)$ for a
particular class of Heegaard diagrams and markings compatible with the unknot.


\begin{proposition} \label{prop:twistedunknot}
Let $\HH = (\Sigma,\bm\alpha, \bm\beta, \O,\X)$ be a Heegaard diagram for an
$m$-pointed unknot in $S^3$, such that $O_i$ and $X_i$ are in the same
component of $\Sigma \minus \bm\alpha$, and $X_i$ and $O_{i+1}$ are in the same
component of $\Sigma \minus \bm\beta$. Let $(\A,\omega)$ be a marking on $\HH$
such that
\begin{inparaenum}
\item all points of $\A$ are contained in a single component of $\Sigma \minus
\bm\alpha$, and
\item for each $i=1, \dots, m$, the component of $\Sigma \minus \bm\beta$ containing $O_i$ contains
a single point $a_i\in\A$ with $\omega(a_i) = r_i$. \end{inparaenum} Let $\YY$
denote the module over $\F[\Z]$ generated by $y_1, \dots, y_m$ modulo the
relation
\begin{equation} \label{eq:twistedunknot}
\sum_{j=1}^m T^{r_1 + \dots + r_j} y_j = 0.
\end{equation}
Then $\HFKtt(\HH,\omega)$ can be identified with $\Lambda^*(\YY)
\otimes_{\F[\Z]} \MM$, such that each map $\psi_i$ is given by multiplication
by $y_i$.
\end{proposition}

(Compare the definition of $\YY_I$ in Section \ref{sec:complex}.)

\begin{proof}[Proof of Proposition \ref{prop:twistedunknot}]
It suffices to take $\MM = \F[\Z]$. Since $\Pi_{\alpha,\beta}$ is generated by
the components of $\Sigma\minus\bm\alpha$ and $\Sigma\minus\bm\beta$ (see
\cite{ManolescuOzsvathQA} or Section \ref{sec:pdomains}), hypotheses (1) and
(2) determine the evaluations of $\omega$ on all periodic domains. By
Proposition \ref{prop:invcofpsi2}, we may assume that $\HH$ and $(\A, \omega)$
are the diagram and marking shown in Figure \ref{fig:unknot}.

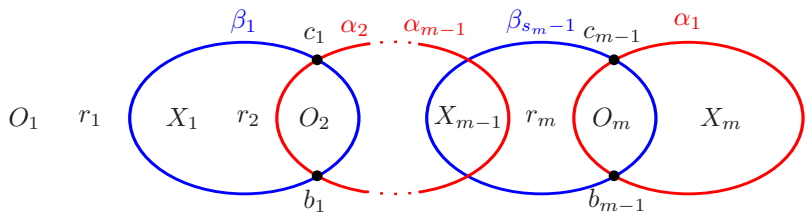
\begin{figure}
 \labellist
 \small
 \pinlabel {{\color{blue} $\beta_1$}} at 103 91
 \pinlabel {{\color{blue} $\beta_{s_m-1}$}} at 229 91
 \pinlabel {{\color{red} $\alpha_1$}} at 291 91
 \pinlabel {{\color{red} $\alpha_2$}} at 150 88
 \pinlabel {{\color{red} $\alpha_{m-1}$}} at 184 88
 \pinlabel $O_1$ at 10 50
 \pinlabel $X_1$ at 77 50
 \pinlabel $O_2$ at 133 50
 \pinlabel $X_{m-1}$ at 198 50
 \pinlabel $O_m$ at 259 50
 \pinlabel $X_m$ at 305 50
 \pinlabel $c_1$ at 133 85
 \pinlabel $b_1$ at 133 15
 \pinlabel $c_{m-1}$ at 259 85
 \pinlabel $b_{m-1}$ at 261 15
 \pinlabel $r_1$ at 38 50
 \pinlabel $r_2$ at 105 50
 \pinlabel $r_m$ at 229 50
\endlabellist
\includegraphics[scale=0.9]{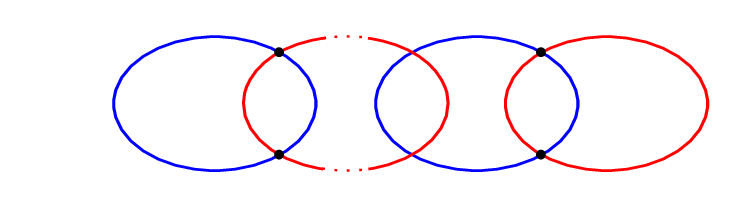}
\caption{Heegaard diagram $\HH$ for the unknot, with twisting as prescribed in
Proposition \ref{prop:twistedunknot}. Points of $\A$ are labeled with their
values of $\omega$.} \label{fig:unknot}
\end{figure}

Generators of the complex $\ul\cfkt(\HH,\omega)$ consist of a choice of $c_j$
or $b_j$ for each $j=1,\dots,m-1$; therefore, $\ul\cfkt(\HH,\omega)$ has rank
$2^{m-1}$ over $\F[\Z]$. It is easy to see that the differential vanishes, so
we may identify $\ul\cfkt(\HH,\omega)$ with its homology. For $j=1,\dots, m-1$,
consider the linear operator $\tau_j$ on $\ul\cfkt(\HH,\omega)$ defined on
generators by
\[
\tau_j(\x) = \begin{cases} \x \setminus \{b_j\}\cup\{c_j\} &
b_j \in \x \\
0 &  b_j \notin \x. \end{cases}
\]
The only domain of $\HH$ that counts for $\psi_1$ is the small bigon containing
$X_1$ with vertices at $b_1$ and $c_1$. For $j=2,\dots,m-1$, the only domains
that count for $\psi_j$ are the two small bigons containing $X_j$ with vertices
at $b_{j-1}$ and $c_{j-1}$, and $b_j$ and $c_j$. Similarly, the only domain
that counts for $\psi_m$ is the small bigon containing $X_m$ with vertices at
$b_{m-1}$ and $c_{m-1}$. Therefore,
\begin{equation}
\begin{aligned}
\label{eq:psis}
\psi_1 (\x) &= T^{r_2} \tau_1 (\x), && \\
\psi_j (\x) &= \tau_{j-1}(\x) + T^{r_{j+1}} \tau_{j}(\x) & (j&=2, \dots, m-1), \\
\psi_{m} (\x) &= \tau_{m-1}(\x), && \\
\end{aligned}
\end{equation}
which implies that
\begin{equation}\label{eq:R=0}
\sum_{j=1}^m T^{r_1 + \dots + r_j} \psi_j (\x)=0.
\end{equation}

Let $\x_0$ denote the generator consisting of all the intersection points
$\{b_j\}$. There is a well-defined linear map
\[
\rho: \Lambda^*(\YY) \to \HFKtt(\HH, \omega)
\]
taking $1$ to $\x_0$ and $y_{i_1} \cdots y_{i_k}$ to $(\psi_{i_1} \circ \dots
\circ \psi_{i_k})(\x_0)$. Moreover, by \eqref{eq:psis}, every element of
$\ul\cfkt(\HH,\omega)$ can be obtained from $\x_0$ by a composition of the
$\psi_i$ maps, so $\rho$ is surjective. As both $\Lambda^*(\YY)$ and
$\ul\cfkt(\HH,\omega)$ are both free $\F[\Z]$-modules of rank $2^{m-1}$, $\rho$
is an isomorphism.
\end{proof}

\section{A cube of resolutions for \texorpdfstring{$\HFKtt$}{HFK}}
\label{sec:ss}
In this section, we show that Manolescu's unoriented skein exact triangle
\cite{ManolescuSkein} holds with twisted coefficients in any
$\mathbb{F}[\Z]$-module $\MM$, and can be iterated in the manner of
Ozsv\'ath-Szab\'o \cite{OSzDouble}.

\subsection{A Heegaard multi-diagram for a link and its resolutions} \label{sec:heegaard}

Fix a connected projection $\DD$ of $L$. Let $c_1,\dots,c_n$ denote the crossings of $\DD$, and let $\p=\{p_1,\dots,p_m\}$ be a set of markings on the edges of $\DD$ so that every edge is marked and $p_1$ is assigned to an outermost edge, as in Section \ref{sec:complex}. This marking specifies an $m$-pointed link $\LL=(L,\p)$. For $I\in
\{0,1,\infty\}^n$, let $\DD_I$ denote the diagram obtained from $\DD$ by taking the $I_j$-resolution of $c_j$, as prescribed in Figure \ref{fig:resolutions}. $\DD_I$ is called a \emph{partial resolution} of $\DD$, and represents an
$m$-pointed link $\LL_I=(L_I,\p)$. In this subsection we construct an admissible multi-pointed Heegaard multi-diagram which encodes all partial resolutions of $\DD$, following \cite{OSzAlternating,ManolescuSkein}.

Let $U_\beta$ denote the closure of a regular neighborhood of $\DD$, and let $U_\alpha = S^3 \minus \operatorname{int} U_\beta$. This determines a genus $(n+1)$ Heegaard splitting $S^3 = U_{\alpha} \cup_{\Sigma} U_{\beta}$, where $\Sigma$ is the oriented boundary of $U_{\alpha}$. The handlebody $U_{\alpha}$ is specified by curves $\alpha_1, \dots, \alpha_{n+1}$ that are the intersections of $\Sigma$ with the bounded regions of $\mathbb{R}^2\minus \DD$.

\begin{figure}
\labellist 
 \pinlabel (a) at 7 198
 \pinlabel (b) at 176 198
 \pinlabel (c) at 342 198
 \small \hair 2pt
 \pinlabel $\beta_j$ at 25 75
 \pinlabel $\gamma_j$ at 75 125 
 \pinlabel $\delta_j$ at 57 19 
 \pinlabel $\beta_j$ at 243 127
 \pinlabel $\gamma_j$ at 194 75
 \pinlabel $\delta_j$ at 264 19
 \pinlabel \rotatebox{50}{$X_{i_1}$} at 117 116
 \pinlabel \rotatebox{-40}{$O_{i_2}$} at 35 116
 \pinlabel \rotatebox{50}{$X_{i_3}$} at 35 34
 \pinlabel \rotatebox{-40}{$O_{i_4}$} at 117 33
 \pinlabel $a_{i_2}$ at 54 96
 \pinlabel $a_{i_4}$ at 96 54
 \pinlabel \rotatebox{50}{$X_{i_1}$} at 285 116
 \pinlabel \rotatebox{-40}{$O_{i_2}$} at 203 116
 \pinlabel \rotatebox{50}{$X_{i_3}$} at 203 34
 \pinlabel \rotatebox{-40}{$O_{i_4}$} at 285 33
 \pinlabel $a_{i_2}$ at 222 96
 \pinlabel $a_{i_4}$ at 264 54
 \pinlabel $p_i$ at 380 178
 \pinlabel $\mu_i$ at 342 79
 \pinlabel $a_i$ at 366 40
 \pinlabel $O_i$ at 366 65
 \pinlabel $X_i$ at 366 79
 \pinlabel $p_{i_2}$ at 63 202 \pinlabel $p_{i_1}$ at 88 202 
 \pinlabel $p_{i_3}$ at 63 154 \pinlabel $p_{i_4}$ at 88 154
 \pinlabel $e_j$ at 103 178 
 \pinlabel $p_{i_2}$ at 232 202 \pinlabel $p_{i_1}$ at 255 202
 \pinlabel $p_{i_3}$ at 232 154 \pinlabel $p_{i_4}$ at 255 154
 \pinlabel $e_j$ at 272 178
\endlabellist
\includegraphics{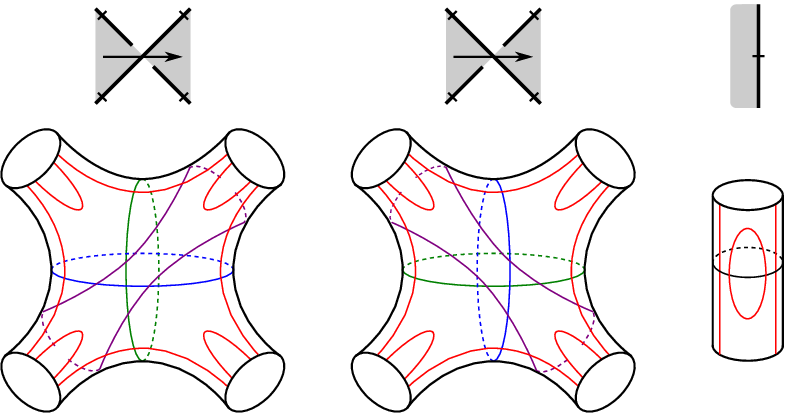}
\caption{(The portions of $\Sigma$ near a crossing $c_j$ (a--b) or a marked point $p_i$ (c). The labeling conventions in (a--b) are the same as in Figure \ref{fig:convention}.} \label{fig:crossing}
\end{figure}
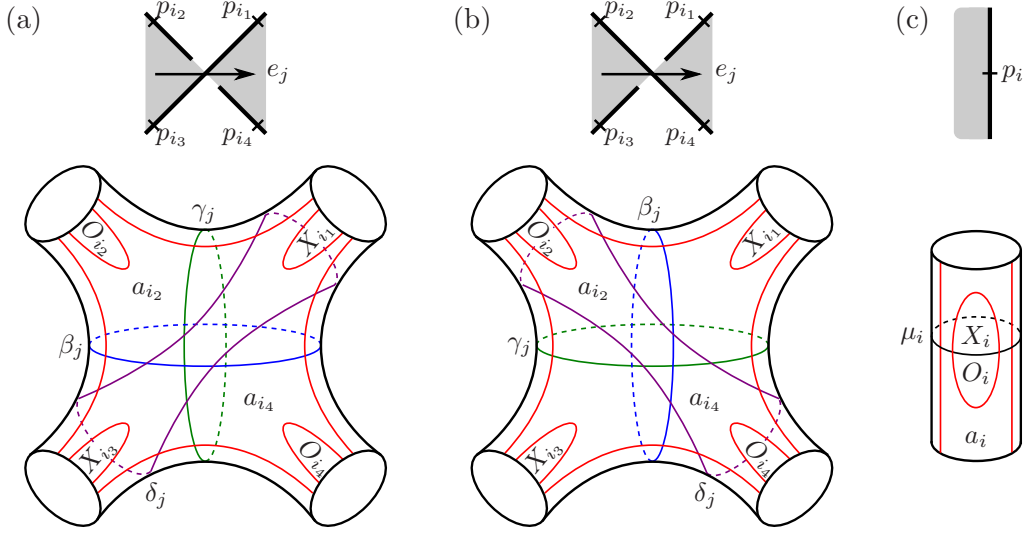

Near each marked point $p_i$, let $\mu_i$ be the boundary of a meridional disk of $U_\beta$. Let $\zeta_i$ be an short arc on the upper half of $\Sigma$ meeting $\mu_i$ once transversally. Orient the edge of $\DD$ containing $p_i$ as the boundary of the black region that it abuts, and orient $\zeta_i$ in the same direction. Let $a_i$ and $X_i$ be the initial and final points of $\zeta_i$, and let $O_i$ be a point on $\zeta_i$ between $a_i$ and $\zeta_i \cap \mu_i$. For $i=2,\dots,m$, let $\alpha_{p_i}$ be the boundary of a disk that contains $X_i$ and $O_i$ but not $a_i$, chosen such that $\alpha_{p_i}$ and $\mu_i$ meet transversally in a pair of points. (See Figure \ref{fig:crossing}(c).) We refer to the configuration $\alpha_{p_i} \cup \mu_i \cup \{O_i, X_i\}$ as a \emph{ladybug}. Set $\O = \{O_1, \dots, O_m\}$, $\X = \{X_1, \dots, \X_m\}$, $\A = \{a_1, \dots, a_m\}$, and $\P = \O \cup \X$.

As shown in Figure \ref{fig:crossing}, the component of $\Sigma \minus (\mu_1 \cup \dots \cup \mu_m)$ corresponding to $c_j$ is a sphere with four punctures. If we view $c_j$ with the incident black regions on the left and right and the white regions on the top and bottom, and the adjacent marked points labeled $p_{i_1}$, $p_{i_2}$, $p_{i_3}$, and $p_{i_4}$ just as in Definition \ref{def:special}, this component contains the basepoints $X_{i_1}$, $O_{i_2}$, $X_{i_3}$, and $O_{i_4}$, as well as the marked points $a_{i_2}$ and $a_{i_4}$. As will be seen below, the positions of $X_{i_1}$, $a_{i_2}$, $X_{i_3}$, and $a_{i_4}$ will motivate the conventions of Definition \ref{def:special}. Let $\beta_j$, $\gamma_j$ and $\delta_j$ be curves on $\Sigma$ as shown in Figure \ref{fig:crossing}(a--b).

For each $I\in\{0,1,\infty\}^n$, let
\[
\bm\eta(I) = \{\eta_{c_1}(I), \dots, \eta_{c_n}(I), \eta_{p_1}(I), \dots, \eta_{p_m}(I)\},
\]
where $\eta_{p_i}(I)$ is a small Hamiltonian translate of $\mu_i$, and $\eta_{c_j}(I)$ is a small Hamiltonian translate of $\beta_j$, $\gamma_j$ or $\delta_j$, according to whether $I_j$ is $0$, $1$ or $\infty$, respectively. We choose these curves so that $\eta_{p_i}(I)$ and $\eta_{p_i}(I')$ (resp.~$\eta_{c_j}(I)$ and $\eta_{c_j}(I')$) meet transversely in exactly two points for each $I \ne I'$, and so that no three curves intersect in the same point. Let
\[
\tilde{\bm\alpha} = \{\alpha_1, \dots, \alpha_{n+1}, \alpha_{p_2}, \dots, \alpha_{p_m}\}.
\]
The Heegaard diagram $\tilde\HH_I = (\Sigma, \tilde{\bm\alpha}, \bm\eta(I), \P)$ then specifies the unoriented $m$-pointed link $\DD_I$. Moreover, for any orientation $\o$ of $\DD_I$, one can partition $\P$ into subsets $\O_{I,\o}$ and $\X_{I,\o}$ of equal size so that $(\Sigma, \tilde{\bm\alpha}, \bm\eta(I), \O_{I,\o}, \X_{I,\o})$ encodes $\DD_I$ with orientation $\o$. In particular, if $I \in \{0,1\}^n$ and $\o$ is the orientation that $\DD_I$ inherits as the boundary of the black regions, then $\O_{I,\o} = \O$ and $\X_{I,\o} = \X$. (On the other hand, if $I \not\in \{0,1\}^\infty$, there is no orientation $\o$ on $\DD_I$ for which this statement holds.)
The multi-diagram
\[
\tilde\HH=(\Sigma, \tilde{\bm\alpha}, \{\bm\eta(I)\}_{I\in\{0,1,\infty\}^n},\P)
\]
thus encodes all unoriented partial resolutions of $\DD$. Note, however, that we cannot partition $\P$ to describe orientations on all the resolutions $\DD_I$ (for $I \in \{0,1,\infty\}^n$) simultaneously. (We do not need to distinguish between $\O$ and $\X$ again until the end of Section \ref{sec:cube}.)

In order to define systems of twisted coefficients, fix a system of weights $\r = (r_1, \dots, r_m)$ as in Definition \ref{def:weights}, which at this point need not be generic. Define $\omega_\r\co \A \to \Z$ by $\omega_\r(a_i) = r_i$.

Note that $\tilde\HH_I$ (and, hence, $\tilde\HH$) is inadmissible when $\DD_I$ has more than one component. Following \cite{ManolescuSkein}, we may achieve admissibility by stretching the tips of the $\tilde{\bm\alpha}$ curves used in
the ladybugs until they reach regions containing $O_1$ or $X_1$. We require that these isotopies avoid the points in $\A\cup\P$. Let $\bm\alpha$ denote the resulting set of curves, and let $\HH_I$ and $\HH$ denote the corresponding
admissible diagrams. We shall henceforth use $\ul\cfkt(\bm\alpha,\bm\eta(I))$ to denote the complex $\ul\cfkt(\Sigma, \bm\alpha, \bm\eta(I), \P, \omega_\r)$, with its relative $\delta$-grading given by (\ref{eqn:grading}), and coefficients in an arbitrary $\mathbb{F}[\Z]$-module $\MM$. For $I\neq I'$, we shall likewise use $\ul\cfkt(\bm\eta(I),\bm\eta(I'))$ to denote the relatively $\delta$-graded complex $\ul\cfkt(\Sigma, \bm\eta(I), \bm\eta(I'), \P, \omega_\r)$. We close this section with a description of the latter.

Each curve in $\bm\eta(I)$ intersects one curve in $\bm\eta(I')$ in exactly two points and is disjoint from all others; therefore, $\abs{\T_{\eta(I)} \cap \T_{\eta(I')}} = 2^{n+m}$. For $I_j=I'_j$, the curves $\eta_{c_j}(I)$ and $\eta_{c_j}(I')$ are related by a small Hamiltonian isotopy, so the two points of $\eta_{c_j}(I)\cap\eta_{c_j}(I')$ differ in their $\delta$-grading contributions by $1$; let $\theta^{I,I'}_{c_j}$ be the point with the smaller contribution. Similarly, let $\theta^{I,I'}_{p_i}$ denote the intersection point of $\eta_{p_i}(I)\cap\eta_{p_i}(I')$ with the smaller $\delta$-grading contribution, for $i=1,\dots,m$. Suppose $I$ and $I'$ differ in $\epsilon(I,I')$ entries. Then there are $2^{\epsilon(I,I')}$ generators in $\T_{\eta(I)}\cap\T_{\eta(I')}$ that use all of the points
$\theta^{I,I'}_{c_j}$ and $\theta^{I,I'}_{p_i}$; we denote these generators by $\Theta^{I,I'}_1, \dots, \Theta^{I,I'}_{2^{\epsilon(I,I')}}$, indexed arbitrarily.

\begin{lemma} \label{lem:isom}
There is an isomorphism of relatively graded $\F[\Z]$-modules,
\begin{equation*} \label{eqn:isom}
\ul \cfkt(\bm\eta(I), \bm\eta(I')) \cong (H^1(S^1;\mathbb{F}))^{\otimes
(m+n-\epsilon(I,I'))} \otimes_\F V^{\otimes \epsilon(I,I')} \otimes_\F \MM,
\end{equation*}
and the summand of $\CFKtt(\bm\eta(I), \bm\eta(I'))$ in the minimal $\delta$-grading is generated by the points $\Theta^{I,I'}_1, \dots, \Theta^{I,I'}_{2^{\epsilon(I,I')}}$. Moreover, the differential $\partial$ on $\ul \cfkt(\bm\eta(I), \bm\eta(I'))$ is zero.
\end{lemma}

\begin{proof}[Proof of Lemma \ref{lem:isom}]
It suffices to take $\MM = \F[\Z]$. Both modules above are free of rank $2^{n+m}$; we must simply show that the generators $\Theta^{I,I'}_1, \dots, \Theta^{I,I'}_{2^{\epsilon(I,I')}}$ have the same $\delta$-gradings. For this,
suppose $\Theta^{I,I'}_a$ and $\Theta^{I,I'}_b$ differ near a single crossing $c_j$. As in \cite[Lemma 11] {ManolescuOzsvathQA}, there is class $\phi\in\pi_2(\Theta^{I,I'}_a,\Theta^{I,I'}_b)$ whose domain is an annulus
containing a single point of $\P$, with one boundary component comprised of segments of the curves $\eta_{c_j}(I)$ and $\eta_{c_j}(I')$, and the other boundary component equal to $\eta_{p_i}(I)$ (or $\eta_{p_i}(I')$) for some $p_i$ near $c_j$. By Lipshitz's formula for the Maslov index \cite{LipshitzCylindrical}, $\mu(\phi)=1$, so
\[
\delta(\Theta^{I,I'}_a) - \delta(\Theta^{I,I'}_b) = P(\phi)-\mu(\phi)=0.
\]
Finally, the only regions of $\Sigma \minus (\bm\eta(I)\cup\bm\eta(I'))$ that do not contain basepoints are thin bigons bounded by the curves $\eta_{c_j}(I)$ and $\eta_{c_j}(I')$ with $I_j=I'_j$ or the curves $\eta_{p_i}(I)$ and $\eta_{p_i}(I')$. These bigons come in pairs, so the differential $\partial$ is zero.
\end{proof}

\subsection{Periodic domains}
\label{sec:pdomains}

In this subsection, we describe the multi-periodic domains of $\HH$, generalizing the results of Manolescu--Ozsv\'ath \cite[Section 3.1]{ManolescuOzsvathQA}.

Let $\Pi_{\alpha}$ and $\Pi_{\eta(I)}$ denote the groups of $\mathbb{Z}$-linear combinations of the components of $\Sigma\minus\bm\alpha$ and $\Sigma \minus \bm\eta(I)$, respectively. Since $(\Sigma, \bm\alpha, \bm\eta(I))$ is a
Heegaard diagram for $S^3$, the curves in $\bm\alpha$ and $\bm\eta(I)$ span $H_1(\Sigma;\Z)$. Therefore,
\begin{equation} \label{eq:PialphaetaI}
\Pi_{\alpha,\eta(I)} = \Pi_{\alpha} + \Pi_{\eta(I)},
\end{equation}
by \cite[Corollary 7]{ManolescuOzsvathQA}; that is, any periodic domain of
$(\Sigma, \bm\alpha, \bm\eta(I))$ is a sum of components of $\Sigma\minus
\bm\alpha$ with components of $\Sigma\minus \bm\eta(I)$. Note that the latter
are either annuli or pairs of pants.

Now, consider distinct tuples $I, I'$. For $i=1,\dots, m$, there is a
periodic domain $D_{p_i}^{\eta(I), \eta(I')}$ with boundary $\eta_{p_i}(I) -
\eta_{p_i}(I')$, formed as the sum of two thin bigons with opposite signs.
Likewise, for $I_j=I'_j$, there is a periodic domain
$D_{c_{j}}^{\eta(I),\eta(I')}$ with boundary $\eta_{c_{j}}(I) -
\eta_{c_{j}}(I')$. The lemma below is an analogue of \cite[Lemma
9]{ManolescuOzsvathQA}.

\begin{lemma} \label{lem:PietaIetaI'}
The group $\Pi_{\eta(I), \eta(I')}$ is spanned by $\Pi_{\eta(I)}$,
$\Pi_{\eta(I')}$, and periodic domains of the forms $D_{c_{j}}^{\eta(I),
\eta(I')}$ and $D_{p_i}^{\eta(I), \eta(I')}$.
\end{lemma}

\begin{proof}[Proof of Lemma \ref{lem:PietaIetaI'}]
Let $D$ be a domain in $\Pi_{\eta(I), \eta(I')}$ and suppose that, for some
$I_j\neq I_j'$, $\eta_{c_j}(I)$ appears with non-zero multiplicity in the
boundary of $D$. There is a pair of pants in $\Pi_{\eta(I)}$ bounded by
$\eta_{c_j}(I)$ and two curves, $\eta_{p_i}(I)$ and $\eta_{p_{i'}}(I)$. Adding
some multiple of this pair of pants, we obtain a domain whose boundary does not
contain any multiple of $\eta_{c_j}(I)$. Iterating this sort of procedure, we
can write $D$ as the sum of domains in $\Pi_{\eta(I)}$ and $\Pi_{\eta(I')}$
with a domain $D'$ whose boundary consists of curves of the forms
$\eta_{p_i}(I)$, $\eta_{p_i}(I')$ and $\eta_{c_{j'}}(I)$, $\eta_{c_{j'}}(I')$
for $I_{j'}=I'_{j'}$. We may then write $D'$ as the sum of a domain in
$\Pi_{\eta(I)}$ with domains of the forms $D_{p_i}^{\eta(I), \eta(I')}$ and
$D_{c_{j'}}^{\eta(I), \eta(I')}$. This proves Lemma \ref{lem:PietaIetaI'}.
\end{proof}

The following result generalizes \cite[Lemma 10]{ManolescuOzsvathQA}.

\begin{lemma} \label{lem:doublyperiodic}
Suppose $I^0, \dots, I^k \in \{0,1,\infty\}^n$ is a sequence of distinct
tuples, and $k\geq 1$. Then \[ \Pi_{{\alpha}, \eta(I^0), \dots, \eta(I^k)} =
\Pi_{{\alpha}} + \Pi_{\eta(I^0), \eta(I^1)} + \dots + \Pi_{\eta(I^{k-1}),
\eta(I^k)}.
\]
\end{lemma}

\begin{proof}[Proof of Lemma \ref{lem:doublyperiodic}]
We claim that
\[
\Pi_{{\alpha}, \eta(I^0), \dots, \eta(I^k)} = \Pi_{{\alpha}, \eta(I_0), \dots,
\eta(I^{k-1})} + \Pi_{\eta(I^{k-1}), \eta(I^k)}
\] for $k\geq 1$. This claim, together with \eqref{eq:PialphaetaI}, implies Lemma \ref{lem:doublyperiodic} by induction.
Let $D$ be a domain in $\Pi_{{\alpha}, \eta(I^0), \dots, \eta(I^k)}$ and
suppose that, for some $I^k_j\neq I^{k-1}_j$, the curve $\eta_{c_{j}}(I^k)$
appears with nonzero multiplicity in the boundary of $D$. As above, there is a
pair of pants in $\Pi_{\eta(I^k)}$ bounded by $\eta_{c_{j}}(I^k)$ and two
curves, $\eta_{c_{p_i}}(I^k)$ and $\eta_{c_{p_{i'}}}(I^k)$. Adding some
multiple of this pair of pants, we obtain a domain whose boundary does not
contain any multiple of $\eta_{c_j}(I^k)$. Iterating this procedure, and adding
domains of the forms $D_{p_i}^{\eta(I^{k-1}), \eta(I^k)}$ and
$D_{c_{j'}}^{\eta(I^{k-1}), \eta(I^k)},$ we obtain a domain in $\Pi_{{\alpha},
\eta(I_0), \dots, \eta(I^{k-1})}$. Reversing this process proves the claim.
\end{proof}

We shall use the following proposition in many places throughout this paper;
compare with \cite[Lemma 11]{ManolescuOzsvathQA}.

\begin{proposition} \label{prop:mu=P}
Suppose $\phi$ and $\phi'$ are two Whitney polygons for which
$D(\phi)-D(\phi')$ is a multi-periodic domain of $\HH$. Then
\[
P(\phi)-\mu(\phi) = P(\phi')-\mu(\phi')
\]
where $P(\phi)$ denotes the total multiplicity of $\phi$ at all the basepoints.
\end{proposition}

\begin{proof}[Proof of Proposition \ref{prop:mu=P}]
By Lemmas \ref{lem:PietaIetaI'} and \ref{lem:doublyperiodic}, the difference
$D(\phi)-D(\phi')$ is a linear combination of components of
$\Sigma\minus\bm\alpha$, components of the complements
$\Sigma\minus\bm\eta(I)$, and domains of the forms $D_{c_{j}}^{\eta(I),
\eta(I')}$ and $D_{p_i}^{\eta(I), \eta(I')}$. It is easy to verify that these
domains all satisfy $P= \mu$, exactly as in the proof of \cite[Lemma
11]{ManolescuOzsvathQA}. Proposition \ref{prop:mu=P} then follows from the
additivity of $P$ and $\mu$.
\end{proof}

Next, we describe the periodic domains of $\HH_I$ that avoid the basepoints in $\P$. 
Let $S_I^1,\dots,S_I^{k_I}$ denote the components of $L_I$, labeled so that $p_1$ lies on $S_I^{k_I}$. Then $H_2(S^3 \minus L_I;\Z)$ is freely generated by the homology classes of tori $T_I^1, \dots, T_I^{k_I-1}$ obtained as the boundaries of regular neighborhoods of $S_I^1, \dots, S_I^{k_I-1}$. These tori correspond to positive periodic domains $\tilde P_I^1, \dots, \tilde P_I^{k_I}$ in $\Pi^0_{\tilde\alpha,\eta(I)}$, where the boundary of $\tilde P_I^{\ell}$ consists of \begin{inparaenum}\item the $\tilde{\bm\alpha}$ circles of the ladybugs associated to the points of $\p$ on
$\DD_I^{\ell}$, and \item a copy of $\eta_{c_j}(I)$ for every crossing $c_j$ such that $S_I^{\ell}$ enters and leaves a neighborhood of $c_j$ exactly once. \end{inparaenum} The torus $T_I^{\ell}$ can then be recovered by capping off the boundary components of $\tilde P_I^{\ell}$ with disks. Finally, let $P_I^{\ell}$ be the domain in $\Pi^0_{\alpha,\eta(I)}$ corresponding to $\tilde P_I^{\ell}$; although $P_I^{\ell}$ has both positive and negative multiplicities in general, its boundary multiplicities and its multiplicities at points of $\A\cup \P$ agree with those of $\tilde P_I^{\ell}$.

Let $[\omega_\r]_I$ denote the element of $H^2(S^3 \minus L_I)$ associated to the marking $(\A,\omega_{\Omega})$. The previous paragraph implies that the evaluation of $[\omega_\r]_I$ on $[T_I^{\ell}]$ is equal to the sum of the weights of the marked points on $S_I^\ell$. The following proposition then follows from the genericity of $\r$ together with Lemma \ref{lemma:generic}, Proposition \ref{prop:unlink}, and equation \eqref{eqn:err3}.

\begin{proposition} \label{prop:generic}
\begin{enumerate}
\item For any $I \in \{0,1\}$ for which $\DD_I$ is disconnected, the cohomology class $[\omega_\r]_I$ is nonzero, so the complex $\CFKtt(\bm\alpha, \bm\eta(I))$ is acyclic.

\item Let $I^\infty = (\infty, \dots, \infty)$, so that $\DD_{I^\infty} = \DD$. If $\r = \r_\Omega$ for some function $\Omega\co \{1, \dots, n\} \to \Z$, then the cohomology class $[\omega_\r]_{I^\infty}$ is zero, so
\[
\CFKtt(\bm\alpha, \bm\eta(I^\infty)) \cong \CFKtil( \bm\alpha, \bm\eta(I^\infty)) \otimes_{\F} \MM.
\]
\end{enumerate}
\end{proposition}

\subsection{Construction of the cube of resolutions}\label{sec:cube}
For distinct tuples $I,I'\in \{0,1,\infty\}^n,$ we write $I<I'$ if $I_j\leq
I'_j$ for $j=1,\dots,n$. If $I'$ is obtained from $I$ by changing a single
entry from $0$ to $1$, from $1$ to $\infty$, or from $\infty$ to $0$, we say
that $I'$ is a \emph{cyclic successor} of $I$. In the first two cases, $I'$ is
called an \emph{immediate successor} of $I$. A \emph{successor sequence} (resp.
\emph{cyclic successor sequence}) is a sequence of tuples $I^0,\dots,I^k$ such
that each $I^j$ is an immediate (resp. cyclic) successor of $I^{j-1}$. For any
cyclic successor sequence $I^0, \dots, I^k$, let
\[
f_{I^0, \dots, I^k}\co \CFKtt(\bm\alpha, \bm\eta(I^0)) \to \CFKtt(\bm\alpha,
\bm\eta(I^k))
\]
be the map defined by
\begin{equation*} \label{eq:fI0Ik}
f_{I^0, \dots, I^k}(\x) = F_{\alpha, \eta(I^0), \dots, \eta(I^k)} \left(\x
\otimes \left(\Theta^{I^0,I^1}_1 + \Theta^{I^0,I^1}_2\right) \otimes \dots
\otimes \left(\Theta^{I^{k-1},I^k}_1 + \Theta^{I^{k-1},I^k}_2\right)\right).
\end{equation*}
We shall eventually incorporate these maps into a cube of resolutions complex
which is quasi-isomorphic to $\CFKtt(\bm\alpha,\bm\eta(I^{\infty})).$ First, we
prove an analogue of Manolescu's unoriented skein exact triangle for
coefficients in an arbitrary $\F[Z]$-module $\MM$.

\begin{theorem} \label{thm:manolescu}
Suppose $I^0,I^1,I^2$ is a cyclic successor sequence of tuples in
$\{0,1,\infty\}^n$ which differ in only one coordinate. Then, the triangle
\[
\xymatrix@!0@C=6.1pc@R=5pc{
\HFKtt(\LL_{I^0},[\omega_{\Omega}]_{I^0}) \ar[rr]^{(f_{I^0,I^1})_*} && \HFKtt(\LL_{I^1},[\omega_{\Omega}]_{I^1}) \ar[dl]^{(f_{I^1,I^2})_*} \\
& \HFKtt(\LL_{I^2},[\omega_{\Omega}]_{I^2}) \ar[ul]^{(f_{I^2,I^0})_*} }
\]
is exact.
\end{theorem}

As in \cite{ManolescuSkein, OSzDouble}, Theorem \ref{thm:manolescu} follows
immediately from the proposition below.

\begin{proposition}
\label{prop:homalgebra} Suppose $I^0,I^1,I^2$ is a cyclic successor sequence of
tuples in $\{0,1,\infty\}^n$ which differ in only one coordinate. Then,
\begin{enumerate}
\item the composite $f_{I^{0},I^{1}}\circ f_{I^{2},I^{0}}$ is chain homotopic to zero, \[f_{I^{0},I^{1}}\circ f_{I^{2},I^{0}} = \partial\circ f_{I^{2},I^{0},I^{1}} + f_{I^{2},I^{0},I^{1}}\circ\partial;\]
\item the map \[f_{I^{0},I^{1},I^{2}}\circ f_{I^{2},I^{0}}+ f_{I^{1},I^{2}}\circ f_{I^{2},I^{0},I^{1}}:\ul\cfkt(\bm\alpha,\bm\eta(I^{2}))\rightarrow \ul\cfkt(\bm\alpha,\bm\eta(I^{2}))\] is a quasi-isomorphism.
\end{enumerate}
\end{proposition}

\begin{proof}[Proof of Proposition \ref{prop:homalgebra}]
This is a straightforward adaptation of Manolescu's proofs of \cite[Lemmas 6 and 7]{ManolescuSkein}. Simply note that the relevant polygons in Manolescu's proofs avoid the markings in $\A$ since every such marking lies in the same component of $\Sigma\minus(\bm\eta(I^0)\cup\bm\eta(I^1)\cup\bm\eta(I^2))$ as a basepoint.
\end{proof}

For tuples $I<I'$ in $\{0,1,\infty\}^n$, let
\[D_{I,I'}:\ul\cfkt(\bm\alpha,\bm\eta(I))\rightarrow
\ul\cfkt(\bm\alpha,\bm\eta(I'))\] denote the sum, over all successor sequences
$I=I^0<\dots<I^k=I'$, of the maps $f_{I^0,\dots,I^k}$, and let $D_{I,I}$ denote
the differential $\partial$ on $\ul\cfkt(\bm\alpha,\bm\eta(I))$. For $S
\subset \{0,1,\infty\}^n$, let
\[
X(S) = \bigoplus_{I \in S} \ul\CFKtil(\bm\alpha, \bm\eta(I)),
\]
and set $X = X(\{0,1,\infty\}^n)$. We define a map $D:X\rightarrow X$ by \[D =
\bigoplus_{I\leq I'}D_{I,I'}.\] Below, we show that $D$ is a differential. As a
warmup, we prove the following.

\begin{lemma} \label{lem:triangles}
Suppose $I^0,I^1,I^2$ is a cyclic successor sequence of tuples in
$\{0,1,\infty\}^n$. If these tuples differ in only one coordinate, then
\begin{equation*}\label{eqn:Fthetaone}F_{\eta(I^0),\eta(I^1),\eta(I^2)}\left( (\Theta^{I^0,I^1}_1 + \Theta^{I^0,I^1}_2)
\otimes (\Theta^{I^1,I^2}_1 + \Theta^{I^1,I^2}_2) \right) = 0.\end{equation*}
Otherwise,
\begin{equation*}\label{eqn:Fthetatwo}F_{\eta(I^0),\eta(I^1),\eta(I^2)}\left( (\Theta^{I^0,I^1}_1 + \Theta^{I^0,I^1}_2)
\otimes (\Theta^{I^1,I^2}_1 + \Theta^{I^1,I^2}_2) \right)=
\Theta^{I^0,I^2}_1+\Theta^{I^0,I^2}_2+\Theta^{I^0,I^2}_3+\Theta^{I^0,I^2}_4.\end{equation*}
\end{lemma}

\begin{proof}[Proof of Lemma \ref{lem:triangles}]
Let $\z\in\T_{\eta(I^0)}\cap\T_{\eta(I^2)}$. Each of the four tubular regions
of $\Sigma\minus(\bm\eta(I^0)\cup\bm\eta(I^1)\cup\bm\eta(I^2))$ in the
neighborhood of a crossing contains a basepoint, as does the tubular region on
each side of a ladybug. Therefore, the domain of any Whitney triangle
$\psi\in\pi_2(\Theta^{I^0,I^1}_r,\Theta^{I^1,I^2}_s,\z)$ which avoids $\P$ is a
union of small triangles. Suppose $I^0,$ $I^1$ and $I^2$ differ only in their
$j\Th$ coordinates. Near $p_i$ and $c_{j'}$ for $j'\neq j$, these triangles
look like those shaded in Figure \ref{fig:smalltriangles}(a) and (b). Near
$c_j$, the domain of $\psi$ looks like one of the four triangles shaded in (d).
Thus, $\z$ is of the form $\Theta^{I^0,I^2}_{\kappa(a,b)}$, for some $2:1$ map
\[\kappa:\{1,2\}\times\{1,2\}\rightarrow \{1,2\}.\] Moreover, $\mu(\psi)=0$ and
$\psi$ has a unique holomorphic representative. The first statement of Lemma
\ref{lem:triangles} follows immediately.

Now, suppose $I^0,I^1$ and $I^1,I^2$ differ in their $j_1\Th$ and $j_2\Th$
coordinates, respectively. Near $p_i$ and $c_{j'}$ for $j'\neq j_1,j_2$, the
domain of $\psi$ looks like the shaded triangles in (a) and (b). Near
$c_{j_1}$, the domain of $\psi$ is a small triangle with vertices at
intersection points between the curves $\eta_{c_{j_1}}(I^0),$
$\eta_{c_{j_1}}(I^1)$ and $\eta_{c_{j_1}}(I^2)$. Figure
\ref{fig:smalltriangles}(c) shows a picture of this triangle when
$\eta_{c_{j_1}}(I^0)$ is isotopic to $\beta_{j_1}$ and $\eta_{c_{j_1}}(I^1),
\eta_{c_{j_1}}(I^2)$ are isotopic to $\gamma_{j_1}$. The same reasoning applies
near $c_{j_2}$. Therefore, $\z$ is of the form $\Theta^{I^0,I^2}_{\nu(a,b)}$
for some $1:1$ map
\[\nu:\{1,2\}\times\{1,2\}\rightarrow \{1,2,3,4\}.\] As above, $\mu(\psi)=0$
and $\psi$ has a unique holomorphic representative. The second statement of
Lemma \ref{lem:triangles} follows immediately. \end{proof}

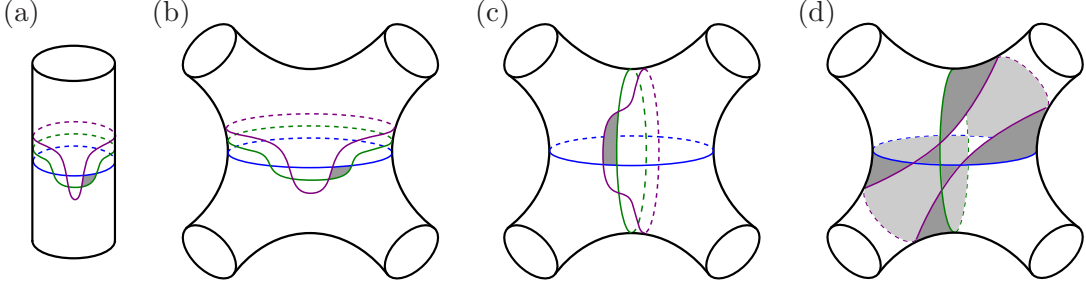
\begin{figure}
\labellist \pinlabel (a) at 4 110 \pinlabel (b) at 65 110 \pinlabel (c) at 196
110 \pinlabel (d) at 328 110
\endlabellist
 \includegraphics[width=14cm]{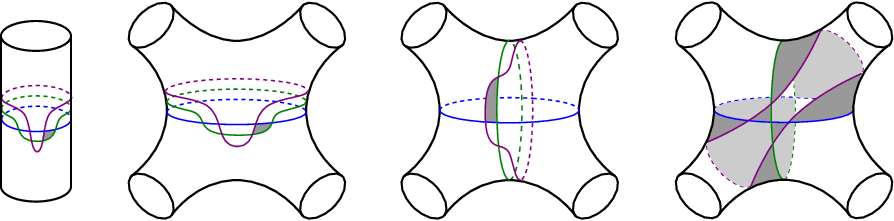}
\caption{Some of the possible
$(\bm\eta(I^0),\bm\eta(I^1),\bm\eta(I^2))$-triangles.}
\label{fig:smalltriangles}
\end{figure}

\begin{proposition} \label{prop:d^2=0}
For tuples $I < I'$ in $\{0,1,\infty\}^n$,
\begin{equation}\label{eqn:D^2=0}
\sum_{\substack{I = I^0 < \dots < I^k = I' \\ \text{succ. seq.}}} F_{\eta(I^0),
\dots, \eta(I^k)}\left(( \Theta_1^{I^0,I^1} + \Theta_2^{I^0,I^1}) \otimes \dots
\otimes (\Theta_1^{I^{k-1},I^k} + \Theta_2^{I^{k-1},I^k})\right) = 0,
\end{equation}
It then follows from the $\mathcal{A}_{\infty}$ relation (\ref{eqn:Ainfty})
that $D^2=0$.
\end{proposition}

\begin{proof}[Proof of Proposition \ref{prop:d^2=0}]
For $k=1$, this is just the statement that $\Theta_1^{I^0,I^1} +
\Theta_2^{I^0,I^1}$ is a cycle in $\ul\cfkt(\bm\eta(I^0),\bm\eta(I^1))$.

Suppose $k=2$. If $I=I^0$ and $I'=I^2$ differ in only one coordinate, then the
proposition follows from Lemma \ref{lem:triangles}. Otherwise, there are
exactly two tuples, $I^1$ and $J^1$, with $I^0<I^1<I^2$ and $I^0<J^1<I^2$. By
Lemma \ref{lem:triangles}, the contributions of these two successor sequences
to the sum (\ref{eqn:D^2=0}) cancel.

Now, suppose $k>2$. For any $a_1,\dots,a_k \in \{1,2\}$ and $b \in \{1, \dots,
2^{\epsilon(I^0, I^k})\}$, there exists a class
\[
\psi \in \pi_2(\Theta^{I^0,I^1}_{a_1}, \dots, \Theta^{I^{k-1},I^k}_{a_k}, \Theta^{I^0,I^k}_{b})
\]
with $P(\psi)-\mu(\psi) = 0$, gotten by concatenating the Whitney triangles
described in the proof of Lemma \ref{lem:triangles} with the Whitney disks in
$\pi_2(\Theta^{I^i,I^j}_p,\Theta^{I^i,I^j}_q)$ described in the proof of Lemma
\ref{lem:isom}. Suppose, for a contradiction, that the coefficient of some
$\Theta\in \T_{\eta(I^0)}\cap\T_{\eta(I^k)}$ in the sum (\ref{eqn:D^2=0}) is
nonzero. Then there is a Whitney $(k+1)$-gon
\[
\psi' \in \pi_2(\Theta^{I^0,I^1}_{a_1}, \dots, \Theta^{I^{k-1},I^k}_{a_k}, \Theta)
\]
with $P(\psi')-\mu(\psi') = k-2.$ Since $\Theta^{I^0,I^k}_{b}$ has the minimal
$\delta$-grading among all generators of $\ul\cfkt(\bm\eta(I^0),\bm\eta(I^k))$,
there is a class $\phi\in\pi_2(\Theta^{I^0,I^k}_{b}, \Theta)$ with
$P(\phi)-\mu(\phi)\leq 0$. Then $P(\psi*\phi)-\mu(\psi*\phi) \leq 0$ as well.
On the other hand, $D(\psi*\phi)-D(\psi')$ is a multi-periodic domain, so
Proposition \ref{prop:mu=P} implies that $P(\psi*\phi)-\mu(\psi*\phi) =
P(\psi')-\mu(\psi')$, a contradiction.
\end{proof}

The main theorem of this section is as follows.

\begin{theorem}
\label{thm:cubeofres} The complex $(X(\{0,1\}^{n}),D)$ is quasi-isomorphic to
$(\ul\cfkt(\bm\alpha,\bm\eta(I^\infty),\partial).$
\end{theorem}

This theorem follows rather quickly from the lemma below.

\begin{lemma}
\label{lem:acyclic} For $0\leq k \leq n$, consider the complex
$X(\{0,1\}^{n-k-1}\times\{0,1,\infty\}\times\{\infty\}^k)$, with its
differential induced by $D$. Then,
\begin{equation}\label{eqn:acyclic}H_*(X(\{0,1\}^{n-k-1}\times\{0,1,\infty\}\times\{\infty\}^k),D)=0.\end{equation}
\end{lemma}

\begin{proof}[Proof of Lemma \ref{lem:acyclic}]
Consider the decreasing filtration of
$X(\{0,1\}^{n-k-1}\times\{0,1,\infty\}\times\{\infty\}^k)$ induced by the
grading which assigns to an element $\x$ in the summand
$\ul\cfkt(\bm\alpha,\bm\eta(I))$ the number $I_1+\dots + I_{n-k-1}$. The
homology of the associated graded object is a direct sum
\begin{equation}\label{eqn:assgraded}\bigoplus_{J\in\{0,1\}^{n-k-1}}
H_*(X(\{J\times\{0,1,\infty\}\times\{\infty\}^k),D).\end{equation} Each complex
in (\ref{eqn:assgraded}) is the mapping cone of a map
\begin{equation}\label{eq:mappingcones}X(J\times\{0,1\}\times\{\infty\}^k) \rightarrow
X(J\times\{\infty\}\times\{\infty\}^k).\end{equation} Let \[I^0 = J\times\{0\}\times\{\infty\}^k,\quad I^1 = J\times\{1\}\times\{\infty\}^k, \quad I^2 = J\times\{\infty\}\times\{\infty\}^k.\] Then $X(J\times\{0,1\}\times\{\infty\}^k)$ is the mapping cone, $MC(f_{I^0,I^1})$, of \[f_{I^0,I^1}: \ul\cfkt(\bm\alpha,\bm\eta(I^0))\rightarrow \ul\cfkt(\bm\alpha,\bm\eta(I^1)),\] and the map in \eqref{eq:mappingcones} is \[f_{I^0,I^1,I^2} + f_{I^1,I^2}: MC(f_{I^0,I^1})\rightarrow \ul\cfkt(\bm\alpha,\bm\eta(I^2)).\]
The quasi-isomorphism in part $(2)$ of Proposition \ref{prop:homalgebra} factors through this map, which implies that $f_{I^0,I^1,I^2} + f_{I^1,I^2}$ is also a quasi-isomorphism. The terms in (\ref{eqn:assgraded}) are therefore zero, which implies (\ref{eqn:acyclic}).
\end{proof}

\begin{proof}[Proof of Theorem \ref{thm:cubeofres}]
For $0\leq k \leq n$, the complex
$X(\{0,1\}^{n-k-1}\times\{0,1,\infty\}\times\{\infty\}^k)$ is the mapping cone
of \[G_k:X(\{0,1\}^{n-k}\times\{\infty\}^{k})\rightarrow
X(\{0,1\}^{n-k-1}\times\{\infty\}^{k+1}),\] where $G_k$ is the sum, over all
$I\in\{0,1\}^{n-k}\times\{\infty\}^k$ and $I'\in
\{0,1\}^{n-k-1}\times\{\infty\}^{k+1}$, of the maps $D_{I,I'}$. By Lemma
\ref{lem:acyclic}, $G_k$ must be a quasi-isomorphism. The composition
\[G=G_0\circ\dots G_{n-1}:X(\{0,1\}^{n})\rightarrow X(\{\infty\}^n)\] is
therefore a quasi-isomorphism, proving Theorem \ref{thm:cubeofres}.
\end{proof}

Proposition \ref{prop:generic} and equations \eqref{eqn:err3} and \eqref{eqn:err4} immediately imply the following
corollary.

\begin{corollary} \label{cor:iso}
For any system of weights $\r$,
\[
H_*(X(\{0,1\}^n),D) \cong \HFKtt(\LL,[\omega_\r]_{I^\infty}; \MM) \cong \ul\HFK(L, [\omega_\r]_{I^\infty}; \MM) \otimes_\F (V^{\otimes (m-|L|)}.
\]
In particular, if $\r = \r_\Omega$ for a function $\Omega\co \{1, \dots, n\} \to \Z$, then
\[
H_*(X(\{0,1\}^n),D) \cong \HFK(L) \otimes_{\F} (V^{\otimes (m-|L|)} \otimes_\F \MM. \qedhere
\]
\end{corollary}

Note that $X(\{0,1\}^n)$ has a decreasing filtration induced by the grading $Q$ which assigns to any element of the summand $\ul\cfkt(\bm\alpha,\bm\eta(I))$ the number $|I|$. We shall refer to $Q$ as the \emph{filtration grading}. This
filtration gives rise to a spectral sequence $\SS_\MM^\r$. (If $\r = \r_\Omega$, we may denote this spectral sequence $\SS_\MM^\Omega$ as in the Introduction.) The $E_1$ page of $\SS_\MM^\r$ is the direct sum
\[
\bigoplus_{I \in \{0,1\}^n} \HFKtt(\LL_I, [\omega_\r]_I; \MM),
\]
and its $d_1$ differential is the sum of the maps $(f_{I,I'})_*$, over immediate successors $I'$ of $I$. We shall be interested in the case that $\r$ is generic and $\MM=\FF$. In this case, the $E_1$ page of $\SS_\FF^\r$ is a sum over connected resolutions,
\begin{equation}\label{eq:e2twisted}
\bigoplus_{I\in\RR(D)} \HFKt(\LL_I)\otimes_{\F}\FF,
\end{equation}
since $\HFKtt(\LL_I, [\omega_{\Omega}]_I;\FF)$ vanishes if $\DD_I$ is disconnected, by Proposition \ref{prop:generic}, and is isomorphic to $\HFKt(\LL_I)\otimes_{\F}\FF$ if $\DD_I$ is connected, by \eqref{eqn:err}. Since no edge in the cube of resolutions of $\DD$ can join two connected resolutions, the $d_1$ differential of $\SS_\FF^\r$ is zero. Therefore,
$E_2(\SS_\FF^\r) \cong E_1(\SS_\FF^\r)$. In Section \ref{sec:delta}, we prove that $\SS_\FF^\r$ collapses at its $E_3$ page. Since $\FF$ is a field, Corollary \ref{cor:iso} implies that
\[
H_*(E_2(\SS_\FF^\r),d_2(\SS_\FF^\r)) \cong \ul\HFK(L, [\omega_\r]) \otimes_{\F} V^{\otimes(m-|L|)};
\]
if $\r = \r_\Omega$, then
\[
H_*(E_2(\SS_\FF^\r),d_2(\SS_\FF^\r)) \cong \HFK(L) \otimes_{\F} V^{\otimes(m-|L|)} \otimes_\F \FF.
\]
In Section \ref{sec:twisted}, we show that
$(E_2(\SS_\FF^\r), d_2(\SS_\FF^\r))$ is isomorphic to the complex $(C^\r(\DD), \partial^\r)$ defined in Section \ref{sec:complex}. Combined with the grading calculations in Section \ref{sec:delta}, this proves Theorem
\ref{thm:main}.

We end this section with a brief discussion of orientations and gradings.

Recall that for $I \in \{0,1\}^n$, the Heegaard diagram $(\Sigma, \bm\alpha, \bm\eta(I), \O,\X)$ determines $L_I$ as an oriented link, where $L_I$ is oriented as the boundary of the black regions in $\DD_I$. Therefore, for $I,I' \in \{0,1\}^n$, the complexes $\CFKtt(\bm\alpha, \bm\eta(I))$ and $\CFKtt(\bm\eta(I), \bm\eta(I'))$ come equipped with Maslov and Alexander gradings.

Suppose $I'$ is an immediate successor of $I$, differing in the $j\Th$ entry. The Maslov and Alexander gradings of $\Theta^{I,I'}_1$ differ from those of $\Theta^{I,I'}_2$ by 1 (in the same direction); from now on, we shall
assume that $\Theta^{I,I'}_1$ is the unique element of $\T_{\eta(I)} \cap \T_{\eta(I')}$ in the maximal Maslov grading. Furthermore, we may consider the chain maps
\[
\begin{gathered}
\Psi^{\alpha,\eta(I)}_i \co  \CFKtt(\bm\alpha, \bm\eta(I)) \to
\CFKtt(\bm\alpha,\bm\eta(I)) \\
\Psi^{\eta(I), \eta(I')}_i \co  \CFKtt(\bm\eta(I), \bm\eta(I')) \to
\CFKtt(\bm\eta(I), \bm\eta(I'))
\end{gathered}
\]
defined in Section \ref{sec:psi}, which count disks that go over the basepoints in $\X$. We shall use these maps to describe the differentials in the spectral sequence $\SS^{\r}_\FF$. The following lemma will be useful.

\begin{lemma} \label{lem:psitheta1}
Suppose $I'$ is an immediate successor of $I$ which differs from $I$ in its $j\Th$ coordinate. Let $i_1$ and $i_3$ be the special indices associated to the crossing $c_j$, as shown in Figure \ref{fig:convention}. Then
\[
\Psi^{\eta(I),\eta(I')}_{i_1} (\Theta^{I,I'}_1) = \Psi^{\eta(I),\eta(I')}_{i_3} (\Theta^{I,I'}_1) = \Theta^{I,I'}_2,
\]
while
$\Psi^{\eta(I),\eta(I')}_i(\Theta^{I,I'}_1)=0$ for $i \ne i_1, i_3$.
\end{lemma}

\begin{proof}[Proof of Lemma \ref{lem:psitheta1}]
It is not hard to see that there is a unique $\z\in\T_{\eta(I)}\cap\T_{\eta(I')}$ such that there exists a Whitney disk
$\phi\in\pi_2(\Theta^{I,I'}_1,\z)$ with $X_{i_1}(\phi)=1$, which avoids $\O$ and all other $X_i$. Namely, $\z$ is the point $\Theta^{I,I'}_2$ and the domain of $D(\phi)$ is an annulus, as in the proof of Lemma \ref{lem:isom}. There are actually two such disks in $\pi_2(\Theta^{I,I'}_1,\Theta^{I,I'}_2)$ with $\mu=1$, exactly one of which admits a holomorphic representative. (Compare \cite[proof of Lemma 9.4]{OSz3Manifold}.) This proves that $\Psi^{\eta(I),
\eta(I')}_{i_1}(\Theta^{I,I'}_1) = \Theta^{I,I'}_2$. The other statements follow similarly.
\end{proof}

\section{On \texorpdfstring{$\delta$}{delta}-gradings}
\label{sec:delta}
The summands $\CFKtt(\bm\alpha,\bm\eta(I))$ of $X(\{0,1\}^n)$ are endowed with
canonical absolute $\delta$-gradings, by Lemma \ref{lem:orientationunlink}, and
the complex $X(\{\infty\}^n)=\CFKtt(\bm\alpha,\bm\eta(I^{\infty}))$ has an
absolute $\delta$-grading determined by the orientation of the original link
$L$. Let $\Delta$ denote the grading on $X(\{0,1\}^n)$ obtained by shifting the
$\delta$-grading on each summand $\CFKtt(\bm\alpha, \bm\eta(I))$ by $(\abs{I} -
n_-(\DD))/2$. The two main results of this section are as follows.

\begin{theorem} \label{thm:gradedcube}
With respect to $\Delta$,
\begin{enumerate}
\item the differential $D$ on $X(\{0,1\}^n)$ is homogeneous of degree $1$, and
\item the quasi-isomorphism
\[
G:X(\{0,1\}^n)\rightarrow X(\{\infty\}^n)
\]
coming from Theorem \ref{thm:cubeofres} is grading-preserving.
\end{enumerate}
\end{theorem}

\begin{theorem} \label{thm:collapse}
If $\r$ is generic, then the differential $d_k(\SS_\FF^\r)$ vanishes for $k>2$. Therefore,
\[
H_*(E_2(\SS_\FF^\r),d_2(\SS_\FF^\r)) \cong \ul\HFK(L, [\omega]_\r) \otimes_{\F} V^{\otimes(m-|L|)};
\]
as graded vector spaces over $\FF$, with respect to the $\delta$-grading on $\HFK(L)$.
\end{theorem}

\begin{proof} [Proof of Theorem \ref{thm:collapse}]
By definition, $d_k(\SS^\r_\FF)$ is homogeneous of degree $k$ with respect to the filtration grading $Q$, defined in Section \ref{sec:cube}. By Theorem \ref{thm:gradedcube}, $\Delta$ descends to a grading on the pages of $\SS_\FF^\r$. Recall, from the previous section, that $E_2(\SS_\FF^\r)$ consists of a copy of the group $\HFKtil(\LL_I)\otimes_{\F}\FF$ for each $I\in\RR(\DD)$. Since $\LL_I$ is a pointed unknot, this group is supported in the $\Delta$-grading $(\abs{I} -n_-(\DD))/2$; that is, the gradings $\Delta$ and $Q$ on $E_2(\SS^\r_\FF)$ are related by
\[
\Delta = (Q-n_-(\DD))/2.
\]
This relationship therefore holds for all $k\geq 2$. Suppose that $x$ is a nonzero, homogeneous element of $E_k(\SS_\FF^\r)$. If $d_k(\SS_\FF^\r)(x) = y \ne 0$, then
\[
\begin{aligned}
0 &= (2 \Delta(y) - Q(y)) - (2 \Delta(x) - Q(y)) \\
&= 2(\Delta(y) - \Delta(x)) - (Q(y) - Q(x)) \\
&= 2 - k.
\end{aligned}
\]
Thus, $d_k(\SS_\FF^\r)$ vanishes for $k>2$. The second statement follows immediately from Theorem \ref{thm:gradedcube} and Corollary \ref{cor:iso}.
\end{proof}

The rest of this section is devoted to proving Theorem \ref{thm:gradedcube}.

\subsection{The relative \texorpdfstring{$\delta$}{delta}-grading}
First, we show that the maps $f_{I^0,\dots,I^0}$ are homogeneous with respect to the relative $\delta$-grading. For a Whitney polygon $\psi$, let $\delta(\psi)$ denote the difference $P(\psi)-\mu(\psi)$. Note that this quantity is additive under concatenation of polygons.

\begin{proposition} \label{prop:welldefined}
Suppose $I^0<\dots<I^k$ is a successor sequence of tuples in $\{0,1,\infty\}^n$. For $i=0,\dots,k$, let $\delta_{I^i}$ be an arbitrary absolute lift of the relative $\delta$-grading on $\ul\cfkt(\bm\alpha,\bm\eta(I^i))$. Then, for each $i,j$ with $0\leq i< j \leq k$, there are constants $\delta(f_{I^i,\dots,I^j})$ such that $f_{I^i, \dots, I^j}$ is homogeneous of degree $\delta(f_{I^i,\dots,I^j})$ with respect to these absolute lifts, and
\begin{equation} \label{eq:compose}
\delta(f_{I^i,\dots,I^j}) = \delta(f_{I^i,\dots,I^l}) + \delta(f_{I^l,\dots,I^j}) - 1,
\end{equation}
for any $i < l < j$.
\end{proposition}

\begin{proof}[Proof of Proposition \ref{prop:welldefined}]
Choose some $\x^s \in \Ta \cap \T_{\eta(I^i)}$ and $\y^s \in \Ta \cap \T_{\eta(I^j)}$ for $s=1,2$, and suppose $\psi_s$ is a Whitney $(j-i+2)$-gon in $ \pi_2(\x^s, \Theta^{I^i,I^{i+1}}_{e^s_i}, \dots,
\Theta^{I^{j-1},I^j}_{e^s_{j-1}}, \y^s),$ where $e_i^s, \dots, e_{j-1}^s \in \{1,2\}$. Such polygons always exist since the pairs $(\bm\alpha,\bm\eta(I^l))$ span $H_1(\Sigma;\mathbb{Z})$ for $l=i,\dots,j$ (see, e.g., \cite[Proposition 8.3]{OSz3Manifold}). We claim that
\begin{equation}\label{eqn:deltagrading}
\delta_{I^j}(\y^1) -\delta_{I^i}(\x^1) + \delta(\psi_1) =
\delta_{I^j}(\y^2) -\delta_{I^i}(\x^2) + \delta(\psi_2),
\end{equation}
which enables us to define the quantity
\begin{equation} \label{eqn:Deltaij}
\delta(f_{I^i,\dots,I^j}) =\delta_{I^j}(\y^1) -\delta_{I^i}(\x^1) +
\delta(\psi_1) + i-j+1
\end{equation}
independently of $\x^s$, $\y^s$ and $\psi_s$. To prove \eqref{eqn:deltagrading}, let $\psi_s'$ be the Whitney $(j-i+1)$-gon in
$\pi_2(\x^s, \Theta^{I^i,I^{i+1}}_1, \dots, \Theta^{I^{j-1},I^j}_1, \y^s)$
obtained by concatenating $\psi_s$ with Whitney disks $\phi_i \in
\pi_2(\Theta^{I^l,I^{l+1}}_1, \Theta^{I^l,I^{l+1}}_2)$ where necessary. Since
each $\phi_i$ satisfies $P(\phi_i) = \mu(\phi_i) = 1$, we have that
$\delta(\psi_s')=\delta(\psi_s)$. Choose some Whitney disks $\phi_x \in
\pi_2(\x^1,\x^2)$ and $\phi_y \in \pi_2(\y^1,\y^2)$, and consider the concatenation
\[
\psi_2''=\phi_x*\psi_2'*\overline{\phi}_y \, \in \,
\pi_2(\x^1,\Theta^{I^i,I^{i+1}}_1, \dots, \Theta^{I^{j-1},I^j}_1, \y^1)
\]
The difference $D(\psi_1')-D(\psi_2'')$ is a multi-periodic domain, so
\[
\delta(\psi_1')-\delta(\psi_2'') = 0, \] by Proposition \ref{prop:mu=P}. Thus,
\begin{align*}
\delta(\psi_1) = \delta(\psi_1') &= \delta(\psi_2'')\\
&= \delta(\psi_2') + \delta(\phi_x)- \delta(\phi_y)\\
&=\delta(\psi_2)+
(\delta_{I^i}(\x^1)-\delta_{I^i}(\x^2))-(\delta_{I^j}(\y^1)-\delta_{I^j}(\y^2)),
\end{align*}
from which (\ref{eqn:deltagrading}) follows. Now, if $\y \in \T_\alpha \cap \T_{\eta(I^j)}$ appears with nonzero coefficient in
$f_{I^i, \dots, I^j}(\x)$ for some $\x \in \T_\alpha \cap \T_{\eta(I^i)}$, then there exists a $(j-i+2)$-gon $\psi \in \pi_2(\x, \Theta^{I^i,I^{i+1}}_{e_i}, \dots,
\Theta^{I^{j-1},I^j}_{e_{j-1}}, \y)$
with \[\delta(\psi) = P(\psi)-\mu(\psi) = 0-(i-j+1) = j-i-1.\] By \eqref{eqn:Deltaij}, $\delta_{I^j}(\y) -\delta_{I^i}(\x) =\delta(f_{I^i,\dots,I^j})$. It follows that $f_{I^i,\dots,I^j}$ is homogeneous of degree $\delta(f_{I^i,\dots,I^j}).$

For the second part, let $\x$ and $\y$ be as above, and let $\z\in \T_\alpha
\cap \T_{\eta(I^l)}$. Choose Whitney polygons
\[
\psi_1 \in \pi_2(\x, \Theta^{I^i,I^{i+1}}_{e_i}, \dots,
\Theta^{I^{l-1},I^l}_{e_{l-1}}, \z) \quad \text{and} \quad \psi_2 \in \pi_2(\z,
\Theta^{I^l,I^{l+1}}_{e_i}, \dots, \Theta^{I^{j-1},I^j}_{e_{j-1}}, \y),
\]
and let $\psi = \psi_1 * \psi_2$. Then,
\[
\begin{aligned}
\delta(f_{I^i,\dots,I^l}) + \delta(f_{I^l,\dots,I^j}) &=
 (\delta_{I^l}(\z) -\delta_{I^i}(\x) + \delta(\psi_1) +
 i-l+1) \\
 &\quad+ (\delta_{I^j}(\y) -\delta_{I^l}(\z) + \delta(\psi_2) +
 l-j+1) \\
&=\delta_{I^j}(\y) -\delta_{I^i}(\x) + \delta(\psi) +
 i-j+2 \\
&= \delta(f_{I^i,\dots,I^j})+1,
\end{aligned}
\]
completing the proof of Proposition \ref{prop:welldefined}.
\end{proof}

\begin{remark}
Proposition \ref{prop:welldefined} shows that that the grading shifts
$\delta(f_{I^i,\dots,I^j})$ are well-defined and satisfy additivity properties
under composition even if some of the maps $f_{I^i,\dots,I^j}$ are zero.
\end{remark}

The next result shows that the maps $D_{I,I'}$ are homogeneous with respect to
the relative $\delta$-grading.

\begin{proposition} \label{prop:commute}
Suppose $\bar I= I^0<\dots<I^k$ and $\bar J=J^0<\dots<J^k$ are successor sequences
of tuples in $\{0,1,\infty\}^n$ with $I^0=J^0$ and $I^k=J^k$. For any absolute lifts $\delta_{I^0}$ and $\delta_{I^k}$ of the relative $\delta$-gradings on $\CFKtt(\bm\alpha,
\bm\eta(I^0))$ and $ \CFKtt(\bm\alpha, \bm\eta(I^k))$, the grading shifts
$\delta(f_{I^0,\dots, I^k})$ and $\delta(f_{J^0,\dots, J^k})$ are equal. In particular, the map
\[
D_{I^0,I^k}\co \CFKtt(\bm\alpha, \bm\eta(I^0)) \to (\bm\alpha, \bm\eta(I^k))
\]
is homogeneous of degree $\delta(D_{I^0,I^k}) = \delta(f_{I^0,\dots, I^k})$
with respect to these absolute lifts.
\end{proposition}

\begin{proof}[Proof of Proposition \ref{prop:commute}]
The sequences $\bar I$ and $\bar J$ can be connected by an ordered list of
sequences in which one sequence in the list differs from the next in a single
place. It is therefore enough to prove Proposition \ref{prop:commute} for $\bar I$ and $\bar J$, where
\[
\bar J = I^0<\dots<I^{i-1}<J^i<I^{i+1}<\dots<I^k.
\]
For $i=0,\dots,k$, let $\delta_{I^i}$ and $\delta_{J^i}$ denote arbitrary
absolute lifts of the relative $\delta$-gradings on the complexes
$\ul\cfkt(\bm\alpha,\bm\eta(I^i))$ and $\ul\cfkt(\bm\alpha,\bm\eta(J^i))$. By
\eqref{eq:compose}, we need only show that
\[
\delta(f_{I^{i-1},J^i}) + \delta(f_{J^i,I^{i+1}}) = \delta(f_{I^{i-1},I^i}) +
\delta(f_{I^i,I^{i+1}}).
\]
It is helpful to have in mind the following diagram, which commutes up to
homotopy.
\[
\xymatrix @!0@C=11pc@R=5pc{
\ul\cfkt(\bm\alpha,\bm\eta(I^{i-1})) \ar[r]^-{f_{I^{i-1},I^i}} \ar[d]_-{f_{I^{i-1},J^i}}& \ul\cfkt(\bm\alpha,\bm\eta(I^i))\ar[d]^-{f_{I^i,I^{i+1}}}\\
\ul \cfkt(\bm\alpha,\bm\eta(J^i)) \ar[r]_-{f_{J^i,I^{i+1}}} &
\ul\cfkt(\bm\alpha,\bm\eta(I^{i+1})).}
\]
Choose generators
\[
\x^1\in \T_{\alpha}\cap \T_{\eta(I^{i-1})},
 \quad \y^1\in \T_{\alpha} \cap \T_{\eta(I^{i})},
 \quad \x^2\in \T_{\alpha} \cap \T_{\eta(J^{i})},
 \quad \y^2\in \T_{\alpha} \cap \T_{\eta(I^{i+1})},
\]
and Whitney triangles
\begin{align*}
\psi_{I^{i-1},I^i} & \in \pi_2(\x^1,\Theta_1^{I^{i-1},I^i},\y^1), &
\psi_{I^{i},I^{i+1}} & \in \pi_2(\y^1,\Theta_1^{I^{i},I^{i+1}},\y^2),\\
\psi_{I^{i-1},J^i} & \in \pi_2(\x^1,\Theta_1^{I^{i-1},J^i},\x^2), &
\psi_{J^{i},I^{i+1}} & \in \pi_2(\x^2,\Theta_1^{J^{i},I^{i+1}},\y^2).
\end{align*}
Let $\psi_1 \in
\pi_2(\x^1,\Theta_1^{I^{i-1},I^i},\Theta_1^{I^{i},I^{i+1}},\y^2)$ and $\psi_2
\in \pi_2(\x^1,\Theta_1^{I^{i-1},J^i},\Theta_1^{J^{i},I^{i+1}},\y^2)$ denote
the Whitney rectangles obtained by concatenation,
\[
\psi_1 = \psi_{I^{i-1},I^i}*\psi_{I^{i},I^{i+1}}, \quad \psi_2 =
\psi_{I^{i-1},J^i}*\psi_{J^{i},I^{i+1}}.
\]
As in the proof of Lemma \ref{lem:triangles}, there exists some generator
$\Theta^{J^i,I^i}_s \in \T_{\eta(J^i)} \cap \T_{\eta(I^i)}$ (one of the four
generators with minimal $\delta$-grading) such that there are Whitney triangles
\[
\tau_1 \in \pi_2(\Theta_1^{I^{i-1},J^i}, \Theta_s^{J^i,I^i},
 \Theta_1^{I^{i-1},I^i}) \quad \text{and} \quad
\tau_2 \in\pi_2(\Theta^{J^i,I^i}_s, \Theta_1^{I^i,I^{i+1}},
 \Theta_1^{J^i,I^{i+1}}),
\]
whose domains are disjoint unions of small triangles, so that
\[
P(\tau_1) = \mu(\tau_1) = P(\tau_2)= \mu(\tau_2)=0.
\]
Let $\phi_1$ and $\phi_2$ denote the Whitney pentagons in $\pi_2(\x^1,\Theta_1^{I^{i-1},J^i},\Theta^{J^i,I^i}_s,\Theta_1^{I^i,I^{i+1}},\y^2)$
obtained by concatenating $\tau_1$ with $\psi_1$ at
$\Theta^{I^{i-1}, I^i}_1$ and $\tau_2$ with $\psi_2$ at $\Theta^{J^i, I^{i+1}}_1$, respectively. 
The difference $D(\phi_1)-D(\phi_2)$ is a multi-periodic domain. Therefore,
\begin{align*}
0&=\delta(\phi_1) - \delta(\phi_2)\\
&=\delta(\psi_1)-\delta(\psi_2)\\
&=(\delta(\psi_{I^{i-1},I^i})+\delta(\psi_{I^{i},I^{i+1}}))
 - (\delta(\psi_{I^{i-1},J^i})+\delta(\psi_{J^{i},I^{i+1}}))\\
&=(\delta(f_{I^{i-1},I^i}) +\delta_{I^{i-1}}(\x^1)-\delta_{I^i}(\y^1)
 + \delta(f_{I^{i},I^{i+1}}) +\delta_{I^i}(\y^1) -\delta_{I^{i+1}}(\y^2)) \\
& \quad -(\delta(f_{I^{i-1},J^i})
+\delta_{I^{i-1}}(\x^1)-\delta_{J^i}(\x^2)
 + \delta(f_{J^{i},I^{i+1}}) +\delta_{J^i}(\x^2)-\delta_{I^{i+1}}(\y^2))\\
&=(\delta(f_{I^{i-1},I^i}) + \delta(f_{I^{i},I^{i+1}}))-(\delta(f_{I^{i-1},J^i})+
\delta(f_{J^{i},I^{i+1}})),
\end{align*}
completing the proof of Proposition \ref{prop:commute}.
\end{proof}

Before proceeding further, we pause to record a fact about Alexander and Maslov
gradings that will be useful in Section \ref{sec:twisted}. Recall that, for
$I\in\{0,1\}^n,$ we orient the diagrams $\DD_I$ as boundaries of the black
regions. These orientations determine absolute Maslov and Alexander gradings on
the complexes $\ul\cfkt(\bm\alpha,\bm\eta(I)),$ per the discussion in Section
\ref{sec:cube}.

\begin{proposition} \label{prop:maslovshift}
Suppose $I^0<\dots<I^k$ is a successor sequence of tuples in $\{0,1\}^n$. For any $e_1, \dots, e_k \in \{1,2\}$, the map
\[
F_{\alpha, \eta(I^0), \cdots, \eta(I^k)}(\cdot \otimes \Theta^{I^0, I^1}_{e_1}
\otimes \dots \otimes \Theta^{I^{k-1},I^k}_{e_k}):
\ul\cfkt(\bm\alpha,\bm\eta(I^0))\rightarrow\ul\cfkt(\bm\alpha,\bm\eta(I^k))
\]
is homogeneous with respect to both the Alexander and Maslov gradings.
Moreover, the Alexander and Maslov grading shifts
of
\[
F_{\alpha, \eta(I^0), \cdots, \eta(I^k)}(\cdot \otimes \Theta^{I^0, I^1}_{e_1}
\otimes \dots \otimes \Theta^{I^{i-1},I^i}_1 \otimes \dots \otimes
\Theta^{I^{k-1},I^k}_{e_k})
\]
are one greater than those of
\[
F_{\alpha, \eta(I^0), \cdots, \eta(I^k)}(\cdot \otimes \Theta^{I^0, I^1}_{e_1}
\otimes \dots \otimes \Theta^{I^{i-1},I^i}_2 \otimes \dots \otimes
\Theta^{I^{k-1},I^k}_{e_k}).
\]
\end{proposition}

\begin{proof}[Proof of Proposition \ref{prop:maslovshift}]
This follows from the same reasoning as was used in the proof of Proposition
\ref{prop:welldefined}. The key element in the latter was Proposition
\ref{prop:mu=P}, which, in turn, follows from the fact that any doubly-periodic
domain $D$ in the multi-diagram $\HH$ satisfies $\mu(D)=P(D)$. To prove that
\[F_{\alpha, \eta(I^0), \cdots, \eta(I^k)}(\cdot \otimes \Theta^{I^0,
I^1}_{e_0} \otimes \dots \otimes \Theta^{I^{k-1},I^k}_{e_{k-1}})\] is
homogeneous with respect to the Alexander grading, we simply need the
modification that $O(D) = X(D)$, which is clearly true. These two facts also
imply that $\mu(D)=2O(D)$, which is the modification we need for the
homogeneity statement about Maslov gradings. The second statement in
Proposition \ref{prop:maslovshift} follows from the fact that the Maslov and
Alexander gradings of $\Theta^{I^{i-1},I^i}_1$ are each $1$ greater than those
of $\Theta^{I^{i-1},I^i}_2$.
\end{proof}

\subsection{The absolute \texorpdfstring{$\delta$}{delta}-grading}
In this subsection, we compute certain absolute $\delta$-grading shifts. These
calculations, in conjunction with Propositions \ref{prop:welldefined} and
\ref{prop:commute}, complete the proof of Theorem \ref{thm:gradedcube}.

Suppose $I^0<I^1<I^2$ is a successor sequence of tuples in $\{0,1,\infty\}^n$
which differ only in their $j^{\text{th}}$ coordinates, and consider the maps
\begin{equation} \label{eq:triple}
\xymatrix {
\ul\cfkt(\bm\alpha,\bm\eta(I^{2})) \ar[r]^-{f_0} \ar@/_1.6pc/[rr]_{H_1}  &
\ul\cfkt(\bm\alpha,\bm\eta(I^0))\ar[r]^-{f_1} \ar@/_1.6pc/[rr]_{H_{2}} &
\ul\cfkt(\bm\alpha,\bm\eta(I^1)) \ar[r]^-{f_{2}} &
\ul\cfkt(\bm\alpha,\bm\eta(I^{2})),}
\end{equation}
where
\[
f_0=f_{I^2,I^0},\quad
f_1 = f_{I^0,I^1},\quad f_{2} = f_{I^1,I^2},\] and
\[
H_1=f_{I^2,I^0,I^1},\quad H_{2}= f_{I^0,I^1,I^2}.
\]
According to Proposition \ref{prop:homalgebra}, the sum $\Phi = f_2 \circ H_1 +
H_2 \circ f_0$ is a grading-preserving quasi-isomorphism. Now, fix an orientation on $\DD_{I^2}$. If the crossing $c_j$ is positive, then
$\DD_{I^1}$ naturally inherits an orientation from $\DD_{I^2}$. We
choose an orientation of $\DD_{I^0}$ that agrees with the orientation of $\DD_{I^2}$ on every
component of $\DD_{I^0}$ that does not pass through a neighborhood of $c_j$. Likewise, if $c_j$ is negative,
then $\DD_{I^0}$ inherits an orientation from $\DD_{I^2}$, and we choose an orientation of $\DD_{I^1}$ that agrees with that of $\DD_{I^2}$ on every
component of $\DD_{I^1}$ away from $c_j$. For $i=0,1,2$, let $n_{\pm}(\DD_{I^i})$ denote the number of $\pm$
crossings in $\DD_{I^i}$ with respect to these orientations.

\begin{proposition} \label{prop:triple}
If $c_j$ is positive, then $\delta(f_0) = n_-(\DD_{I^2}) - n_-(\DD_{I^0})$ and
$\delta(f_2) = \frac12$. If $c_j$ is negative, then $\delta(f_0) = \frac12$ and
$\delta(f_2) = n_+(\DD_{I^2}) - n_+(\DD_{I^1})$. In either case, $\delta(H_1) =
-\delta(f_2)$ and $\delta(H_2) = -\delta(f_0)$.
\end{proposition}

Before proving Proposition \ref{prop:triple}, we illustrate how it is used to prove Theorem \ref{thm:gradedcube}, starting with the corollary below.

\begin{proposition}
\label{prop:diffgradingshift} Suppose $I^0<\dots<I^k$ is a successor sequence
of tuples in $\{0,1\}^n$. Then $\delta(f_{I^0 \dots I^k}) = (2-k)/2.$
\end{proposition}

\begin{proof}[Proof of Proposition \ref{prop:diffgradingshift}]
Suppose $I^{i}$ and $I^{i+1}$ differ in their $j\Th$ entries, and let $J$ be
the tuple obtained by replacing this entry with $\infty$. We may identify
$f_{I^i,I^{i+1}}$ with the map $f_1$ in \eqref{eq:triple}. Note that $\DD_{J}$
is a diagram for an unlink with only one crossing. By Proposition
\ref{prop:welldefined}, $\delta(f_0) + \delta(f_1) = \delta(H_1) + 1$. If $c_j$
is positive, then $\delta(f_0) = 0$ and $\delta(H_1)= -\frac12$, by Proposition
\ref{prop:triple}; otherwise, $\delta(f_0) = \frac12$ and $\delta(H_1)= 0$. In
either case, $\delta(f_1) = \frac12$. According to Proposition
\ref{prop:welldefined}, \[\delta(f_{I^0,\dots,I^k}) = \delta(f_{I^0,I^1})+
\dots + \delta(f_{I^{k-1},I^k}) - (k-1) = (2-k)/2,\] as claimed.
\end{proof}

\begin{proof}[Proof of Theorem \ref{thm:gradedcube}]
Suppose $I^0\leq I^k$ are tuples in $\{0,1\}^n$ which differ in $k$ entries.
The grading shift of $D_{I^0,I^k}$ with respect to the grading $\Delta$ is
\[(2-k)/2 - (|I^0|-n_-(\DD))/2 + (|I^k|-n_-(\DD))/2 = 1,\] by Proposition
\ref{prop:commute} and Proposition \ref{prop:diffgradingshift}. This proves the
first statement of Theorem \ref{thm:gradedcube}.

Now, let $G_I$ denote the restriction of $G$ to the summand $\CFKtt(\bm\alpha,
\bm\eta(I))$. Recall that $G_I$ is the sum, over all sequences $I = I^0 < \dots
< I^n = I^\infty$ with $I^k \in \{0,1\}^{n-k} \times \{\infty\}^k$, of the
compositions $D_{I^{n-1}, I^n} \circ \dots \circ D_{I^0, I^1}$. It follows
easily from Propositions \ref{prop:welldefined} and \ref{prop:commute} that
$G_I$ is homogeneous. Choose a sequence $I = I^0<\dots<I^n = I^{\infty}$ as
above, and let $J = (0,\dots,0)$. Choose absolute lifts of the relative
$\delta$-gradings on the complexes $\ul\cfkt(\bm\alpha,\bm\eta(I^i))$. By
Propositions \ref{prop:welldefined}, \ref{prop:commute} and Proposition
\ref{prop:diffgradingshift}, \[ \delta(D_{J,I^1}) = \delta(D_{J,I^0}) +
\delta(D_{I^0,I^1}) -1 = (2-|I|)/2+ \delta(D_{I^0,I^1}) -1= \delta(D_{I^0,I^1})
- \abs{I}/2.
\]
Adding $\delta(D_{I^1,I^2}) + \dots + \delta(D_{I^{n-1},I^n})$ to both sides,
we have $\delta(G_J) = \delta(G_I) - \abs{I}/2$; that is, $\delta(G_I) = |I|/2
+ C$ for some constant $C$. For the second statement of Theorem
\ref{thm:gradedcube}, it suffices to show that $C=-n_-(\DD)$.

Define a tuple $I^0 \in \{0,1\}^n$ according to the following rule: if $c_j$ is
a positive crossing, let $(I^0)_j=1$; otherwise, let $(I^0)_j=0$. Note that
$\DD_{I^0}$ is the oriented (Seifert) resolution of $\DD$. For $i=1, \dots, n$,
let $I^i$ be the tuple obtained by changing the last $i$ entries of $I^0$ to
$\infty$. One of the terms appearing in $G_{I^0}$ is $D_{I^{n-1},I^n} \circ
\dots \circ D_{I^0,I^1}$. In this composition, $n_+(\DD)$ of the maps are of
the form $f_2$, as in \eqref{eq:triple}, while $n_-(\DD)$ are of the form
$H_{2}$. Therefore,
\[
\delta(G_{I^0}) = n_+(\DD)/2 - n_-(\DD)/2 = |I^0|/2-n_-(\DD)/2,
\]
which implies that $C=-n_-(\DD)$.
\end{proof}

The rest of this section is devoted to proving Proposition \ref{prop:triple}.
For this, it helps to know the $\delta$-gradings of certain generators. Let
$A_1,\dots,A_k$ denote the regions in the diagram $\DD_I$ that are not adjacent
to the marking $p_m$. Recall that a \emph{Kauffman state} is a bijection which
assigns to each crossing $c$ of $\DD_I$ one of the regions $A_i$ incident to
$c$.

A generator $\x$ of $\ul\cfkt(\bm\alpha,\bm\eta(I))$ is said to be
\emph{Kauffman} if $\x$ does not contain any intersections points between
ladybug and non-ladybug curves. A Kauffman generator $\x$ determines a Kauffman
state $s_\x$ as follows: for each crossing $c$, let $s_\x(c)$ be the region
whose corresponding $\bm\alpha$ curve intersects $\eta(I)_c$ in a point of
$\x$. (This correspondence is $2^{m-1}$-to-1.) Let $\delta(\x,c)\in\{0,\pm
1/2\}$ be the quantity defined in Figure \ref{fig:kauffman}, according to which
region is assigned to $c$ in $s_\x$. Ozsv{\'a}th and Szab{\'o}
\cite{OSzAlternating} prove that
\begin{equation}\label{eqn:kauff}\delta(\x) = \sum_c
\delta(\x,c).\end{equation}

 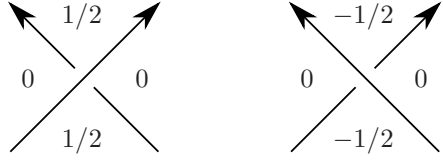
\begin{figure}[!htbp]
 \labellist
 \small\hair 2pt
  \pinlabel $0$ at 30 47
    \pinlabel $0$ at 81 47
 \pinlabel $1/2$ at 54 75
  \pinlabel $1/2$ at 54 19

   \pinlabel $0$ at 155 47
    \pinlabel $0$ at 206 47
 \pinlabel $-1/2$ at 180 75
  \pinlabel $-1/2$ at 180 19

\endlabellist
\begin{center}
\includegraphics[width=6cm]{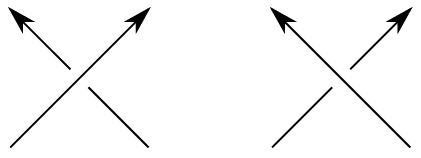}
\caption{The local contributions to $\delta(\x)$ near a crossing.}
\label{fig:kauffman}
\end{center}
\end{figure}

\begin{proof}[Proof of Proposition \ref{prop:triple}]
We only consider the case in which $c_j$ is a positive crossing; the proof for a
negative crossing is extremely similar.

First, suppose that the projections $\DD_{I^0}$, $\DD_{I^1}$, and $\DD_{I^2}$
are connected. In this case, we use an argument due to Manolescu and Ozsv\'ath
\cite{ManolescuOzsvathQA}. Choose some Kauffman generators $\x^0 \in \T_\alpha
\cap \T_{\eta(I^0)}$ and $\x^1 \in \T_\alpha \cap \T_{\eta(I^1)}$. As in
\cite[Section 3.5]{ManolescuOzsvathQA}, we may find corresponding Kauffman
generators $\y^0, \y^1 \in \T_\alpha \cap \T_{\eta(I^2)}$ such that
\begin{inparaenum}\item$s_{\x_i}(c)= s_{\y_i}(c)$ for each $c \ne c_j$, and
\item there exist homotopy classes $\psi_0 \in \pi_2(\y_0,
\Theta^{I^2,I^0}_{e_0}, \x_0)$ and $\psi_1 \in \pi_2(\x_1,
\Theta^{I^1,I^2}_{e_1}, \y_1)$ with $\delta(\psi_i) =0$ (for some
$e_i\in\{1,2\}$). \end{inparaenum}

Since $\DD_{I^1}$ is the oriented resolution of $\DD_{I^2}$, $\delta(\x_1,c) =
\delta(\y_1,c)$ for every crossing $c \ne c_j$, while $\delta(\y_1,c_j) =
\frac12$. Therefore, $\delta(f_2) = \delta(\y_1) - \delta(\x_1) = \frac12$, as
claimed. On the other hand, the sign of a crossing in $\DD_{I^0}$ need not be
the same as its sign in $\DD_{I^2}$. For any crossing $c \ne c_j$ that is
negative in $\DD_{I^2}$ and positive in $\DD_{I^0}$, we have $\delta(\x_0,c) =
\delta(\y_0,c) + \frac12$; the number of such crossings is $n_-(\DD_{I^2}) -
n_-(\DD_{I^0})$. Likewise, if $c \ne c_j$ is positive in $\DD_{I^2}$ and
negative in $\DD_{I^0}$, then $\delta(\x_0,c) = \delta(\y_0,c) - \frac12$; the
number of such crossings is $n_+(\DD_{I^2})-1- n_+(\DD_{I^0})$. Since
$n_+(\DD_{I^2}) + n_-(\DD_{I^2}) = n$ and $n_+(\DD_{I^0}) + n_-(\DD_{I^0}) =
n-1$,
\[
\begin{aligned}
\delta(f_0) &= \delta(\x_0) - \delta(\y_0) \\
&= \frac12 \big((n_-(\DD_{I^2}) - n_-(\DD_{I^0})\big) - \frac12\big((n_+(\DD_{I^2})-
n_+(\DD_{I^0})-1)\big) \\
&= n_-(\DD_{I^2}) - n_-(\DD_{I^0}),
\end{aligned}
\]
as claimed.

If either $\DD_{I^0}$ or $\DD_{I^1}$ is disconnected, then the corresponding complex has no Kauffman generators, so the argument above does
not apply. We remedy this situation as follows. Let $\DD'$ be the planar diagram
obtained from $\DD$ by performing a finger move just outside of each crossing
as in Figure \ref{fig:isotopy}, and let $\DD_{I^0}'$, $\DD_{I^1}'$ and
$\DD_{I^2}'$ be the corresponding resolutions of $\DD'$, leaving the newly
introduced crossings unresolved. Notice that all three of these diagrams are
connected, so the argument above applies.

\begin{figure}[!htbp]
\labellist
 \pinlabel (a) at 9 80 \pinlabel (b) at 109 80
\endlabellist
\begin{center}
\includegraphics[width=5.5cm]{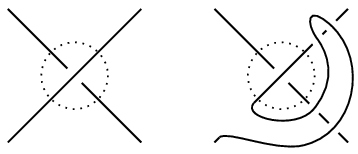}
\caption{(a) A crossing of $\DD$. (b) The modified crossing in $\DD'$.}
\label{fig:isotopy}
\end{center}
\end{figure}

Let $(\Sigma',\bm\alpha',\bm\eta'(I^0),\bm\eta'(I^1),\bm\eta'(I^2),\O,\X)$
denote the Heegaard multi-diagram encoding $\DD_{I^0}'$, $\DD_{I^1}'$ and
$\DD_{I^{2}}'$ that is obtained from $\DD'$ using the procedure in Section
\ref{sec:heegaard}, except that we do not place ladybugs on the tubes
corresponding to the edges that are contained entirely in Figure
\ref{fig:isotopy}(b). This diagram is related to $(\Sigma,\bm\alpha,
\bm\eta(I^0),\bm\eta(I^1),\bm\eta(I^2),\O,\X)$ by a sequence of handleslides,
isotopies and index one/two stabilizations avoiding $\O\cup\X$. (Essentially,
these Heegaard moves account for the Reidemeister II moves introduced by the
operation in Figure \ref{fig:isotopy}.) We therefore have diagrams,
\[
\xymatrix @!0@C=9pc@R=5pc{
\ul\cfkt(\bm\alpha,\bm\eta(I^2)) \ar[r]^-{f_0} \ar[d]_-{\Phi_2}& \ul\cfkt(\bm\alpha,\bm\eta(I^0))\ar[d]^-{\Phi_0}\\
 \ul\cfkt(\bm\alpha,\bm\eta'(I^2)) \ar[r]_-{f_0'} & \ul \cfkt(\bm\alpha,\bm\eta'(I^0))}
\quad  \quad
 \xymatrix @!0@C=9pc@R=5pc{
\ul\cfkt(\bm\alpha,\bm\eta(I^1)) \ar[r]^-{f_2} \ar[d]_-{\Phi_1}& \ul\cfkt(\bm\alpha,\bm\eta(I^2))\ar[d]^-{\Phi_2}\\
 \ul\cfkt(\bm\alpha,\bm\eta'(I^1)) \ar[r]_-{f_2'} & \ul \cfkt(\bm\alpha,\bm\eta'(I^2)),}
\]
which commute up to homotopy, where $\Phi_0$, $\Phi_1$, and $\Phi_2$ are the
grading-preserving chain homotopy equivalences associated to these Heegaard moves. An
argument very similar to that in the proof of Proposition \ref{prop:commute} shows that
$\delta(f_0) = \delta(f_0') = n_-(\DD'_{I^2}) - n_-(\DD'_{I^0}) =
n_-(\DD_{I^2}) - n_-(\DD_{I^0})$ and $\delta(f_2) = \delta(f_2') = \frac12$, as
required.

Finally, note that $H_2 \circ f_0 + f_2 \circ H_1$ is a grading-preserving
quasi-isomorphism, so at least one of these terms is nonzero. Therefore, $\delta(H_2) + \delta(f_0) =
\delta(f_2) + \delta(H_1) = 0$, by Proposition \ref{prop:welldefined}, completing the proof of Proposition \ref{prop:triple}.
\end{proof}

\section{The \texorpdfstring{$d_2$}{d\textunderscore2} differential}
\label{sec:twisted}
From now on, we shall assume that $\r$ is generic. Recall from Section
\ref{sec:cube} that the $E_2$ term of $\SS_\FF^\r$ is the direct sum
\[
\bigoplus_{I \in \RR(\DD)} \HFKtt(\bm\alpha, \bm\eta(I); \FF).
\]
With respect to this direct sum decomposition, the differential $d_2(\SS^\r_\FF)$ is a sum of maps
\[
d_{I,I''}\co \HFKtt(\bm\alpha, \bm\eta(I);\FF) \to \HFKtt(\bm\alpha,
\bm\eta(I'');\FF)
\]
over all pairs $I,I''$ for which $I''$ is a double successor of $I$. The purpose of this section is to compute these $d_{I,I''}$.

Suppose that $I,I''\in\RR(\DD)$ and that $I''$ is a double successor of $I$
which differs from $I$ in its $j_1\Th$ and $j_2\Th$. Let $J$ be the tuple
obtained from $I$ by changing its $j_1\Th$ and $j_2\Th$ entries from $0$s to
$\infty$s. Then $\DD_J$ is a 2-crossing diagram for the 2-component unlink
$L_J$. The four complete resolutions $\DD_I$, $\DD_{I^1}$, $\DD_{I^2}$ and
$\DD_{I''}$ described in Section \ref{sec:complex} are obtained from $\DD_J$ by
resolving these two crossings. Recall that $d_{I,I''}$ is defined in terms of
maps that count pseudo-holomorphic polygons in the multi-diagram
\[
\HH_{I,I''}=(\Sigma, \bm\alpha, \bm\eta(I), \bm\eta(I^1), \bm\eta(I^2),
\bm\eta(I''), \O,\X).
\]
Our strategy for computing $d_{I,I''}$ is as follows. First, we describe a
sequence of Heegaard moves from $\HH_{I,I''}$ to a ``standard'' genus 3
multi-diagram which encodes the four resolutions above. We then determine the
relevant polygon-counting maps for this genus 3 diagram. Fortunately, it
suffices to explicitly compute only a handful of these maps; the rest are
determined via the $\Psi_i$ maps defined in Section \ref{sec:psi}. Next, we
argue that this model computation determines $d_{I,I''}$ to the extent that we
can recover the isomorphism type of the complex
$(E_2(\SS^\r_\FF),d_2(\SS^\r_\FF))$. Finally, we show that this
complex is isomorphic to $(C^\r(\DD),\partial^\r)$.

As in Section \ref{sec:complex}, we assume that the marked points are ordered 
$p_1, \dots, p_m$ according to the orientation of $\DD_I$. For $i = 1, \dots,
m$, the value of $\omega_{\r}$ on the unique point of $\A$ that is
contained in the same component of $\Sigma \minus \bm\eta(I)$ as $O_i$ equals
$r_i$. As a notational convenience, we define
\[
R(i,j) = \begin{cases} r_i + \dots + r_j & i \le j \\ 0 & i> j,
\end{cases}
\]
so that $A = R(1,a)$, $B=R(a+1,b)$, $C = R(b+1,c)$, and $D=R(c+1,d)$.

\subsection{A model computation} \label{sec:destab}

We may reduce $\HH_{I,I''}$ to an admissible genus 3 multi-diagram via a
sequence of handleslides and isotopies in the complement of $\O\cup\X\cup\A$,
followed by index one/two destabilizations, as follows. Consider a crossing
$c_j$, where $j \ne j_1, j_2$. The curves $\eta_{c_j}(I)$, $\eta_{c_j}(I^1)$,
$\eta_{c_j}(I^2)$ and $\eta_{c_j}(I'')$ are pairwise isotopic and each
intersects either one or two of the $\bm\alpha$ curves corresponding to regions
of $\mathbb{R}^2\minus\DD$ in exactly one point. Call these curves $\alpha_1$
and, if necessary, $\alpha_2$. First, we handleslide all of the ladybug
$\bm\alpha$ circles which pass through $\eta_{c_j}(I)$ over $\alpha_1$, as in
Figure \ref{fig:hmoves2}(b). (This takes two handleslides for each such
$\bm\alpha$ curve.) Second, we handleslide $\alpha_2$ over $\alpha_1$ (if
applicable), as in (c). Third, we handleslide all other $\bm\eta(I)$ (resp.
$\bm\eta(I^1)$, $\bm\eta(I^2)$ and $\bm\eta(I'')$) curves which intersect
$\alpha_1$ over $\eta_{c_j}(I)$ (resp. $\eta_{c_j}(I^1)$, $\eta_{c_j}(I^2)$ and
$\eta_{c_j}(I'')$), as in (d). The resulting multi-diagram is the connected sum
of a multi-diagram of smaller genus with a standard torus piece. Handlesliding
further, we can ``move" this torus piece until it is adjacent to the region
containing $O_1$. We perform these operations for each $j\neq j_1,j_2$, and
then destabilize $n-2$ times.\footnote{This is destabilization in the sense of
multi-diagrams; see \cite{OSz4Manifold, RobertsSpectralSequence}.}

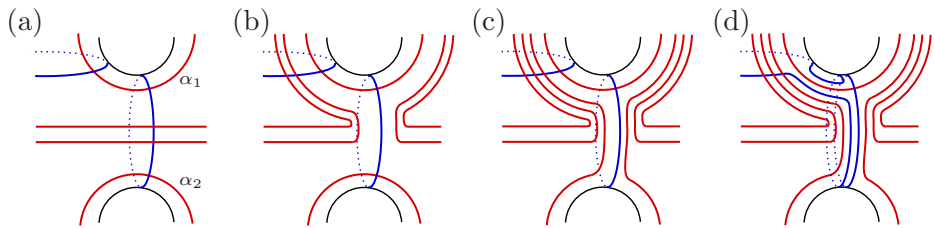
\begin{figure}
\labellist
 \pinlabel (a) at 2 115
\pinlabel (b) at 129 115 \pinlabel (c) at 262 115 \pinlabel (d) at 398 115
 \tiny \hair 2pt
 \pinlabel $\alpha_1$ at 95 82
 \pinlabel $\alpha_2$ at 95 25

\endlabellist
\includegraphics[width = 12cm]{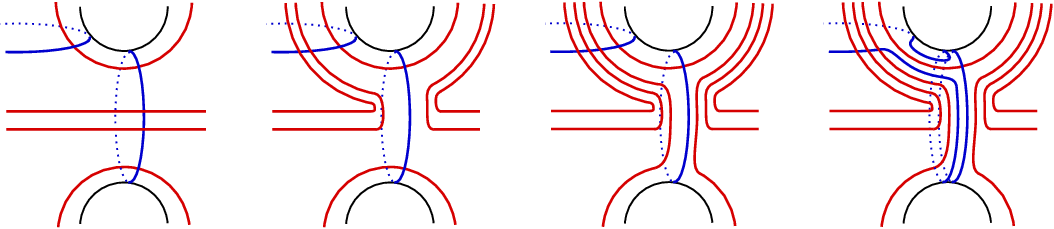}
\caption{A sequence of handleslides; in this example, there is one ladybug
$\bm\alpha$ curve passing through the vertical curves $\eta_{c_j}(I)$,
$\eta_{c_j}(I^1)$, $\eta_{c_j}(I^2)$ and $\eta_{c_j}(I'')$. For convenience, we
have not shown the $\bm\eta(I^1)$, $\bm\eta(I^2)$ or $\bm\eta(I'')$ curves
here; they are simply small translates of the blue $\bm\eta(I)$ curves in this
picture.} \label{fig:hmoves2}
\end{figure}

The genus 3 multi-diagram so obtained is the one we would associate to the planar diagram $\DD_J$, following Section \ref{sec:heegaard}. Let us refer to this multi-diagram as $\hat\HH^3_{I,I''}$. There are two cases to consider. If the smoothing of $c_{j_1}$ in $\DD_I$ connects the white regions --- i.e., $\gamma_{I^1} = \gamma_I \cup e_{j_1}$ --- then $\hat\HH^3_{I,I''}$ is isotopic to the multi-diagram
\[
\HH^3_{I,I''}=(\Sigma_3,\bm a,\bm\beta,\bm\gamma,\bm\delta,\bm\epsilon,\O,\X),
\]
depicted in Figure \ref{fig:genus3a}, where $\bm a$, $\bm\beta$, $\bm\gamma$, $\bm\delta$, and $\bm\epsilon$ are the images of the tuples $\bm\alpha$, $\bm\eta(I)$, $\bm\eta(I^1)$, $\bm\eta(I^2)$ and $\bm\eta(I'')$, respectively, after these Heegaard moves. On the other hand, if the smoothing of $c_{j_2}$ connects the black regions --- i.e., $\gamma_{I^1} = \gamma_I \minus e_{j_1}$ --- then $\hat\HH^3_{I,I''}$ is isotopic to the multi-diagram in Figure \ref{fig:genus3b}, also denoted by $\HH^3_{I,I''}$. (Note that in either case, the ladybug curves in $\HH^3_{I,I''}$ are stretched just enough to achieve admissibility, rather than all the way to the region containing $X_1$ as in the definition of $\HH$.)
We shall distinguish these two cases using the number $\nu = \nu_{I,I''}$, defined to be $1$ in the first case and $0$ in the second, as in Section \ref{sec:complex}.

\begin{figure}
\labellist
 \tiny \hair 2pt
 \pinlabel \rotatebox{20}{$r_2$} at 29 112
 \pinlabel \rotatebox{25}{$O_2$} at 36 98
 \pinlabel \rotatebox{25}{$X_2$} at 53 71
 \pinlabel \rotatebox{30}{$r_3$} at 64 58
 \pinlabel \rotatebox{85}{$r_a$} at 113 30
 \pinlabel \rotatebox{98}{$O_a$} at 129 28
 \pinlabel \rotatebox{98}{$X_a$} at 161 28
 \pinlabel {$r_{a+1}$} at 211 64
 \pinlabel \rotatebox{15}{$O_{a+1}$} at 185 86
 \pinlabel \rotatebox{15}{$X_{a+1}$} at 175 122
 \pinlabel \rotatebox{-15}{$O_b$} at 174 160
 \pinlabel \rotatebox{-15}{$X_b$} at 183 192
 \pinlabel \rotatebox{15}{$r_{b+1}$} at 244 206
 \pinlabel \rotatebox{15}{$O_{b+1}$} at 250 192
 \pinlabel \rotatebox{15}{$X_{b+1}$} at 258 162
 \pinlabel \rotatebox{-15}{$O_c$} at 257 123
 \pinlabel \rotatebox{-15}{$X_c$} at 248 91
 \pinlabel \rotatebox{-5}{$r_{c+1}$} at 250 13
 \pinlabel \rotatebox{92}{$O_{c+1}$} at 267 29
 \pinlabel \rotatebox{92}{$X_{c+1}$} at 304 28
 \pinlabel \rotatebox{92}{$r_{c+2}$} at 324 29
 \pinlabel \rotatebox{-86}{$r_d$} at 320 257
 \pinlabel \rotatebox{-83}{$O_d$} at 303 259
 \pinlabel \rotatebox{-83}{$X_d$} at 270 259
 \pinlabel \rotatebox{-85}{$r_{d+1}$} at 179 259
 \pinlabel \rotatebox{-85}{$O_{d+1}$} at 161 259
 \pinlabel \rotatebox{-85}{$X_{d+1}$} at 129 259
 \pinlabel \rotatebox{-79}{$r_{d+2}$} at 112 256
 \pinlabel \rotatebox{-24}{$r_{m}$} at 60 227
 \pinlabel \rotatebox{-23}{$O_{m}$} at 51 214
 \pinlabel \rotatebox{-23}{$X_{m}$} at 36 187
 \pinlabel \rotatebox{-18}{$r_1$} at 28 171
 \pinlabel $O_1$ at 28 158
 \pinlabel $X_1$ at 28 127
\endlabellist
\includegraphics{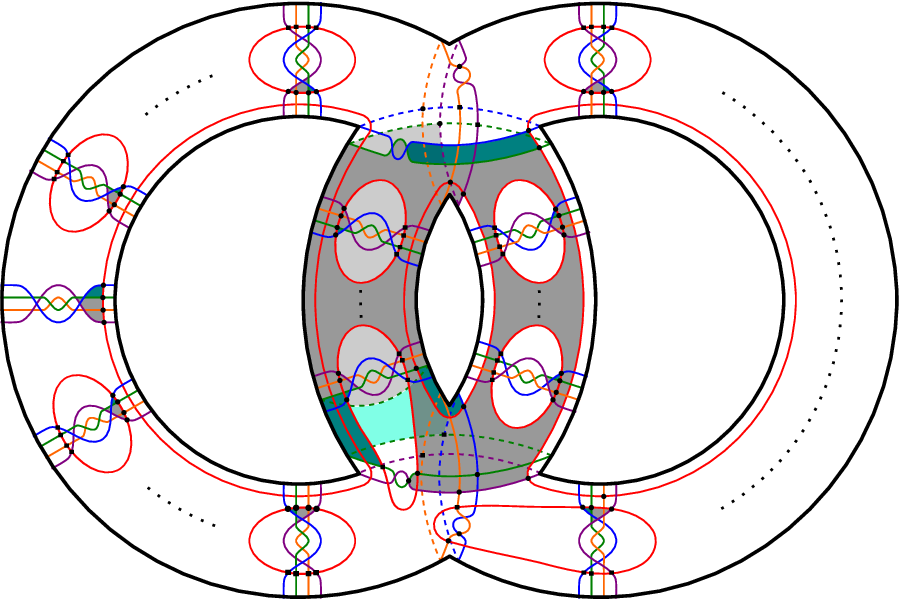}
\caption{$\HH^3_{I,I''}$ in the case that $\gamma_{I^1} = \gamma_I \cup e_{j_1}$. The tuples $\bm a$, $\bm\beta$, $\bm\gamma$, $\bm\delta$ and $\bm\epsilon$ are drawn in red, blue, green, orange, and purple, respectively. The crossing $c_{j_1}$ is on the top and $c_{j_2}$ is on the bottom. The shaded regions represent the domains of the triangle classes considered in the proof of Proposition \ref{prop:g21(w0)}. Points in $\A$ are labeled with their corresponding values of $\omega_\r$.} \label{fig:genus3a}
\end{figure}

\begin{figure}
\labellist
 \tiny \hair 2pt
 \pinlabel \rotatebox{20}{$r_2$} at 29 112
 \pinlabel \rotatebox{25}{$O_2$} at 36 98
 \pinlabel \rotatebox{25}{$X_2$} at 53 71
 \pinlabel \rotatebox{30}{$r_3$} at 64 58
 \pinlabel \rotatebox{85}{$r_a$} at 113 30
 \pinlabel \rotatebox{98}{$O_a$} at 129 28
 \pinlabel \rotatebox{98}{$X_a$} at 161 28
 \pinlabel $r_{c+1}$ at 211 64
 \pinlabel \rotatebox{15}{$O_{c+1}$} at 185 86
 \pinlabel \rotatebox{15}{$X_{c+1}$} at 175 122
 \pinlabel \rotatebox{-15}{$O_d$} at 174 160
 \pinlabel \rotatebox{-15}{$X_d$} at 183 192
 \pinlabel \rotatebox{15}{$r_{b+1}$} at 244 206
 \pinlabel \rotatebox{15}{$O_{b+1}$} at 250 192
 \pinlabel \rotatebox{15}{$X_{b+1}$} at 258 162
 \pinlabel \rotatebox{-15}{$O_c$} at 257 123
 \pinlabel \rotatebox{-15}{$X_c$} at 248 91
 \pinlabel \rotatebox{-5}{$r_{a+1}$} at 250 13
 \pinlabel \rotatebox{92}{$O_{a+1}$} at 267 29
 \pinlabel \rotatebox{92}{$X_{a+1}$} at 304 28
 \pinlabel \rotatebox{92}{$r_{a+2}$} at 324 29
 \pinlabel \rotatebox{-86}{$r_b$} at 320 257
 \pinlabel \rotatebox{-83}{$O_b$} at 303 259
 \pinlabel \rotatebox{-83}{$X_b$} at 270 259
 \pinlabel \rotatebox{-85}{$r_{d+1}$} at 179 259
 \pinlabel \rotatebox{-85}{$O_{d+1}$} at 161 259
 \pinlabel \rotatebox{-85}{$X_{d+1}$} at 129 259
 \pinlabel \rotatebox{-79}{$r_{d+2}$} at 112 256
 \pinlabel \rotatebox{-24}{$r_{m}$} at 60 227
 \pinlabel \rotatebox{-23}{$O_{m}$} at 51 214
 \pinlabel \rotatebox{-23}{$X_{m}$} at 36 187
 \pinlabel \rotatebox{-18}{$r_1$} at 28 171
 \pinlabel $O_1$ at 28 158
 \pinlabel $X_1$ at 28 127
\endlabellist
\includegraphics{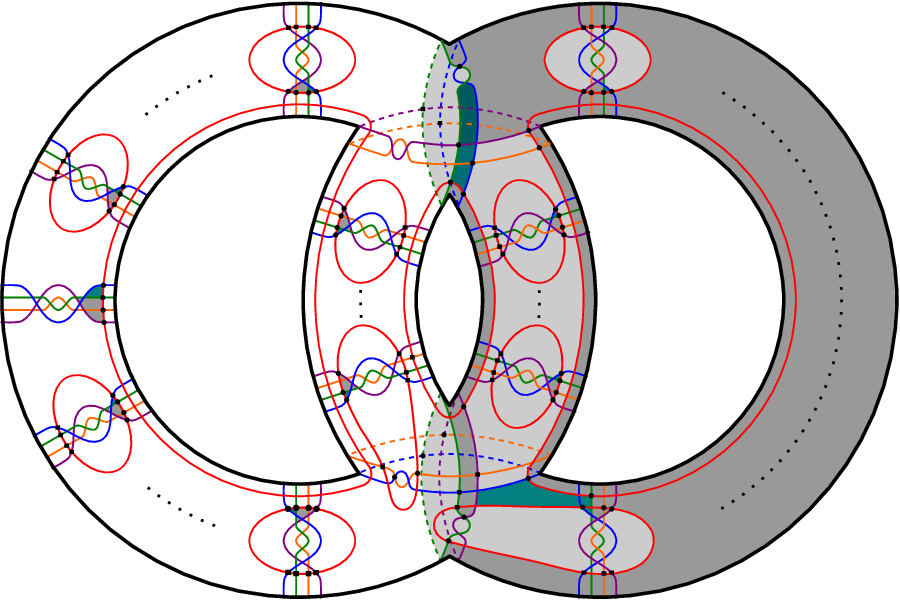}
\caption{$\HH^3_{I,I''}$ in the case that $\gamma_{I^1} = \gamma_I \minus e_{j_1}$. } \label{fig:genus3b}
\end{figure}

In Figures \ref{fig:genus3a} and \ref{fig:genus3b}, we have indicated, by
circles and squares, some intersection points between the $\bm a$ curves and
the $\bm \beta$, $\bm\gamma$, $\bm \delta$ and $\bm\epsilon$ curves. For each
$i=2,\dots,m$, let $w_i$ (resp. $x_i$, $y_i$ and $z_i$) be the circular
intersection point between $a_{p_i}$ and some $\bm \beta$ (resp. $\bm\gamma$,
$\bm\delta$ and $\bm\epsilon$) curve, and let $w_i'$ (resp. $x_i'$, $y_i'$ and
$z_i'$) be the square intersection point between $a_{p_i}$ and the same $\bm
\beta$ (resp. $\bm\gamma$, $\bm\delta$ and $\bm\epsilon$) curve. Note that
every point of $\T_ a \cap \T_\beta$ (resp. $\T_ a \cap \T_\gamma$, $\T_ a \cap
\T_\delta$ and $\T_ a \cap \T_\epsilon$) contains either $w_i$ or $w_i'$ (resp.
$x_i$ or $x_i'$, $y_i$ or $y_i'$ and $z_i$ or $z_i'$) for $i=2,\dots,m$, and
\[ \abs{\T_ a \cap \T_\beta} = \abs{\T_ a \cap \T_\gamma} = \abs{\T_ a \cap
\T_\delta} = \abs{\T_ a \cap \T_\epsilon} = 2^{m-1}.
\]
By construction, the unique point $\w_0$ (resp. $\x_0$, $\y_0$ and $\z_0$) of
$\T_{ a}\cap\T_{\beta}$ (resp. $\T_ a \cap \T_\gamma$, $\T_ a \cap \T_\delta$
and $\T_ a \cap \T_\epsilon$) in the top Maslov grading contains all of the
$w_i$ (resp. $x_i$, $y_i$ and $z_i$). For $2, \dots, m$, let $\w_i$ be the
generator obtained from $\w_0$ by replacing $w_i$ by $w_i'$, and define $\x_i$,
$\y_i$, and $\z_i$ similarly. These constitute all of the generators in the
second-to-top Maslov gradings of their respective complexes.

The two lemmas below are easy exercises in counting holomorphic disks; compare with Proposition \ref{prop:unlink} and Lemma \ref{lem:psitheta1}.

\begin{lemma} \label{lemma:differential}
The differentials on $\CFKtt(\bm a, \bm\beta)$, $\CFKtt(\bm a,
\bm\epsilon)$, $\CFKtt(\bm\beta, \bm\gamma)$, $\CFKtt(\bm\beta, \bm\delta)$,
$\CFKtt(\bm\gamma, \bm\epsilon)$, and $\CFKtt(\bm\delta, \bm\epsilon)$ are all
zero. The differentials on $\CFKtt(\bm a, \bm\gamma)$ and
$\CFKtt(\bm a, \bm\delta)$ are given by
\begin{align*}
{\partial}_{ a\gamma}(\x) &=
  \begin{cases}
  (1 + T^{B+C}) (\x \minus \{x_{a+1}\} \cup \{x'_{a+1}\}) & x_{a+1} \in \x \\
  0 & x_{a+1} \not\in \x
  \end{cases} \\
{\partial}_{ a\delta}(\y) &=
  \begin{cases}
  (1 + T^{C+D}) (\y \minus \{y_{c+1}\} \cup \{y'_{c+1}\}) & y_{c+1} \in \y \\
  0 & y_{c+1} \not\in \y.
  \end{cases}
\end{align*}
\end{lemma}

\begin{lemma} \label{lem:psitheta2}
We have
\begin{equation*} \label{eq:psitheta2}
\begin{aligned}
\psi_i^{\beta\gamma} (\Theta_1^{\beta\gamma}) &=
\begin{cases}
\Theta_2^{\beta\gamma} & i \in \{a,c\}  \\
0 & \text{otherwise}
\end{cases}
&{\Psi}_i^{\gamma\epsilon} (\Theta_1^{\gamma\epsilon}) &=
\begin{cases}
\Theta_2^{\gamma\epsilon} & i \in \{b,d\} \\
0 & \text{otherwise}
\end{cases}
\\
{\Psi}_i^{\beta\delta} (\Theta_1^{\beta\delta}) &=
\begin{cases}
\Theta_2^{\beta\delta} & i \in \{b,d\} \\
0 & \text{otherwise}
\end{cases}
&{\Psi}_i^{\delta\epsilon} (\Theta_1^{\delta\epsilon}) &=
\begin{cases}
\Theta_2^{\delta\epsilon} & i \in \{a,c\} \\
0 & \text{otherwise},
\end{cases}
\\
\end{aligned}
\end{equation*}
while $\Psi_i^{\beta\gamma}(\Theta_2^{\beta\gamma})=\Psi_i^{\gamma\epsilon}(\Theta_2^{\gamma\epsilon})=\Psi_i^{\beta\delta}(\Theta_2^{\beta\delta})=\Psi_i^{\delta\epsilon}(\Theta_2^{\delta\epsilon})=0$ for $i=1,\dots,m$.
\end{lemma}

The genericity of $\r$ implies that $B+C$ and $C+D$ are nonzero, so
\begin{equation}
\label{eqn:vanishing}
H_*(\CFKtt(\bm a, \bm\gamma; \FF), \partial_{ a\gamma}) =
H_*(\CFKtt(\bm a, \bm\delta; \FF), \partial_{ a\delta}) = 0;
\end{equation}
compare with Lemma \ref{lemma:generic}.
In other words, every cycle in $\CFKtt(\bm a, \bm\gamma; \FF)$ or
$\CFKtt(\bm a, \bm\delta; \FF)$ is a boundary, so we may define maps
$\partial_{ a\gamma}^{-1}$ and $\partial_{ a\delta}^{-1}$ up to
addition of cycles. Indeed, for any $\x \in \T_ a \cap \T_\gamma$
containing $x'_{a+1}$, we may take $\partial_{ a\gamma}^{-1}(\x) =
(1+T^{B+C})^{-1} (\x \minus \{x'_{a+1}\} \cup \{x_{a+1}\})$, and define
$\partial_{ a\delta}^{-1}$ similarly.

Let $f_{ a\beta\gamma}$, $f_{ a\beta\delta}$,
$f_{ a\gamma\epsilon}$, $f_{ a\delta\epsilon}$,
$f_{ a\beta\gamma\epsilon}$, and $f_{ a\beta\delta\epsilon}$ be the maps defined by
\[
\begin{aligned}
f_{ a\beta\gamma}(\x) &= F_{ a\beta\gamma} (\x \otimes
 (\Theta_1^{\beta\gamma} + \Theta_2^{\beta\gamma})), \\
f_{ a\beta\gamma\epsilon}(\x) &= F_{ a\beta\gamma\epsilon} (\x \otimes
 (\Theta_1^{\beta\gamma} + \Theta_2^{\beta\gamma}) \otimes
 (\Theta_1^{\gamma\epsilon} + \Theta_2^{\gamma\epsilon})),
\end{aligned}
\]
and so on. According to the discussion in Section \ref{sec:cube}, these maps
fit into a commutative diagram,
\begin{equation} \label{eq:filtered}
\xy
 {(30,0)*{\CFKtt(\bm a, \bm\gamma;\FF)}="g"};
 {(0,-15)*{\CFKtt(\bm a, \bm\beta;\FF)}="b"};
 {(60,-15)*{\CFKtt(\bm a, \bm\epsilon;\FF)}="e"};
 {(30,-30)*{\CFKtt(\bm a, \bm\delta;\FF),}="d"};
 {\ar^{f_{ a\beta\gamma}} "b"; "g"};
 {\ar^{f_{ a\gamma\epsilon}} "g"; "e"};
 {\ar_{f_{ a\beta\delta}} "b"; "d"};
 {\ar_{f_{ a\delta\epsilon}} "d"; "e"};
 {\ar^{f_{ a\beta\gamma\epsilon}+f_{ a\beta\delta\epsilon}} "b"; "e"};
 {\ar@/^12pt/^{\partial_{ a\gamma}} "g"+<-6pt,8pt>; "g"+<6pt,8pt>};
 {\ar@/_12pt/_{\partial_{ a\delta}} "d"+<-6pt,-8pt>; "d"+<6pt,-8pt>};
 \endxy
\end{equation}
which may be viewed as a filtered complex, where the filtration is by
horizontal position. Let $\mathcal{S}_{I,I''}$ denote the spectral sequence associated to
this filtered complex. Its $d_1$ differential vanishes due to \eqref{eqn:vanishing}. Since the differentials $\partial_{ a\beta}$ and $\partial_{ a\epsilon}$ are also zero, we may identify the complexes
$\CFKtt(\bm a, \bm\beta;\FF)$ and $\CFKtt(\bm a, \bm\epsilon;\FF)$ with their homologies. With respect to this identification, $E_2(\mathcal{S}_{I,I''})$
is the mapping cone
\begin{equation} \label{eq:E2}
\xymatrix{ \CFKtt(\bm a, \bm\beta; \FF) \ar[r]^{g} & \CFKtt(\bm a,
\bm\epsilon; \FF), }
\end{equation}
where, for any generator $\w \in \CFKtt(\bm a, \bm\beta;\FF)$, we have
\begin{equation*} \label{eq:g(w)}
\begin{aligned}
g(\w) &= (f_{ a\beta\gamma\epsilon} +
  f_{ a\beta\delta\epsilon} +
  f_{ a\gamma\epsilon} \circ \partial_{ a\gamma}^{-1} \circ
  f_{ a\beta\gamma} + f_{ a\delta\epsilon} \circ
  \partial_{ a\delta}^{-1} \circ f_{ a\beta\delta})(\w) \\
&= F_{ a\beta\gamma\epsilon} (\w \otimes (\Theta_1^{\beta\gamma}+
  \Theta_2^{\beta\gamma}) \otimes (\Theta_1^{\gamma\epsilon} +
  \Theta_2^{\gamma\epsilon}))
 + F_{ a\beta\delta\epsilon} (\w \otimes (\Theta_1^{\beta\delta} +
  \Theta_2^{\beta\delta}) \otimes (\Theta_1^{\delta\epsilon} +
  \Theta_2^{\delta\epsilon})) \\
 & \quad +  F_{ a\gamma\epsilon}(\partial_{ a\gamma}^{-1}
 F_{ a\beta\gamma}(\w \otimes (\Theta_1^{\beta\gamma} +
 \Theta_2^{\beta\gamma})) \otimes (\Theta_1^{\gamma\epsilon} +
 \Theta_2^{\gamma\epsilon})) \\
 & \quad + F_{ a\delta\epsilon}(\partial_{ a\delta}^{-1}
 F_{ a\beta\delta}(\w \otimes (\Theta_1^{\beta\delta} +
 \Theta_2^{\beta\delta})) \otimes (\Theta_1^{\delta\epsilon} +
 \Theta_2^{\delta\epsilon})).
\end{aligned}
\end{equation*}
Our goal, then, will be to understand the map $g$.

\subsection{Commutation with the basepoint action} \label{sec:commutation}

For $k, l \in \{1,2\}$, define
\begin{multline} \label{eq:gkl(w)}
g^{k,l}(\w) =
 F_{ a\beta\gamma\epsilon} (\w \otimes \Theta_k^{\beta\gamma} \otimes \Theta_l^{\gamma\epsilon}) +
 F_{ a\beta\delta\epsilon} (\w \otimes \Theta_l^{\beta\delta} \otimes \Theta_k^{\delta\epsilon}) \\
 +  F_{ a\gamma\epsilon}(\partial_{ a\gamma}^{-1}
 F_{ a\beta\gamma}(\w \otimes \Theta_k^{\beta\gamma}) \otimes
 \Theta_l^{\gamma\epsilon}) +
 F_{ a\delta\epsilon}(\partial_{ a\delta}^{-1}
 F_{ a\beta\delta}(\w \otimes \Theta_l^{\beta\delta}) \otimes
 \Theta_k^{\delta\epsilon}).
\end{multline}
Clearly, $g = g^{1,1}+g^{1,2}+ g^{2,1} + g^{2,2}$. The lemma below follows directly from Proposition \ref{prop:maslovshift}.

\begin{lemma}
\label{lem:homogeneousg} The maps $g^{k,l}$ are each
homogeneous with respect to the Maslov grading. Moreover, the Maslov grading shifts of these maps are related by
\begin{equation*} \label{eq:grshift}
M(g^{1,1}) = M(g^{1,2})+1 = M(g^{2,1})+1= M(g^{2,2})+2.
\end{equation*}
\end{lemma}


This decomposition enables us to understand how $g$ interacts with
the maps $\Psi_i^{ a\beta}$ and $\Psi_i^{ a\epsilon}$.

\begin{proposition} \label{prop:psi-gkl}
For any $i =1,\dots,m$ and $k,l \in \{1,2\}$, we have $g^{k,l} \circ
\Psi^{ a\beta}_i = \Psi^{ a\epsilon}_i \circ g^{k,l}$, with the
following exceptions:
\begin{enumerate}
\item If $i=a$ or $i=c$, then $g^{1,l} \circ \Psi^{ a\beta}_i =
\Psi^{ a\epsilon}_i \circ g^{1,l} + g^{2,l}$.
\item If $i=b$ or $i=d$, then $g^{k,1} \circ \Psi^{ a\beta}_i =
\Psi^{ a\epsilon}_i \circ g^{k,1} + g^{k,2}$.
\end{enumerate}
\end{proposition}

\begin{proof}[Proof of Proposition \ref{prop:psi-gkl}]
From the $\mathcal{A}_{\infty}$ relation \eqref{eqn:Ainfty2} and the fact that $\partial_{ a\beta} = \partial_{ a\epsilon}=0$, we have, for each $\w \in
\CFKtt(\bm a,\bm\beta;\FF)$ and $\x \in \CFKtt(\bm a, \bm\gamma;\FF)$, that
\begin{align}
\label{eq:abg-assoc}
 F_{ a\beta\gamma}( \Psi_i^{ a\beta}(\w), \Theta_k^{\beta\gamma}) &=
 F_{ a\beta\gamma}( \w,
  \Psi_i^{\beta\gamma}(\Theta_k^{\beta\gamma}))
 + \Psi_i^{ a\gamma}(F_{ a\beta\gamma}( \w, \Theta_k^{\beta\gamma}))
 + \partial_{ a\gamma}(\Psi_i^{ a\beta\gamma}( \w,
   \Theta_k^{\beta\gamma})), \\
\label{eq:age-assoc}
 F_{ a\gamma\epsilon}( \Psi_i^{ a\gamma}(\x), \Theta_l^{\gamma\epsilon}) &=
 F_{ a\gamma\epsilon}( \x,
  \Psi_i^{\gamma\epsilon}(\Theta_l^{\gamma\epsilon}))
 + \Psi_i^{ a\epsilon}(F_{ a\gamma\epsilon}( \x, \Theta_l^{\gamma\epsilon}))
 + \Psi_i^{ a\gamma\epsilon}( \partial_{ a\gamma}(\x),
   \Theta_k^{\gamma\epsilon}).
\end{align}
Applying $F_{ a\gamma\epsilon}( \partial_{ a\gamma}^{-1}(\cdot),
\Theta_l^{\gamma\epsilon})$ to both sides of \eqref{eq:abg-assoc}, we have that
\begin{equation}
\begin{aligned}\label{agedelta}
 F_{ a\gamma\epsilon} & ( \partial_{ a\gamma}^{-1} (F_{ a\beta\gamma}(
  \Psi_i^{ a\beta}(\w), \Theta_k^{\beta\gamma})),
  \Theta_l^{\gamma\epsilon}) =
 F_{ a\gamma\epsilon}( \partial_{ a\gamma}^{-1} (F_{ a\beta\gamma}( \w,
  \Psi_i^{\beta\gamma}(\Theta_k^{\beta\gamma}))), \Theta_l^{\gamma\epsilon}) \\
 & +  F_{ a\gamma\epsilon}( \partial_{ a\gamma}^{-1}(  \Psi_i^{ a\gamma}(F_{ a\beta\gamma}( \w,
 \Theta_k^{\beta\gamma}))), \Theta_l^{\gamma\epsilon})
 + F_{ a\gamma\epsilon}(\Psi_i^{ a\beta\gamma}( \w,
   \Theta_k^{\beta\gamma}), \Theta_l^{\gamma\epsilon}).
\end{aligned}
\end{equation}
In the second term on the right of \eqref{agedelta}, we may commute $\partial_{ a\gamma}^{-1}$
past $\Psi_i^{ a\gamma}$. We substitute $\x =
\partial_{ a\gamma}^{-1}( F_{ a\beta\gamma}( \w,
 \Theta_k^{\beta\gamma}))$ into \eqref{eq:age-assoc} to obtain
\begin{equation} \label{eq:zigzag-assoc}
\begin{aligned}
 F_{ a\gamma\epsilon} & ( \partial_{ a\gamma}^{-1} (F_{ a\beta\gamma}(
  \Psi_i^{ a\beta}(\w), \Theta_k^{\beta\gamma})),
  \Theta_l^{\gamma\epsilon}) =
 F_{ a\gamma\epsilon}( \partial_{ a\gamma}^{-1} (F_{ a\beta\gamma}( \w,
  \Psi_i^{\beta\gamma}(\Theta_k^{\beta\gamma}))), \Theta_l^{\gamma\epsilon}) \\
 & + F_{ a\gamma\epsilon}( \partial_{ a\gamma}^{-1}( F_{ a\beta\gamma}( \w,
  \Theta_k^{\beta\gamma})),  \Psi_i^{\gamma\epsilon}(\Theta_l^{\gamma\epsilon}))
 + \Psi_i^{ a\epsilon}(F_{ a\gamma\epsilon}( \partial_{ a\gamma}^{-1}( F_{ a\beta\gamma}( \w,
  \Theta_k^{\beta\gamma})) , \Theta_l^{\gamma\epsilon})) \\
 & + \Psi_i^{ a\gamma\epsilon}( F_{ a\beta\gamma}( \w, \Theta_k^{\beta\gamma}),
   \Theta_k^{\gamma\epsilon})
 + F_{ a\gamma\epsilon}(\Psi_i^{ a\beta\gamma}( \w,
   \Theta_k^{\beta\gamma}), \Theta_l^{\gamma\epsilon}).
\end{aligned}
\end{equation}
Similarly, the $\mathcal{A}_{\infty}$ relation for the quadrilateral-counting
maps yields
\begin{equation} \label{eq:quad-assoc}
\begin{aligned}
 F_{ a\beta\gamma\epsilon} & (\Psi^{ a\beta}(\w),
  \Theta_k^{\beta\gamma}, \Theta_l^{\gamma\epsilon}) =
 F_{ a\beta\gamma\epsilon} (\w, \Psi_i^{\beta\gamma}
  (\Theta_k^{\beta\gamma}), \Theta_l^{\gamma\epsilon})
+F_{ a\beta\gamma\epsilon} (\w, \Theta_k^{\beta\gamma}
, \Psi_i^{\gamma\epsilon}(\Theta_l^{\gamma\epsilon}))  \\
& +\Psi_i^{ a\epsilon} (F_{ a\beta\gamma\epsilon} (\w,
  \Theta_k^{\beta\gamma}, \Theta_l^{\gamma\epsilon}))
+F_{ a\gamma\epsilon}(\Psi_i^{ a\beta\gamma}(\w,
  \Theta_k^{\beta\gamma}), \Theta_l^{\gamma\epsilon})
 +F_{ a\beta\epsilon}(\w, \Psi_i^{\beta\gamma\epsilon}
  (\Theta_k^{\beta\gamma}, \Theta_l^{\gamma\epsilon})) \\
& +\Psi_i^{ a\gamma\epsilon}(F_{ a\beta\gamma}(\w,
  \Theta_k^{\beta\gamma}), \Theta_l^{\gamma\epsilon})
 +\Psi_i^{ a\beta\epsilon}(\w, F_{\beta\gamma\epsilon}
  (\Theta_k^{\beta\gamma}, \Theta_l^{\gamma\epsilon})).
\end{aligned}
\end{equation}
Adding \eqref{eq:zigzag-assoc} and \eqref{eq:quad-assoc} and canceling terms,
we have
\begin{equation} \label{eq:abge-assoc}
\begin{aligned}
 F_{ a\gamma\epsilon} &( \partial_{ a\gamma}^{-1} (F_{ a\beta\gamma}(
  \Psi_i^{ a\beta}(\w), \Theta_k^{\beta\gamma})),
  \Theta_l^{\gamma\epsilon}) +
 F_{ a\beta\gamma\epsilon} (\Psi^{ a\beta}(\w),
  \Theta_k^{\beta\gamma}, \Theta_l^{\gamma\epsilon}) = \\
 &\Psi_i^{ a\epsilon}(
  F_{ a\gamma\epsilon}( \partial_{ a\gamma}^{-1}( F_{ a\beta\gamma}( \w,
   \Theta_k^{\beta\gamma})) , \Theta_l^{\gamma\epsilon})
  +F_{ a\beta\gamma\epsilon} (\w,
  \Theta_k^{\beta\gamma}, \Theta_l^{\gamma\epsilon})) \\
 &+F_{ a\gamma\epsilon}( \partial_{ a\gamma}^{-1} (F_{ a\beta\gamma}( \w,
  \Psi_i^{\beta\gamma}(\Theta_k^{\beta\gamma}))), \Theta_l^{\gamma\epsilon})
 + F_{ a\beta\gamma\epsilon} (\w, \Psi_i^{\beta\gamma} (\Theta_k^{\beta\gamma}), \Theta_l^{\gamma\epsilon}) \\
 &+ F_{ a\gamma\epsilon}( \partial_{ a\gamma}^{-1}( F_{ a\beta\gamma}(
  \w,  \Theta_k^{\beta\gamma})),  \Psi_i^{\gamma\epsilon}(\Theta_l^{\gamma\epsilon}))
 +F_{ a\beta\gamma\epsilon} (\w, \Theta_k^{\beta\gamma}, \Psi_i^{\gamma\epsilon}(\Theta_l^{\gamma\epsilon}))  \\
 &+F_{ a\beta\epsilon}(\w, \Psi_i^{\beta\gamma\epsilon}
  (\Theta_k^{\beta\gamma}, \Theta_l^{\gamma\epsilon}))
 +\Psi_i^{ a\beta\epsilon}(\w, F_{\beta\gamma\epsilon}
  (\Theta_k^{\beta\gamma}, \Theta_l^{\gamma\epsilon})).
\end{aligned}
\end{equation}
Similarly,
\begin{equation} \label{eq:abde-assoc}
\begin{aligned}
 F_{ a\delta\epsilon} & ( \partial_{ a\delta}^{-1} (F_{ a\beta\delta}(
  \Psi_i^{ a\beta}(\w), \Theta_l^{\beta\delta})),
  \Theta_k^{\delta\epsilon}) +
 F_{ a\beta\delta\epsilon} (\Psi^{ a\beta}(\w),
  \Theta_l^{\beta\delta}, \Theta_k^{\delta\epsilon}) = \\
 & \Psi_i^{ a\epsilon}(
  F_{ a\delta\epsilon}( \partial_{ a\delta}^{-1}( F_{ a\beta\delta}( \w,
   \Theta_l^{\beta\delta})) , \Theta_k^{\delta\epsilon})
  +F_{ a\beta\delta\epsilon} (\w,
  \Theta_l^{\beta\delta}, \Theta_k^{\delta\epsilon})) \\
 & + F_{ a\delta\epsilon}( \partial_{ a\delta}^{-1}( F_{ a\beta\delta}(
  \w,  \Theta_l^{\beta\delta})),  \Psi_i^{\delta\epsilon}(\Theta_k^{\delta\epsilon}))
 +F_{ a\beta\delta\epsilon} (\w, \Theta_l^{\beta\delta}, \Psi_i^{\delta\epsilon}(\Theta_k^{\delta\epsilon}))  \\
 &+F_{ a\delta\epsilon}( \partial_{ a\delta}^{-1} (F_{ a\beta\delta}( \w,
  \Psi_i^{\beta\delta}(\Theta_l^{\beta\delta}))), \Theta_k^{\delta\epsilon})
 + F_{ a\beta\delta\epsilon} (\w, \Psi_i^{\beta\delta} (\Theta_l^{\beta\delta}), \Theta_k^{\delta\epsilon}) \\
 &+F_{ a\beta\epsilon}(\w, \Psi_i^{\beta\delta\epsilon}
  (\Theta_l^{\beta\delta}, \Theta_k^{\delta\epsilon}))
 +\Psi_i^{ a\beta\epsilon}(\w, F_{\beta\delta\epsilon}
  (\Theta_l^{\beta\delta}, \Theta_k^{\delta\epsilon})).
\end{aligned}
\end{equation}

The first lines in \eqref{eq:abge-assoc} and \eqref{eq:abde-assoc} sum to
$g^{k,l}(\Psi_i^{ a\beta}(\w))$, and the sum of the second lines equals
$\Psi_i^{ a\epsilon}(g^{k,l}(\w))$. By Lemma \ref{lem:psitheta2}, if $k=1$ and
$i \in \{b,d\}$, then the sum of the third lines of \eqref{eq:abge-assoc} and
\eqref{eq:abde-assoc} equals $g^{2,l}(\w)$; otherwise, it equals zero.
Similarly, the sum of the fourth lines equals $g^{k,2}(\w)$ if $l=1$ and $i \in
\{a,c\}$, and zero otherwise. Thus, to finish the proof of Proposition
\ref{prop:psi-gkl}, we only need to show that the sum of the fifth lines is
zero; that is,
\begin{multline} \label{eq:fifthlines}
F_{ a\beta\epsilon}(w, \Psi_i^{\beta\gamma\epsilon}
  (\Theta_k^{\beta\gamma}, \Theta_l^{\gamma\epsilon}) + \Psi_i^{\beta\delta\epsilon}
  (\Theta_l^{\beta\delta}, \Theta_k^{\delta\epsilon})) \\
+\Psi_i^{ a\beta\epsilon}(w, F_{\beta\gamma\epsilon}
  (\Theta_k^{\beta\gamma}, \Theta_l^{\gamma\epsilon})+ F_{\beta\delta\epsilon}
  (\Theta_l^{\beta\delta}, \Theta_k^{\delta\epsilon})) = 0.
\end{multline}
This follows from an argument nearly identical to those in the proofs of Lemma
\ref{lem:triangles} and Proposition \ref{prop:d^2=0}. Let
$\Theta_{k,l}^{\beta\epsilon} \in \T_\beta \cap \T_\epsilon$ denote the
generator consisting of the point of $\beta_{c_{j_2}} \cap \epsilon_{c_{j_2}}$
nearest the point of $\beta_{c_{j_2}} \cap \gamma_{c_{j_2}}$ in
$\Theta_{k}^{\beta\gamma}$; the point of $\beta_{c_{j_1}} \cap
\epsilon_{c_{j_1}}$ nearest the point of $\gamma_{c_{j_1}} \cap
\epsilon_{c_{j_1}}$ in $\Theta_{l}^{\gamma\epsilon}$; and, for each
$i=1,\dots,m$, the intersection point $\theta_{p_i}^{\beta\epsilon}$ of
$\beta_{p_i}\cap\epsilon_{p_i}$ with smallest $\delta$-grading contribution, as
in Section \ref{sec:heegaard}. As in Lemma \ref{lem:isom},
$\Theta_{1,1}^{\beta\epsilon}$, $\Theta_{2,1}^{\beta\epsilon}$,
$\Theta_{1,2}^{\beta\epsilon}$, and $\Theta_{2,2}^{\beta\epsilon}$ are the
generators in $\T_\beta \cap \T_\epsilon$ with minimal $\delta$-grading.
Moreover, it is not hard to see that \[ F_{\beta\gamma\epsilon}
(\Theta_k^{\beta\gamma}, \Theta_l^{\gamma\epsilon})= F_{\beta\delta\epsilon}
(\Theta_l^{\beta\delta}, \Theta_k^{\delta\epsilon}) =
\Theta^{\beta\epsilon}_{k,l};
\]
the domains that contribute to these maps are simply disjoint unions of small triangles. Furthermore, the $\delta$-grading shifts of $\Psi_i^{\beta\gamma\epsilon}$ and
$\Psi_i^{\beta\gamma\epsilon}$ are one less than those of
$F_{\beta\gamma\epsilon}$ and $F_{\beta\delta\epsilon}$, respectively. Thus,
\[
\Psi_i^{\beta\gamma\epsilon} (\Theta_k^{\beta\gamma},
\Theta_l^{\gamma\epsilon})= \Psi_i^{\beta\delta\epsilon}
(\Theta_l^{\beta\delta}, \Theta_k^{\delta\epsilon}) = 0,
\]
and both terms on the left side of \eqref{eq:fifthlines} vanish.
\end{proof}

Next, we describe the actions of the maps $\Psi_i^{ a\beta}$ and $\Psi_i^{
a\epsilon}$. The diagrams $(\Sigma_3,\bm a,\bm\beta,\O,\X)$ and $(\Sigma_3,\bm
a,\bm\epsilon,\O,\X)$ both satisfy the hypotheses of Proposition
\ref{prop:twistedunknot} with respect to the marking $(\A,\omega_{\r})$.
Therefore, without any direct computation, we know that
\begin{equation} \label{eq:psiabrel}
 \sum_{i=1}^m T^{R(1,i)} \Psi^{ a\beta}_i = 0,
\end{equation}
\begin{multline} \label{eq:psiaerel}
\sum_{i \in \{1,\dots,a, d+1, \dots, m\}} T^{R(1,i)} \Psi^{ a\epsilon}_i
 + \sum_{i=c+1}^d T^{A + R(c+1,i)} \Psi^{ a\epsilon}_i \\
 + \sum_{i=b+1}^c T^{A + D + R(b+1,i)} \Psi^{ a\epsilon}_i
 + \sum_{i=a+1}^b T^{A + D + C + R(a+1,i)} \Psi^{ a\epsilon}_i = 0.
\end{multline}
Any element of $\CFKtt(\bm a, \bm\beta;\FF)$ may be obtained from $\w_0$
through a sum of compositions of the $\Psi_i^{ a\beta}$ maps, by Proposition
\ref{prop:twistedunknot}. Therefore, by Proposition \ref{prop:psi-gkl}, the
values $g^{1,1}(\w_0)$, $g^{1,2}(\w_0)$, $g^{2,1}(\w_0)$ and $g^{2,2}(\w_0)$
determine the entire function $g$. We shall see momentarily that
$g^{2,1}(\w_0)$ is a nonzero multiple of $\z_0$, the unique generator of
$\CFKtt(\bm a, \bm\epsilon;\FF)$ in the top Maslov grading. It follows that the
other values $g^{k,l}(\w_0)$ are completely determined by $g^{2,1}(\w_0)$.
Indeed, it must be the case that $g^{1,1}(\w_0) = 0$ by Lemma
\ref{lem:homogeneousg}. Next, by \eqref{eq:psiabrel} and Proposition
\ref{prop:psi-gkl}, we have
\begin{equation} \label{eq:g12(w0)}
\begin{aligned}
0 &= g \left(  \sum_{i=1}^m T^{R(1,i)} \Psi^{ a\beta}_i(\w_0) \right) \\
&= \sum_{i=1}^m T^{R(1,i)} (g^{1,1} + g^{1,2} + g^{2,1} +
g^{2,2})(\Psi^{ a\beta}_i(\w_0)) \\
&= \sum_{i=1}^m T^{R(1,i)} \Psi^{ a\epsilon}_i( (g^{1,1} + g^{1,2} +
g^{2,1} + g^{2,2})(\w_0)) \\
& \qquad + (T^{A}+T^{A+B+C})( g^{2,1}+ g^{2,2})(\w_0) +
(T^{A+B}+T^{A+B+C+D})( g^{1,2}+ g^{2,2})(\w_0)
\end{aligned}
\end{equation}
The sum of the terms in the top Maslov grading must equal zero, so
\[
(T^A+T^{A+B+C})  g^{2,1}(\w_0) + (T^{A+B}+T^{A+B+C+D})  g^{1,2}(\w_0) = 0, \\
\]
which determines $ g^{1,2}(\w_0)$. Likewise, the sum of the terms in the second-to-top Maslov grading must equal
zero, so
\begin{equation} \label{eq:g22(w0)}
\sum_{i=1}^m T^{R(1,i)} \Psi^{ a\epsilon}_i(( g^{1,2} +  g^{2,1})(\w_0))+
(T^{A}+T^{A+B} +T^{A+B+C} + T^{A+B+C+D})  g^{2,2}(\w_0) = 0,
\end{equation}
which determines $ g^{2,2}(\w_0)$. Thus, $g$ is determined by the following.

\begin{proposition} \label{prop:g21(w0)}
The map $ g^{2,1}$ satisfies
\[
 g^{2,1}(\w_0)  =\frac{T^{B+\nu C}}{1+T^{B+C}} \z_0.
\]
\end{proposition}

\begin{proof}[Proof of Proposition \ref{prop:g21(w0)}]
In each of Figures \ref{fig:genus3a} and \ref{fig:genus3b}, the turquoise
regions represent the domain of a Whitney triangle $\psi_1 \in \pi_2(\w_0,
\Theta_2^{\beta\gamma}, \x_{a+1})$, and the (partially overlapping) gray
regions represent the domain of a triangle $\psi_2 \in \pi_2(\x_0,
\Theta_1^{\gamma\epsilon}, \z_0)$. Both of these domains avoid the basepoints,
and their weights are $\gen{\omega_{\r}, \psi_1} = 0$ and
$\gen{\omega_{\r}, \psi_2} = B + \nu C$. Moreover, one can verify using
Sarkar's formula for the Maslov index of polygons \cite{SarkarMaslov} that
$\mu(\psi_1) = \mu(\psi_2) = 0$. Since the map $F_{ a\beta\gamma} (\cdot
\otimes \Theta_2^{\beta\gamma})$ is homogeneous, it follows that
\[
F_{ a\beta\gamma} (\w_0 \otimes \Theta_2^{\beta\gamma}) = \sum_{i=1}^{m-1}
s_i\w_i
\]
for some coefficients $s_i\in\FF$. Since $\w_0$ is a cycle, the right side of
\eqref{eqn:fabg} must be as well, which implies that $s_i=0$ for $i\neq a+1$.
Furthermore, it is easy to verify that $\psi_1$ is the only positive class in
$\pi_2(\w_0, \Theta_2^{\beta\gamma},\x_{a+1})$ Thus, we may conclude that
\begin{equation}\label{eqn:fabg}
F_{ a\beta\gamma} (\w_0 \otimes \Theta_2^{\beta\gamma}) = s \cdot\x_{a+1}
\end{equation}
for some $s \in \F$. A similar argument shows that
\begin{equation}\label{eqn:fage}F_{ a\gamma\epsilon} (\x_0 \otimes
\Theta_1^{\gamma\epsilon}) = t\cdot T^{B+\nu C} \z_0
\end{equation}
for some $t \in \F$. Therefore,
\begin{equation} \label{eq:abge} F_{
a\gamma\epsilon} (\partial_{ a\gamma}^{-1}( F_{ a\beta\gamma} (\w_0 \otimes
\Theta_2^{\beta\gamma}) \otimes \Theta_1^{\gamma\epsilon}) = s\cdot
t\cdot\frac{T^{B+\nu C}}{1+T^{B+C}} \z_0.
\end{equation}
We shall see that $s$ and $t$ are both equal to $1$. Remarkably, we will not
need any direct analysis of moduli spaces to prove this fact.

Note that the Maslov grading shift of $F_{ a\beta\gamma}(\cdot,
\Theta_k^{\beta\gamma})$ is equal to that of $F_{ a\beta\delta}(\cdot,
\Theta_k^{\beta\delta})$, so $F_{ a\beta\delta}(\w_0,
\Theta_2)$ must be in the second-to-top Maslov grading in $\CFKtt(\bm a,
\bm\delta;\FF)$. 
This implies that $F_{ a\beta\delta}(\w_0, \Theta_1)$ is in the top Maslov
grading and, hence, is a multiple of $\y_0$. However, this multiple must be
zero since $\w_0$ is a cycle while $\partial_{ a\delta}(\y_0) \neq 0$. Thus,
\begin{equation} \label{eq:fabge}
F_{ a\delta\epsilon} (\partial_{ a\delta}^{-1}( F_{ a\beta\delta} (\w_0 \otimes
\Theta_1^{\beta\delta}) \otimes \Theta_2^{\delta\epsilon}) = 0.
\end{equation}

Next, we claim that the two terms in $ g^{2,1}(\w_0)$ which count holomorphic
rectangles both vanish; that is,
\begin{equation} \label{eq:noquad}
F_{ a\beta\gamma\epsilon}(\w_0 \otimes \Theta_2^{\beta\gamma} \otimes
\Theta_1^{\gamma\epsilon}) = F_{ a\beta\delta\epsilon}(\w_0 \otimes
\Theta_1^{\beta\delta} \otimes \Theta_2^{\delta\epsilon}) = 0.
\end{equation}
It follows from \eqref{eqn:fabg}, \eqref{eqn:fage} and the fact that $g^{2,1}$
is homogeneous that both terms in \eqref{eq:noquad} are multiples of $\z_0$. To
prove \eqref{eq:noquad}, we show that the domain of any Whitney rectangle
$\psi$ in $\pi_2(\w_0, \Theta_2^{\beta\gamma}, \Theta_1^{\gamma\epsilon},
\z_0)$ or $\pi_2(\w_0, \Theta_1^{\beta\delta}, \Theta_2^{\delta\epsilon},
\z_0)$ in Figure \ref{fig:genus3a} which avoids $\O\cup\X$ has some negative
multiplicities (the same argument works for the diagram in Figure
\ref{fig:genus3b}). For $i=1, \dots, a$, the local multiplicities of $D(\psi)$
near $p_i$ are as shown in Figure \ref{fig:noquad}(a), for some integers $p,q$.
To avoid negative multiplicities, we are forced to have $p=q=0$; it follows
that the multiplicity of the top region equals that of the bottom region. For
$i=1$, this top region has multiplicity $0$ since it contains $X_m$.
Inductively, the region directly to the right of $X_a$ in Figure
\ref{fig:genus3a} has multiplicity $0$. The multiplicities of $D(\psi)$ in the
regions near $O_{a+1}$ and $O_{c+1}$ are therefore as shown in Figures
\ref{fig:noquad}(b) and (c), for some integers $r,s,t$, and we are forced to
have $r=s=t=0$. Since neither $\Theta_1^{\beta\gamma}$ nor
$\Theta_1^{\delta\epsilon}$ is a corner of $D(\psi)$, the multiplicity on the
underside of the upper-right tube must be $-1$. As a result, $\psi$ has no
holomorphic representative.

\begin{figure}
\labellist
 \hair 2pt
 \pinlabel (a) at -7 365
 \pinlabel (b) at 181 365
 \pinlabel (c) at 181 175
 \tiny
 \pinlabel $p$ at 21 288
 \pinlabel $q$ at 37 288
 \pinlabel \rotatebox{55}{$p+q$} at 26 271
 \pinlabel $-p$ at 61 292
 \pinlabel $-q$ at 61 284
 \pinlabel $-p$ at 61 266
 \pinlabel $O_j$ at 62 317
 \pinlabel $X_j$ at 62 237
 \pinlabel \rotatebox{-65}{$-r-s$} at 212 323
 \pinlabel $-s$ at 230 290
 \pinlabel $0$ at 199 272
 \pinlabel $r$ at 250 286
 \pinlabel $s$ at 270 282
 \pinlabel $1$ at 290 282
 \pinlabel $1$ at 307 282
 \pinlabel $r+s$ at 274 265
 \pinlabel $O_{a+1}$ at 250 320
 \pinlabel $0$ at 300 310
 \pinlabel $0$ at 240 247
 \pinlabel $\Theta^{\beta\gamma}_1$ at 323 305
 \pinlabel $1$ at 330 285
 \pinlabel $0$ at 355 291
 \pinlabel $0$ at 340 318
 \pinlabel \rotatebox{69}{$-1$} at 372 321
 \pinlabel $0$ at 360 274
 \pinlabel $-t$ at 335 260
 \pinlabel $t$ at 304 242
 \pinlabel $-t$ at 330 210
 \pinlabel \rotatebox{-13}{$O_{c+1}$} at 335 234
 \pinlabel \rotatebox{-65}{$-t$} at 225 112
 \pinlabel $0$ at 199 82
 \pinlabel $-t$ at 285 114
 \pinlabel $t$ at 270 76
 \pinlabel $O_{a+1}$ at 250 120
 \pinlabel $0$ at 306 115
 \pinlabel $\Theta^{\delta\epsilon}_1$ at 307 90
 \pinlabel $1$ at 296 150
 \pinlabel $0$ at 340 128
 \pinlabel \rotatebox{69}{$-1$} at 365 117
 \pinlabel \rotatebox{-10}{$-r-s$} at 350 74
 \pinlabel $s$ at 304 54
 \pinlabel \rotatebox{3}{$-s$} at 308 78
 \pinlabel $r$ at 307 67
 \pinlabel \rotatebox{80}{$r+s$} at 290 56
 \pinlabel $1$ at 301 39
 \pinlabel $1$ at 298 29
 \pinlabel \rotatebox{-13}{$O_{c+1}$} at 335 44
 \pinlabel $0$ at 240 55



\endlabellist

\includegraphics[width=10.6cm]{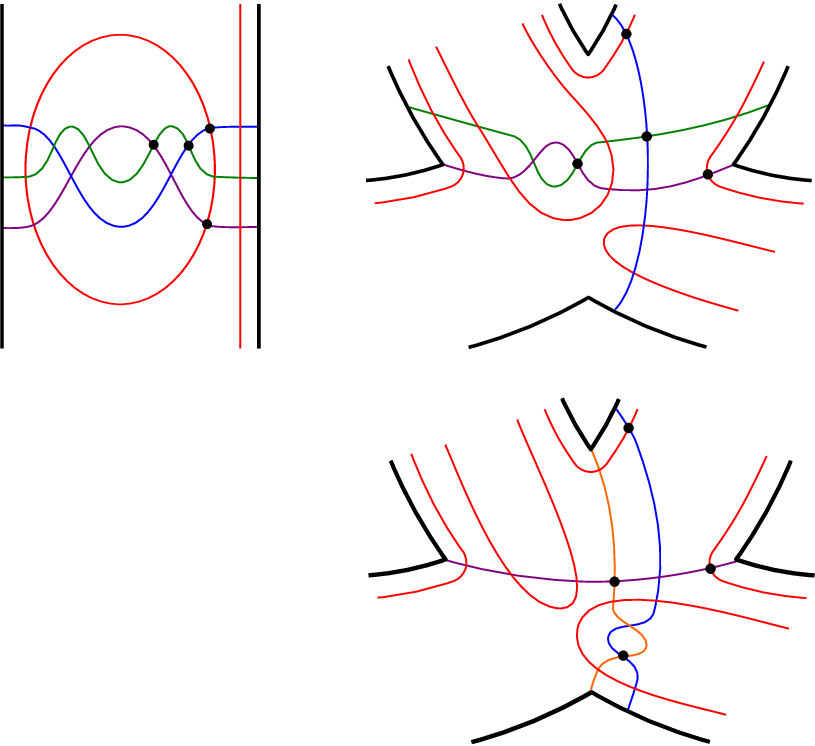}
\caption{Local multiplicities of the Whitney rectangle $\psi$.}
\label{fig:noquad}
\end{figure}
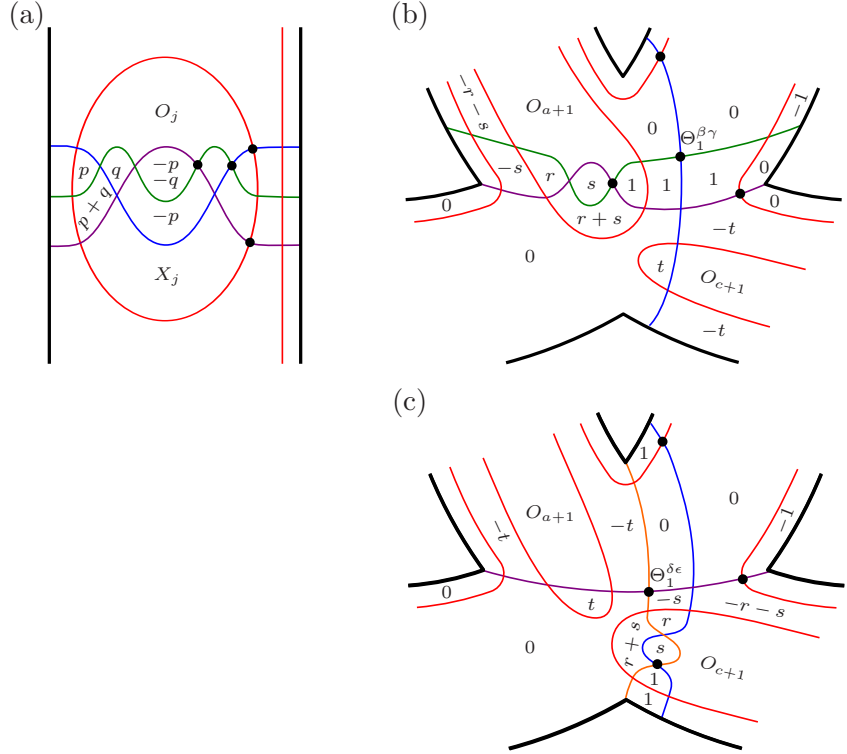

Therefore, the only potentially nonzero contribution to $ g^{2,1}(\w_0)$ (of the four terms in \eqref{eq:gkl(w)}) is that in \eqref{eq:abge}. If $s\cdot t = 0$, then $g(\w_0)$ is also zero, by \eqref{eq:g12(w0)} and \eqref{eq:g22(w0)}.
This implies that $g$ is identically zero, by Proposition \ref{prop:psi-gkl}. On the other hand, Theorem \ref{thm:cubeofres} tells us that the homology of the filtered complex in \eqref{eq:filtered} is $\HFKtt(\LL_J,
[\omega_{\r}]_J; \FF)$. Since the spectral sequence $\mathcal{S}_{I,I''}$ converges no later than the $E_3$ page, $\HFKtt(\LL_J, [\omega_{\r}]_J; \FF)$ is isomorphic to the homology of the mapping cone of $g$, by \eqref{eq:E2}. This means that if $g\equiv 0$, then $\HFKtt(\LL_J, [\omega_{\r}]_J; \FF)$ has rank $2^m$ over $\FF$. But this is plainly
impossible: if $[\omega_{\r}]_J = 0$, then by \eqref{eqn:err},
\[
\HFKtt(\LL_J, [\omega_{\r}]_J; \FF) \cong \HFKt(\LL_J)\otimes_{\F}\FF,
\]
which has rank $2^{m-1}$ over $\FF$; otherwise, if $[\omega_{\r}]_J \neq 0,$ then $\HFKtt(\LL_J, [\omega_{\r}]_J; \FF) = 0$, by Proposition \ref{prop:unlink}. Therefore, it must be the case that $s\cdot t = 1$, completing the proof of Proposition \ref{prop:g21(w0)}.
\end{proof}

\subsection{Sufficiency of the model computation}
In this subsection, we show that the model computation above suffices to describe the complex $(E_2(\SS^{\r}_\FF),d_2(\SS^{\r}_\FF))$ up to isomorphism.

The sequence of Heegaard moves from $\HH_{I,I''}$ to $\HH^3_{I,I''}$ described
at the beginning of Section \ref{sec:destab} induces chain homotopy
equivalences,
\[
\begin{aligned}
\CFKtt(\bm a, \bm\beta) & \to \CFKtt(\bm\alpha, \bm\eta(I)),
 & \CFKtt(\bm a, \bm\gamma) & \to \CFKtt(\bm\alpha, \bm\eta(I^1)), \\
\CFKtt(\bm a, \bm\delta) & \to \CFKtt(\bm\alpha, \bm\eta(I^2)),
 & \CFKtt(\bm a, \bm\epsilon) & \to \CFKtt(\bm\alpha, \bm\eta(I'')),
\end{aligned}
\]
which preserve both Maslov and Alexander gradings. These are compositions of
the maps corresponding to stabilizations with triangle-counting maps
corresponding to isotopies and handleslides. We shall denote these chain
homotopy equivalences by $\Phi_{I,I''}$. There are also maps $\ol \Phi_{I,I''}$
in the reverse direction which are homotopy inverses of the $\Phi_{I,I''}$.
These are compositions of triangle-counting maps with the maps corresponding to
destabilizations.

These Heegaard moves also induce homogeneous maps (denoted by $\Phi_{I,I''}$ as
well),
\[
\begin{aligned}
\CFKtt(\bm\beta, \bm\gamma) & \to \CFKtt(\bm\eta(I), \bm\eta(I^1)),
 & \CFKtt(\bm\gamma, \bm\epsilon) & \to \CFKtt(\bm\eta(I^1), \bm\eta(I'')), \\
\CFKtt(\bm\beta, \bm\delta) & \to \CFKtt(\bm\eta(I), \bm\eta(I^2)),
 & \CFKtt(\bm\delta, \bm\epsilon) & \to \CFKtt(\bm\eta(I^2), \bm\eta(I'')), \\
\end{aligned}
\]
which give rise to injections on homology taking the part of $\HFKtt(\bm\beta,
\bm\gamma)$ in the top Maslov grading to that of $\HFKtt(\bm\eta(I),
\bm\eta(I^1))$, etc. (See \cite{OSz4Manifold}.) Hence,
$\Phi_{I,I''}(\Theta_1^{\beta\gamma}) = T^e \Theta_1^{I,I^1}$ for some $e \in
\Z$. Every point of $\A$ is contained in the same region as a basepoint in the
triple-diagram associated to each pair of consecutive Heegaard diagrams in the
sequence from $(\Sigma,\bm\eta(I), \bm\eta(I^1))$ to
$(\Sigma_3,\bm\beta,\bm\gamma)$. It follows that $\Phi_{I,I''}$, does not pick
up any nontrivial powers of $T$; that is, $e=1$. By the same token,
\[
\Phi_{I,I''}(\Theta_1^{\beta\gamma}) = \Theta_1^{I,I^1}, \quad
\Phi_{I,I''}(\Theta_1^{\beta\delta}) = \Theta_1^{I,I^2}, \quad
\Phi_{I,I''}(\Theta_1^{\gamma\epsilon}) = \Theta_1^{I^1,I''}, \quad
\Phi_{I,I''}(\Theta_1^{\delta\epsilon}) = \Theta_1^{I^2,I''}.
\]
Furthermore, since $\Phi_{I,I''}$ is equivariant with respect to the maps $\Psi_i$, we have
\[
\Phi_{I,I''}(\Theta_2^{\beta\gamma}) = \Theta_2^{I,I^1}, \quad
\Phi_{I,I''}(\Theta_2^{\beta\delta}) = \Theta_2^{I,I^2}, \quad
\Phi_{I,I''}(\Theta_2^{\gamma\epsilon}) = \Theta_2^{I^1,I''}, \quad
\Phi_{I,I''}(\Theta_2^{\delta\epsilon}) = \Theta_2^{I^2,I''},
\]
by Lemma \ref{lem:psitheta1} and Lemma \ref{lem:psitheta2}.

\begin{figure}
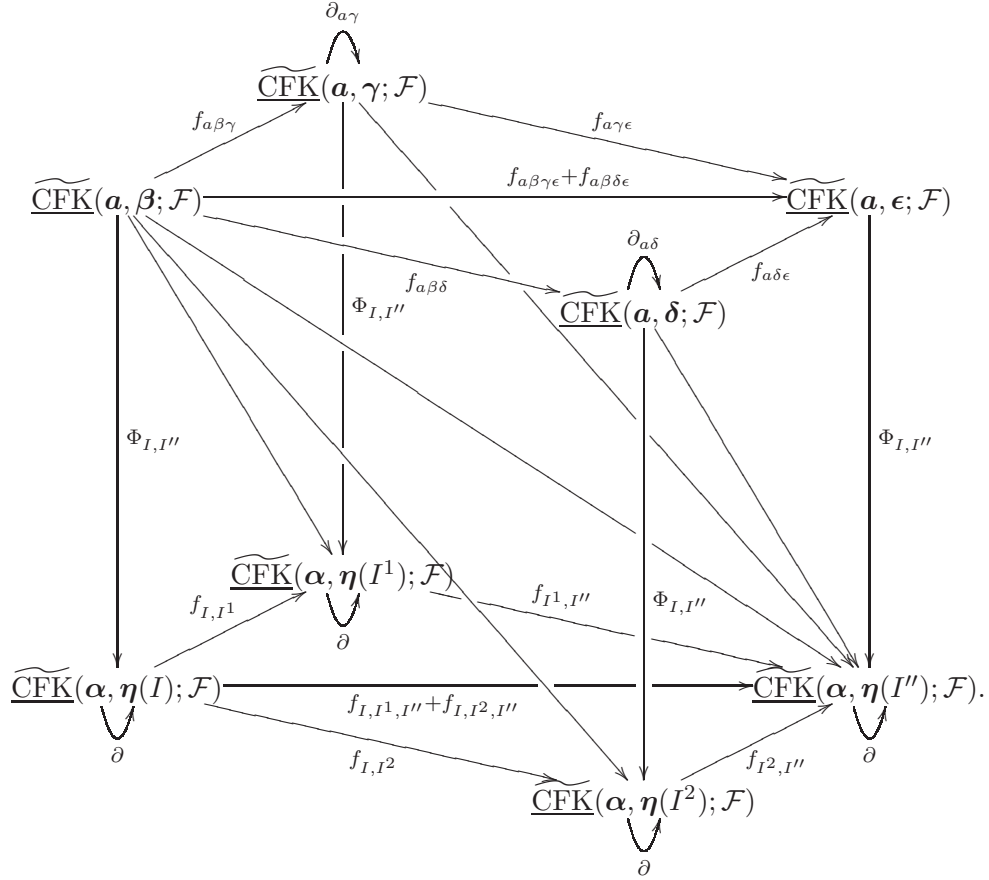

\[
\xy
 {(30,0)*{\CFKtt(\bm a, \bm\gamma;\FF)}="g"};
 {(0,-15)*{\CFKtt(\bm a, \bm\beta;\FF)}="b"};
 {(100,-15)*{\CFKtt(\bm a, \bm\epsilon;\FF)}="e"};
 {(70,-30)*{\CFKtt(\bm a, \bm\delta;\FF)}="d"};
 {(30,-65)*{\CFKtt(\bm\alpha, \bm\eta(I^1);\FF)}="I'1"};
 {(0,-80)*{\CFKtt(\bm\alpha, \bm\eta(I);\FF)}="I"};
 {(100,-80)*{\CFKtt(\bm\alpha, \bm\eta(I'');\FF).}="I''"};
 {(70,-95)*{\CFKtt(\bm\alpha, \bm\eta(I^2);\FF)}="I'2"};
 {\ar^{f_{ a\beta\gamma}} "b"; "g"};
 {\ar^{f_{ a\gamma\epsilon}} "g"; "e"};
 {\ar_(0.6){f_{ a\beta\delta}} "b"; "d"};
 {\ar_{f_{ a\delta\epsilon}} "d"; "e"};
 {\ar^(0.6){f_{ a\beta\gamma\epsilon}+f_{ a\beta\delta\epsilon}} "b"; "e"};
 {\ar^{f_{I,I^1}} "I"; "I'1"};
 {\ar^(0.4){f_{I^1,I''}} |(0.24){\hole} |(0.57){\hole} "I'1"; "I''"};
 {\ar_{f_{I,I^2}} "I"; "I'2"};
 {\ar_{f_{I^2,I''}} "I'2"; "I''"};
 {\ar_(0.42){f_{I,I^1,I''}+f_{I,I^2,I''}} |(0.57){\hole} |(0.7){\hole} "I"; "I''"};
 {\ar^{\Phi_{I,I''}} "b"; "I"};
 {\ar^(0.46){\Phi_{I,I''}}|(0.23){\hole} |(0.33){\hole} |(0.53){\hole} |(0.76){\hole} "g"; "I'1"};
 {\ar^(0.6){\Phi_{I,I''}} "d"; "I'2"};
 {\ar^{\Phi_{I,I''}} "e"; "I''"};
 {\ar "b"; "I'1"};
 {\ar "b"; "I'2"};
 {\ar|(0.7){\hole} "b"; "I''"};
 {\ar|(0.19){\hole} |(0.33){\hole} |(0.57){\hole} "g"; "I''"};
 {\ar "d"; "I''"};
 {\ar@/^12pt/^{\partial_{ a\gamma}} "g"+<-6pt,8pt>; "g"+<6pt,8pt>};
 {\ar@/^12pt/^{\partial_{ a\delta}} "d"+<-6pt,8pt>; "d"+<6pt,8pt>};
 {\ar@/_12pt/_{\partial} "I"+<-6pt,-8pt>; "I"+<6pt,-8pt>};
 {\ar@/_12pt/_{\partial} "I'1"+<-6pt,-8pt>; "I'1"+<6pt,-8pt>};
 {\ar@/_12pt/_{\partial} "I'2"+<-6pt,-8pt>; "I'2"+<6pt,-8pt>};
 {\ar@/_12pt/_{\partial} "I''"+<-6pt,-8pt>; "I''"+<6pt,-8pt>};
 \endxy
\]
\caption{The complex induced by the Heegaard moves from $\HH^3_{I,I''}$ to
$\HH_{I,I''}$.} \label{fig:filteredmap}
\end{figure}

The $\mathcal{A}_{\infty}$ relations \eqref{eqn:Ainfty}, applied to the large
multi-diagram which includes all of the multi-diagrams in the sequence from
$\HH_{I,I''}$ to $\HH^3_{I,I''}$, show that these maps fit into a complex as
shown in Figure \ref{fig:filteredmap}. We may view this complex as a filtered
map between two filtered complexes, which induces a map of spectral sequences.
On the $E_2$ page, we have a commutative square,
\begin{equation} \label{eq:E2map}
\xy
 {(0,0)*{\HFKtt(\bm a, \bm\beta;\FF)}="b"};
 {(40,0)*{\HFKtt(\bm a, \bm\epsilon;\FF)}="e"};
 {(0,-20)*{\HFKtt(\bm\alpha, \bm\eta(I);\FF)}="I"};
 {(40,-20)*{\HFKtt(\bm\alpha, \bm\eta(I'');\FF).}="I''"};
 {\ar^{g} "b"; "e"};
 {\ar^{d_{I,I''}}"I"; "I''"};
 {\ar^{(\Phi_{I,I''})_*}_\cong "b"; "I"};
 {\ar^{(\Phi_{I,I''})_*}_\cong "e"; "I''"};
 \endxy
\end{equation}

The maps $(\Phi_{I,I''})_*$ parametrize $\HFKtt(\bm\alpha, \bm\eta(I);\FF)$ and
$\HFKtt(\bm\alpha, \bm\eta(I'');\FF)$ by groups that we understand concretely.
\eqref{eq:E2map} then says that, with respect to these parametrizations, the
map $d_{I,I''}$ is described by $g$. To show that this determines the
\emph{global} structure of $(E_2(\SS_\FF^\r),d_2(\SS_\FF^\r))$,
we must verify that any two of these parametrizations agree where they overlap.
Specifically, consider another double successor pair $J,J''\in\RR(\DD)$ for
which either $I=J$, $I=J''$, $I''=J$ or $I''=J''$.\footnote{This $J$ is not
related to the $J$ used earlier in this section.} Without loss of generality,
let us assume that $I=J$; the other three cases are treated identically. Let
\[
\HH_{J,J''}^3 = (\Sigma_3, \bm a', \bm\beta', \bm\gamma', \bm\delta',
\bm\epsilon', \O, \X)
\]
be the genus $3$ Heegaard diagram obtained from $\HH_{J,J''}$, as described in
Section \ref{sec:destab}, and let $\w_0'$ denote the point in $\T_{a'} \cap
\T_{\beta'}$ of maximal Maslov grading. Since $\HFKtt(\bm\alpha, \bm\eta(I))$
has rank $1$ in this grading, we know that $(\Phi_{J,J''})_*([\w_0']) = \lambda
(\Phi_{I,I''})_*([\w_0'])$ for some nonzero $\lambda \in \FF$. Since any
element of $\HFKtt(\bm a, \bm \beta)$ (resp.~$\HFKtt(\bm a', \bm\beta')$) can
be obtained from $[\w_0]$ (resp.~$[\w_0']$) via the action of the maps
$\psi_i^{a\beta}$ (resp. $\psi_i^{a'\beta'}$), and $(\Phi_{I,I''})_*$ and
$(\Phi_{J,J''})_*$ are equivariant with respect to these actions, the constant
$\lambda$ completely determines the relationship between the two
parametrizations. In fact, the following proposition implies that $\lambda=1$.

\begin{proposition}\label{prop:naturality}
The elements $\Phi_{I,I''}(\w_0)$ and $\Phi_{J,J''}(\w'_0)$ represent the same
homology class in $\HFKtt(\bm\alpha, \bm\eta(I))$.
\end{proposition}

\begin{proof}[Proof of Proposition \ref{prop:naturality}]
It suffices to show that the composition
\begin{equation}\label{eqn:compos}
\CFKtt(\bm a, \bm\beta;\FF) \xrightarrow{\Phi_{I,I''}} \CFKtt(\bm\alpha,
\bm\eta(I);\FF) \xrightarrow{\ol\Phi_{J,J''}} \CFKtt(\bm a', \bm\beta';\FF)
\end{equation}
sends $\w_0$ to $\w_0'$. Since the maps induced by (de)stabilizations commute
with those induced by isotopies and handleslides, $\Phi_{I,I''}$ and
$\Phi_{J,J''}$ can be factored into the compositions
\[
\begin{aligned}
\CFKtt(\bm a, \bm\beta;\FF)\xrightarrow{\Phi} \CFKtt(\bm a^1, \bm\beta^1;\FF)
\xrightarrow{\Phi_{I,I''}'} \CFKtt(\bm\alpha,
\bm\eta(I);\FF),\\
\CFKtt(\bm\alpha,\bm\eta(I);\FF)\xrightarrow{\ol\Phi_{J,J''}'} \CFKtt(\bm a^k,
\bm\beta^k;\FF)\xrightarrow{\ol\Phi}\CFKtt(\bm a', \bm\beta';\FF),
\end{aligned}
\]
where $(\Sigma,\bm a^1,\bm\beta^1,\O,\X)$ and $(\Sigma,\bm a^k, \bm\beta^k,\O,\X)$
are obtained from $H^3_{I,I''}$ and $H^3_{J,J''}$ by stabilizing $n-2$ times,
and $\Phi$ and $\ol\Phi$ are the maps induced by stabilization and
destabilization, respectively.

By definition, the map $\Phi$ sends $\w_0$ to the unique generator $\w^1_0$ in
$\T_{a^1}\cap\T_{\beta^1}$ of maximal Maslov grading. Likewise, $\ol\Phi$ sends
the unique generator $\w^k_0$ in $\T_{a^k}\cap\T_{\beta^k}$ of maximal Maslov
grading to $\w'_0$. To prove Proposition \ref{prop:naturality}, it then
suffices to show that the composition
\[
\CFKtt(\bm a^1, \bm\beta^1;\FF)
\xrightarrow{\Phi_{I,I''}'}
\CFKtt(\bm\alpha,\bm\eta(I);\FF)\xrightarrow{\ol\Phi_{J,J''}'} \CFKtt(\bm a^k,
\bm\beta^k;\FF)
\]
sends $\w^1_0$ to $\w^k_0$. The map $\ol\Phi_{J,J''}'\circ\Phi_{I,I''}$ is a
composition $\Phi_{k-1}\circ\dots\circ \Phi_1$, where \[\Phi_i: \CFKtt(\bm a^i,
\bm\beta^i;\FF)\rightarrow \CFKtt(\bm a^{i+1}, \bm\beta^{i+1};\FF)\] is the
triangle-counting map induced by the handleslide or isotopy taking $(\Sigma,\bm
a^i, \bm\beta^i,\O,\X)$ to $(\Sigma,\bm a^{i+1}, \bm\beta^{i+1},\O,\X)$, where
either $\bm a^i=\bm a^{i+1}$ or $\bm\beta^i=\bm\beta^{i+1}$. Note that for some
intermediate $j$, we have $\bm a^j = \bm \alpha$ and $\bm \beta^j =
\bm\eta(I)$.

Recall from the previous section that a Kauffman generator is one which does
not contain any intersection point between a ladybug curve and a non-ladybug
curve. Let $\w^i_0$ denote the unique Kauffman generator in
$\T_{a^i}\cap\T_{\beta^i}$ of maximal Maslov grading. It is not hard to see
that, for each $i=1,\dots,k-1$, there is a Whitney triangle $\psi_i$ in either
$\pi_2(\Theta^{a^{i+1} a^{i}},\w^{i}_0,\w^{i+1}_0)$ or
$\pi_2(\w^i_0,\Theta^{\beta^i\beta^{i+1}},\w^{i+1}_0)$ (depending on whether
$\bm a^i=\bm a^{i+1}$ or $\bm\beta^i=\bm\beta^{i+1}$, respectively) which
avoids $\O\cup\X\cup\A$. If $(\Sigma,\bm a^{i+1}, \bm\beta^{i+1},\O,\X)$ is
obtained from $(\Sigma,\bm a^i, \bm\beta^i,\O,\X)$ by an isotopy or handleslide
of a ladybug curve, then $\w^{i+1}_0$ is ``very close" to $\w^i_0$ and the
domain of $\psi_i$ is just a disjoint union of small triangles. Otherwise, if
$(\Sigma,\bm a^{i+1}, \bm\beta^{i+1},\O,\X)$ is obtained from $(\Sigma,\bm a^i,
\bm\beta^i,\O,\X)$ by a handleslide of a non-ladybug curve, then there are two
possibilities.

Without loss of generality, assume $\bm a^i=\bm a^{i+1}$ and that $\bm
\beta^{i+1}$ is obtained from $\bm\beta^i$ by handlesliding $\beta^i_1$ over
$\beta^i_2$, as shown in Figure \ref{fig:hmoves}. If there is no point of
$\w^i_0$ on $a^i_1\cap\beta^i_1$, then $\w^{i+1}_0$ is ``very close" to
$\w^i_0$ and the domain of $\psi_i$ is a disjoint union of small triangles; see
Figure \ref{fig:hmoves}(a). Otherwise, $\w^{i+1}_0$ is ``very close" to
$\w^i_0$ away from the portion of the diagram shown in Figure
\ref{fig:hmoves}(b). In these distant regions, the domain of $\psi_i$ is a
disjoint union of small triangles; near $a^i_1$ and $\beta_2^i$, the domain of
$\psi_i$ consists of the hexagon shown in the figure.

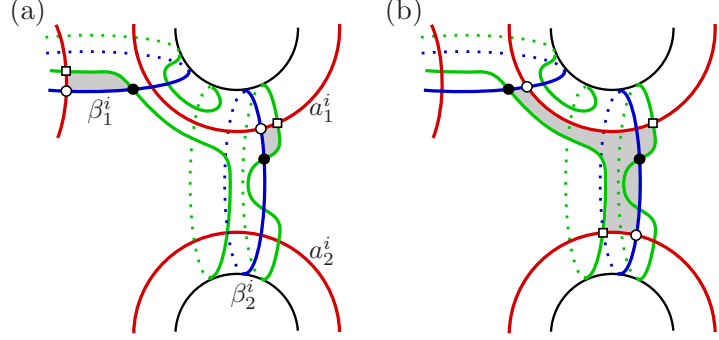
\begin{figure}
\labellist
 \pinlabel (a) at 0 120
 \pinlabel (b) at 129 120
 \hair 2pt
 \small \pinlabel $\beta^i_1$ at 25 86 \pinlabel $\beta^i_2$ at 74 23
\pinlabel $a^i_1$ at 101 87 \pinlabel $a^i_2$ at 101 38
\endlabellist

\includegraphics[width=9cm]{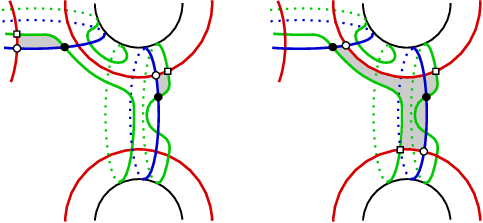}
\caption{The white circles and squares represent $\w_0^i$ and $\w_0^{i+1}$, respectively, and the $\bm\beta^{i+1}$ curves are in green.}
\label{fig:hmoves}
\end{figure}

The concatenation $\psi = \psi_1*\dots*\psi_k$ is therefore a Whitney
$(k+2)$-gon connecting $\w^1_0$ to $\w^k_0$, with evaluation
$\gen{\omega_{\r},\psi}=0$. Suppose that $\psi'$ is another concatenation
of triangles connecting these two generators and missing $\O\cup\X$. Then
$D(\psi')-D(\psi)$ is a multi-periodic domain $P$ on the large multi-diagram
that encodes all intermediate diagrams between $(\Sigma,\bm
a^1,\bm\beta^1,\O,\X)$ and $(\Sigma,\bm a^k,\bm\beta^k,\O,\X)$. One can show,
exactly as in the proof of Lemma \ref{lem:doublyperiodic}, that any such
periodic domain is the sum of doubly-periodic domains in $\Pi^0_{a^i,a^{i+1}}$
or $\Pi^0_{\beta^i,\beta^{i+1}}$, for $i=1,\dots,k-1$, with a periodic domain
$P'$ in $\Pi^0_{a^1, \beta^1}$. The former domains must miss $\A$ since the
handleslides and isotopies all avoid $\A$, and since $(\Sigma,\bm
a^1,\bm\beta^1,\O,\X)$ is a diagram for the unknot in $S^3$, we have $P' = 0$.
Thus, $P$ misses $\A$, so $\gen{\omega_{\r},\psi'} = \gen{\omega_\r,
\psi} = 0$. This implies that the coefficient of $\w^k_0$ in
$\ol\Phi_{J,J''}'\circ\Phi_{I,I''}(\w^1_0)$ is $1$.
\end{proof}

\begin{proof}[Proof of Theorems \ref{thm:main} and \ref{thm:main2}]
For each $I \in \RR(\DD)$ such that either \begin{inparaenum}\item there is a
double successor $I''$ of $I$ with $I''\in \RR(\DD)$ or \item $I$ is a double
successor of some $J\in \RR(\DD)$\end{inparaenum}, Proposition
\ref{prop:naturality} gives us a canonical class $w^I  \in \HFKtt(\bm\alpha,
\bm\eta(I);\FF)$. For all other $I \in \RR(\DD)$, we may take $w^I$ to be any
generator of $\HFKtt(\bm\alpha, \bm\eta(I);\FF)$ in the top Maslov grading.

Recall that $\YY_I$ is the vector space over $\FF$ generated by $y_1,
\dots, y_m$, modulo the relation
\[
\sum_{i=1}^m T^{r_{\sigma_I}(1) + \dots + r_{\sigma_I}(i)} y_{\sigma_I(i)} = 0.
\]
By Proposition \ref{prop:twistedunknot}, there are isomorphisms
\[
\rho_I\co \Lambda^*(\YY_I) \to \HFKtt(\bm\alpha, \bm\eta(I);\FF)
\]
such that $\rho_I(1) = w_I$ and $\rho_I(y_i x) = \psi_i^{\alpha\eta(I)}
(\rho_I(x))$ for all $x \in \Lambda^*(\YY_I)$. By expressing Proposition
\ref{prop:psi-gkl}, Proposition \ref{prop:g21(w0)}, \eqref{eq:g12(w0)} and
\eqref{eq:g22(w0)} in terms of these identifications, we see that if $I''$ is a
double successor of $I$, then $d_{I,I''}$ is as described in Section
\ref{sec:complex}. Thus, the maps $\rho_I$ induce an isomorphism of chain
complexes from $(C(\DD),
\partial^\r)$ to $(E_2(\SS_\FF^\r),d_2(\SS_\FF^\r))$, and the
grading $\Delta$ agrees with the grading on $(C(\DD),\partial^{\r})$
defined in Section \ref{sec:complex}. This identification, combined with
Theorem \ref{thm:collapse}, completes the proof.
\end{proof}

\begin{remark} \label{rmk:untwistedd1}
One can also use the computations in this section to determine the $d_1$ differential of the \emph{untwisted} spectral sequence $\SS_\F$ (which does not depend on $\r$). Unfortunately, the rank of its $E_2$ page, after dividing by $2^{m-|L|}$ to adjust for the number of marked points, is not an invariant of $L$. For instance, the complex associated to a $0$-crossing diagram of the unknot with $m$ marked points consists of a single copy of $\Lambda_*(\YY_I)$ (where $I$ is the empty tuple), with rank $2^{m-1}$. On the other hand, for a $3$-crossing diagram for the unknot obtained by changing one crossing of a diagram for the trefoil, with one marked point on each of the six edges, a \emph{Mathematica} computation shows that the $E_2$ page has rank $48$ rather than $32 = 2^5$.
\end{remark}

\bibliographystyle{amsplain}
\bibliography{bibliography}

\end{document}